\documentclass[11pt, oneside]{article}
\usepackage[margin=20truemm]{geometry}
\usepackage{graphicx}	
\usepackage{amssymb}
\usepackage{mleftright}
\usepackage{amsmath,amsthm,mathtools,color,mathabx}
\usepackage{usebib}
\usepackage{esint}

\newtheorem{theorem}{Theorem}[section]

\newtheorem{lemma}[theorem]{Lemma}
\newtheorem{proposition}[theorem]{Proposition}
\newtheorem{remark}[theorem]{Remark}
\newtheorem{definition}[theorem]{Definition}

\newcommand{\divx}{\mathop{\mathrm{div}}}
\newcommand{\esssup}{\mathop{\mathrm{ess~sup}}}
\newcommand{\essinf}{\mathop{\mathrm{ess~inf}}}

\allowdisplaybreaks[4]

\numberwithin{equation}{section}

\def\Yint#1{\mathchoice
    {\YYint\displaystyle\textstyle{#1}}%
    {\YYint\textstyle\scriptstyle{#1}}%
    {\YYint\scriptstyle\scriptscriptstyle{#1}}%
    {\YYint\scriptscriptstyle\scriptscriptstyle{#1}}%
      \!\iint}
\def\YYint#1#2#3{{\setbox0=\hbox{$#1{#2#3}{\iint}$}
    \vcenter{\hbox{$#2#3$}}\kern-.51\wd0}}
\def\longdash{{-}\mkern-3.5mu{-}} 
\def\tiltlongdash{\rotatebox[origin=c]{15}{$\longdash$}}
\def\fiint{\Yint\tiltlongdash}

\title{Continuity of  the  spatial gradient of  weak solutions  to  very singular parabolic equations  involving the one-Laplacian}
\author{Shuntaro Tsubouchi  \footnote{Graduate School of Mathematical Sciences, The University of Tokyo, Japan. \textit{Email}: \texttt{tsubos@g.ecc.u-tokyo.ac.jp} Funding; The author was supported by the Japan Society for the Promotion of Science through JSPS KAKENHI Grant Number 22KJ0861 during the preparation of the paper.}}
\date{}
\begin{document}
\maketitle
\begin{abstract}
We consider weak solutions to very singular parabolic equations involving a one-Laplace-type operator, which is singular and degenerate, and a $p$-Laplace-type operator with $\frac{2n}{n+2}<p<\infty$, where $n\ge 2$ denotes the space dimension. 
This type of equation is used to describe the motion of a Bingham flow. 
It has been a long-standing open problem of whether the spatial gradients of weak solutions are continuous in space and time.
This paper aims to give an affirmative answer for a wide class of such equations.
This equation becomes no longer uniformly parabolic near the facet, the place where the spatial gradient vanishes.
To achieve our goal, we show local a priori H\"{o}lder continuity of gradients suitably truncated near facets. 
For this purpose, we consider a parabolic approximate problem and appeal to standard methods, including De Giorgi's truncation and comparisons with Dirichlet heat flows.
Our method is a parabolic adjustment of our method developed to prove the corresponding statements for stationary problems.
\end{abstract}
\bigbreak
\bigbreak

{\footnotesize \tableofcontents}

\section{Introduction}\label{Sect:Introduction}
This paper is concerned with regularity for weak solutions to
\begin{equation}\label{Eq (Section 1): (1,p)-Laplace parabolic}
  \partial_{t}u-\Delta_{1}u-\Delta_{p}u=f\quad \textrm{in}\quad \Omega_{T}\coloneqq \Omega\times (0,\,T).
\end{equation}
Here $T\in(0,\,\infty)$ and $p\in(1,\,\infty)$ are fixed constants, and $\Omega\subset {\mathbb R}^{n}$ is an $n$-dimensional bounded Lipschitz domain with $n\ge 2$.
The real-valued functions $f=f(x,\,t)$ and $u=u(x,\,t)$ are respectively given and unknown for $(x,\,t)=(x_{1},\,\dots\,,\,x_{n},\,t)\in\Omega_{T}$.
The divergence operators $\Delta_{1}$ and $\Delta_{p}$ are respectively the one-Laplacian and the $p$-Laplacian, defined as
\(\Delta_{s}u\coloneqq \divx\left(\lvert \nabla u\rvert^{s-2}\nabla u\right)\) for \(s\in\lbrack 1,\,\infty)\).
Here $\nabla u\coloneqq (\partial_{x_{j}}u,\,\dots\,,\,\partial_{x_{n}}u)$ and $\partial_{t}u$ respectively denote a spatial and time derivative of a solution $u=u(x,\,t)$, and  $\divx X\coloneqq \partial_{x_{1}}X_{1}+\cdots+\partial_{x_{n}}X_{n}$  is divergence of a vector field $X=(X_{1},\,\dots\,,\,X_{n})$.

The parabolic $(1,\,p)$-Laplace equation (\ref{Eq (Section 1): (1,p)-Laplace parabolic}) appears in describing the motion of Bingham flows with $p=2$ (see Section \ref{Subsect: Model}).
Mathematical analysis for this parabolic equation at least goes back to the classical textbook \cite{MR0521262} in 1976.
In particular, the unique existence of weak solutions for initial value problems is well-known, which is based on variational inequalities \cite[Chapter VI]{MR0521262}.
However, the continuity of a spatial gradient has been a long-standing open problem even when $f=0$.
For stationary problems, this regularity problem was answered affirmatively by the author \cite{T-scalar}, \cite{T-system}.
By adjusting the methods therein to time-evolutional problems (see Sections \ref{Subsect: Strategy}--\ref{Subsect: Literature} for the details), we would like to prove that a spatial gradient is continuous even for the parabolic equation (\ref{Eq (Section 1): (1,p)-Laplace parabolic}).

Throughout this paper, we let $f$ be in a Lebesgue space $L^{q}(\Omega_{T})$, and mainly assume
\begin{equation}\label{Eq (Section 1): p-q condition}
 p_{\mathrm c}  \coloneqq \frac{2n}{n+2}  <p<\infty,\quad \textrm{and}\quad n+2<q\le\infty.
\end{equation}
This assumption is typical when one considers the continuity of a spatial gradient of a $p$-Poisson flow (see e.g., \cite{BDLS}, \cite[Chapters VIII--IX]{MR1230384}, \cite[Chapter 2]{MR1881297}).
 In this paper, $f\in L^{p^{\prime}}(0,\,T;\,W^{-1,\,p^{\prime}}(\Omega))$ is also assumed with $p^{\prime}\coloneqq p/(p-1)$ denoting the H\"{o}lder conjugate exponent of $p$, since we treat (\ref{Eq (Section 1): (1,p)-Laplace parabolic}) in the $L^{p^{\prime}}(0,\,T;\,W^{-1,\,p^{\prime}}(\Omega))$ sense. We would like to prove $\nabla u\in C^{0}(\Omega;\,{\mathbb R}^{n})$, under the assumption (\ref{Eq (Section 1): p-q condition}). 

More generally, we consider an anisotropic very singular equation
\begin{equation}\label{Eq (Section 1): General Parabolic Eq}
  \partial_{t}u+{\mathcal L}u=f\quad \textrm{in}\quad \Omega_{T}
\end{equation}
with 
\({\mathcal L}u\coloneqq -\divx \left(\nabla E_{1}(\nabla u)+\nabla E_{p}(\nabla u) \right)\).
The detailed conditions of $E_{1}$ and $E_{p}$, which are at least $C^{2}$ outside the origin, are explained later in Section \ref{Subsect:General Setting}.
Our settings allow us to choose 
$E_{s}(z)\coloneqq \lvert z\rvert^{s}/s\,(z\in{\mathbb R}^{n})$ for $s\in\{\,1,\,p\,\}$.
In particular, ${\mathcal L}$ generalizes $-\Delta_{1}-\Delta_{p}$.
 \subsection{Our strategy}\label{Subsect: Strategy}
 To explain the main difficulty, we formally differentiate (\ref{Eq (Section 1): General Parabolic Eq}) by $x_{j}$.
 Then, we have
 \begin{equation}\label{Eq (Section 1): Parabolic equation differentiated}
   \partial_{t}\partial_{x_{j}}u-\divx\left(\nabla^{2}E(\nabla u)\nabla\partial_{x_{j}}u \right)=\partial_{x_{j}}f\quad \textrm{in}\quad \Omega_{T},
 \end{equation}
 where $E\coloneqq E_{1}+E_{p}$.
 Roughly speaking, we face a problem that the coefficient matrix $\nabla^{2}E(\nabla u)$ in (\ref{Eq (Section 1): Parabolic equation differentiated}) becomes no longer uniform parabolic near a facet, the degenerate region of a gradient.
 Here uniform parabolicity can be measured by the ratio, defined as the maximum eigenvalue of $\nabla^{2}E(z)$ divided by the minimum one of $\nabla^{2}E(z)$, which is well-defined for $z\in{\mathbb R}^{n}\setminus\{0\}$.
 For elliptic problems, this type of ratio is often called the ellipticity ratio (see e.g., \cite{MR2291779}).
 Similarly, we would like to call the ratio \textit{parabolicity ratio}.
 For the special case $E(z)=\lvert z\rvert+\lvert z\rvert^{p}/p$, the parabolicity ratio of $\nabla^{2}E(\nabla u)$ is behaves like 
 \(1+\lvert\nabla u\rvert^{1-p}\),
 which blows up as $\nabla u\to 0$.
 Thus, it appears hard to obtain quantitative continuity estimates for $\partial_{x_{j}}u$.
 Such difficulty is substantially caused by the fact that diffusion of one-Laplacian becomes degenerate in the direction of a gradient, while this diffusion may become singular in other directions.
 In other words, $\Delta_{1}$ has anisotropic diffusivity, which is substantially different from $\Delta_{p}$. 
 This property makes it difficult to handle $\Delta_{1}$ in the existing regularity theory.
 
 However, when one constrains a region $\{\lvert \nabla u\rvert\ge \delta\}$ for fixed $\delta\in(0,\,1)$, then over such a place, the parabolicity ratio of $\nabla^{2}E(\nabla u)$ becomes bounded by some constant that depends on $\delta$.
 In this sense, we may say that (\ref{Eq (Section 1): Parabolic equation differentiated}) is locally uniformly parabolic outside a facet.
 With this in mind, we would like to introduce a truncation parameter $\delta\in(0,\,1)$, and consider a truncated gradient
 \({\mathcal G}_{\delta}(\nabla u)\coloneqq (\lvert \nabla u\rvert-\delta)_{+}\frac{\nabla u}{\lvert \nabla u\rvert}.\)
 We aim to prove that ${\mathcal G}_{\delta}(\nabla u)$ is locally H\"{o}lder continuous in $\Omega_{T}$ for each fixed $\delta\in(0,\,1)$. 
 Note that this truncated gradient is supported in $\{\lvert \nabla u\rvert\ge \delta\}$, and therefore such local H\"{o}lder regularity will be expected.
 Here it should be noted that our H\"{o}lder estimates of ${\mathcal G}_{\delta}(\nabla u)$ substantially depends on $\delta\in(0,\,1)$, and most of our continuity estimates will blow up as $\delta$ tends to $0$.
 Also, it is worth mentioning that ${\mathcal G}_{\delta}(\nabla u)$ uniformly converges to ${\mathcal G}_{0}(\nabla u)=\nabla u$ in $\Omega_{T}$.
 Thus, although some of our regularity estimates might break as $\delta$ tends to $0$, we conclude that $\nabla u$ is continuous.

 To give a rigorous proof, we have to appeal to approximation arguments.
 Here it should be recalled that the very singular operator $-\Delta_{1}-\Delta_{p}$, or more general ${\mathcal L}$, will become no longer uniformly parabolic near a facet of a solution.
 This prevents us from applying standard regularity methods, including difference quotient methods, directly to (\ref{Eq (Section 1): (1,p)-Laplace parabolic}) or (\ref{Eq (Section 1): General Parabolic Eq}).
 In particular, although we have formally deduced (\ref{Eq (Section 1): Parabolic equation differentiated}), this equation seems difficult to deal with in the sense of $L^{2}(0,\,T;\,W^{-1,\,2}(\Omega))$.
 For this reason, we have to consider a parabolic approximate problem, whose solution has improved regularity properties.
 Precisely speaking, for an approximation parameter $\varepsilon\in(0,\,1)$, we relax the principal part ${\mathcal L}u$ by \({\mathcal L}_{\varepsilon}u_{\varepsilon}\coloneqq -\divx \left(\nabla E_{\varepsilon}(\nabla u_{\varepsilon})\right)\).
 Here $E_{\varepsilon}$ is defined to be the convolution of $E=E_{1}+E_{p}$ with the Friedrichs standard mollifier $j_{\varepsilon}\in C_{\mathrm c}^{\infty}({\mathbb R}^{n})$.
 We also approximate  $f\in L^{q}(\Omega_{T})\cap L^{p^{\prime}}(0,\,T;\,W^{-1,\,p^{\prime}}(\Omega))$  by $f_{\varepsilon}\in C^{\infty}(\Omega_{T})$, which converges to $f$ in a weak sense.
 The resulting approximate equation is given by
 \(\partial_{t}u_{\varepsilon}+{\mathcal L}_{\varepsilon}u_{\varepsilon}=f_{\varepsilon}\) in \(\Omega_{T}.\)
 For this relaxed problem, we may assume $\nabla u_{\varepsilon}\in L_{\mathrm{loc}}^{\infty}$ and $\nabla^{2}u_{\varepsilon}\in L_{\mathrm{loc}}^{2}$ (see Section \ref{Subsect: Basic Weak Form} for the details), and therefore $u_{\varepsilon}$ satisfies
 \begin{equation}\label{Eq (Section 1): Approximation problem}
  \partial_{t}\partial_{x_{j}}u_{\varepsilon}-\divx\left(\nabla^{2}E_{\varepsilon}(\nabla u_{\varepsilon})\nabla\partial_{x_{j}} u_{\varepsilon} \right)=\partial_{x_{j}}f_{\varepsilon}\quad \text{in}\quad \Omega_{T}
 \end{equation}
 in the weak sense.
 The approximations of $E$ and $f$ are already found in elliptic problems \cite[\S 2]{T-scalar}, where a convergence result for an approximate solution is shown.
 Following the strategy in \cite[\S 2.6]{T-scalar}, we verify that an approximate solution $u_{\varepsilon}$, which satisfies a suitable Dirichlet boundary condition, converges to a weak solution to (\ref{Eq (Section 1): General Parabolic Eq}) strongly in $L^{p}(W^{1,\,p})$.

 With the aid of approximation arguments, our regularity problem is reduced to local a priori H\"{o}lder continuity of
 \({\mathcal G}_{2\delta,\,\varepsilon}(\nabla u_{\varepsilon})\coloneqq \left(\sqrt{\varepsilon^{2}+\lvert\nabla u_{\varepsilon}\rvert^{2}}-2\delta \right)_{+}\frac{\nabla u_{\varepsilon}}{\lvert \nabla u_{\varepsilon}\rvert}\) for each fixed \(\delta\in(0,\,1)\), where the approximation parameter $\varepsilon$ is smaller than $\delta/8$.
 We should remark that the principal part ${\mathcal L}u$ itself is relaxed by ${\mathcal L}_{\varepsilon}u_{\varepsilon}$.
 In particular, for (\ref{Eq (Section 1): Approximation problem}), the parabolicity ratio of $\nabla^{2}E_{\varepsilon}(\nabla u_{\varepsilon})$ should be measured by $V_{\varepsilon}\coloneqq \sqrt{\varepsilon^{2}+\lvert \nabla u_{\varepsilon}\rvert^{2}}$, rather than by $\lvert \nabla u_{\varepsilon}\rvert$. 
 For this reason, we have to consider another truncation mapping ${\mathcal G}_{2\delta,\,\varepsilon}$, instead of ${\mathcal G}_{2\delta}$.
 In this paper, we mainly prove H\"{o}lder estimates of ${\mathcal G}_{2\delta,\,\varepsilon}(\nabla u_{\varepsilon})$, independent of $\varepsilon\in(0,\,\delta/8)$.
 Then, the Arzel\`{a}--Ascoli theorem enables us to conclude that ${\mathcal G}_{2\delta}(\nabla u)$ is also $\alpha$-H\"{o}lder continuous for each fixed $\delta\in(0,\,1)$.
 We have to mention again that the exponent $\alpha\in(0,\,1)$ substantially depends on $\delta$. Indeed, our a priori estimates imply $\alpha\to 0$ as $\delta\to 0$ (see Section \ref{Subsect: Holder a priori}).
 
 Our proof of the H\"{o}lder continuity of ${\mathcal G}_{2\delta,\,\varepsilon}(\nabla u_{\varepsilon})$ is based on the modifications of \cite{T-scalar}, which deals the stationary problem of (\ref{Eq (Section 1): (1,p)-Laplace parabolic}).
 Here we mainly use standard methods in the parabolic regularity theory, including De Giorgi's truncation and comparisons with Dirichlet heat flows.
 Our parabolic computations concerning the H\"{o}lder continuity of ${\mathcal G}_{2\delta,\,\varepsilon}(\nabla u_{\varepsilon})$ are inspired by \cite[\S 3--7]{BDLS}, where the H\"{o}lder gradient continuity of a vector-valued $p$-harmonic flow is discussed (see also \cite{MR1230384}, \cite{MR783531}, \cite{MR814022} as fundamental works).
 Compared with \cite{BDLS}, our analysis is rather classical, since we do not use intrinsic scaling arguments at all (see Section \ref{Subsect: Literature}). 
 Instead, we always avoid analysis near the degenerate regions of $V_{\varepsilon}=\sqrt{\varepsilon^{2}+\lvert\nabla u_{\varepsilon}\rvert^{2}}$, so that we may regard (\ref{Eq (Section 1): Approximation problem}) as a uniformly parabolic equation in the classical sense, depending on $\delta$.
 The local H\"{o}lder estimate of ${\mathcal G}_{2\delta,\,\varepsilon}(\nabla u_{\varepsilon})$ is shown, as long as local $L^{\infty}$-bounds of $V_{\varepsilon}$, uniformly for $\varepsilon$, are guaranteed. We give the proof of this a priori continuity estimate under the settings $p\in(1,\,\infty)$ and $V_{\varepsilon}\in L_{\mathrm{loc}}^{\infty}$. 

 In showing the uniform $L^{\infty}$-bounds of $V_{\varepsilon}$, however, the basic approach will differ, depending on whether $p$ is greater than $p_{\mathrm c}=\frac{2n}{n+2}$ or not. 
 In this paper, where the supercritical case $p_{\mathrm c}<p<\infty$ is considered, we prove local $L^{\infty}$--$L^{p}$ estimates of $V_{\varepsilon}$ by Moser's iteration. For the critical or subcritical case $1<p\le p_{\mathrm c}$, even the local $L^{\infty}$-bound of a weak solution appears non-trivial, and some higher integrability assumption is required to prove the gradient continuity (see \cite{MR1135917}). 
 This remaining case is discussed in another paper \cite{T-subcritical}, under the assumption that a weak solution is in $L_{\mathrm{loc}}^{s}(\Omega_{T})$ with $s>n(2-p)/p\ge 2$. 
 It should be remarked that the compact embedding $W_{0}^{1,\,p}(\Omega)\hookrightarrow\hookrightarrow L^{2}(\Omega)$ holds if and only if $p\in(p_{\mathrm c},\,\infty\rbrack$, and that we cannot apply the Aubin--Lions lemma for $p\in(1,\,p_{\mathrm c}\rbrack$.
 In particular, when $p$ is very close to $1$, we will not be able to use any compact embedding concerning a parabolic function space.
 For this reason, the paper \cite{T-subcritical} fails to deal with any external force term (see also Remark \ref{Rmk; Subcase}).

 Although the continuity of the spatial gradient is proved for $(1,\,p)$-Laplace elliptic and even parabolic problems, it still remains an open problem whether or not the spatial derivative is H\"{o}lder continuous.
 More precisely, the function $w(x)\coloneqq (\lvert x\rvert-1)_{+}^{p^{\prime}}/p^{\prime}\,(x\in{\mathbb R}^{n})$, or the Fenchel dual function of $\lvert \xi\rvert+\lvert \xi\rvert^{p}/p\,(\xi\in{\mathbb R}^{n})$, satisfies the stationary equation $\Delta_{1}w+\Delta_{p}w=n$ in ${\mathbb R}^{n}$.
 The special solution $w$ indicates the best possible regularity of the solution. In particular, weak solutions to $(1,\,p)$-Laplace problems are at most expected to be in $C^{1,\,\alpha}$ with $\alpha\coloneqq \max\{\,1,\,1/(p-1)\,\}$. Even for stationary equations, however, this $C^{1,\,\alpha}$-regularity problem appears difficult, since we will directly face the non-uniformly elliptic structure of the $(1,\,p)$-Laplace operator, especially over facets of solutions.
 When it comes to the $C^{1}$-regularity, we can give affirmative results by appealing to the truncation approach.

 \subsection{Literature overview and comparisons with the paper}\label{Subsect: Literature}
 Section \ref{Subsect: Literature} briefly introduces previous results on the elliptic or parabolic regularity for the $p$-Laplace or the $(1,\,p)$-Laplace equation.
 We also compare this paper with \cite{T-scalar} and \cite{T-system}, the author's recent works on the gradient continuity of the stationary $(1,\,p)$-Laplace problem.

 For the elliptic $p$-Laplace problems, $C^{1,\,\alpha}$-regularity of a weak solution in the class $W^{1,\,p}(\Omega)$ is well-established in the full range $p\in(1,\,\infty)$ (see e.g., \cite{MR0997847}, \cite{MR709038}, \cite{MR672713}, \cite{MR721568}, \cite{MR2635642}, \cite{MR727034}, \cite{MR0244628}).
 For the parabolic $p$-Laplace problems, the H\"{o}lder gradient continuity was established by DiBenedetto--Friedman for $p\in (p_{\mathrm c},\,\infty)$ in \cite{MR783531} and \cite{MR814022} (see also \cite{MR0718944}, \cite{MR743967}, \cite{MR0886719} for weaker results on the gradient continuity).
 In the remaining case $p\in(1,\,p_{\mathrm c}\rbrack$, the same regularity was shown by Choe in \cite{MR1135917}, where a solution $u\in L^{p}(0,\,T;\,W^{1,\,p}(\Omega))\cap C(\lbrack 0,\,T\rbrack;\,L^{2}(\Omega))$ is also assumed to satisfy $u\in L_{\mathrm{loc}}^{s}(\Omega_{T})$ with $s>n(2-p)/p\ge 2$.
 Without this $L^{s}$ higher integrability assumption, no improved regularity result is generally expected (see \cite[III.~7]{MR1066761}).
 
 In the proof of the H\"{o}lder gradient continuity for the parabolic $p$-Laplace equations, found in \cite{BDLS} and \cite[Chapter VIII]{MR1230384}, careful rescaling arguments for local gradient bounds play an important role.
 This method is often called an intrinsic scaling argument, which is a powerful tool for deducing various regularity estimates for the parabolic $p$-Laplace equations (\cite{MR1230384}, \cite{MR2865434}).
 In our analysis for (\ref{Eq (Section 1): (1,p)-Laplace parabolic}), however, we do not appeal to intrinsic scaling arguments, since our equation seems no longer uniformly parabolic on the facet of a solution, even if it is carefully rescaled.
 Instead, we make careful use of the truncation parameter $\delta$, and try to avoid some delicate cases where a gradient may degenerate.
 In particular, we use no intrinsic parabolic cylinders, since our analysis works on standard parabolic cylinders, thanks to the careful truncation of a gradient.
 
 For the stationary $(1,\,p)$-Laplace problem $-\Delta_{1}u-\Delta_{p}u=f$ with $f=f(x)$ in $L^{q}$ for some $q\in(n,\,\infty\rbrack$
 the continuous differentiability of a solution $u=u(x)$ was first shown in \cite{MR4408168}, where a solution is assumed to be both scalar-valued and convex.
 Although the proof therein is rather easily carried out thanks to some basic tools from convex analysis, it is heavily based on a maximum principle and the convexity of a solution.
 Recently, the author removed these technical assumptions in \cite{T-scalar}, \cite{T-system}, where it is proved that a truncated gradient ${\mathcal G}_{\delta}(\nabla u)$, defined in Section \ref{Subsect: Strategy}, is locally H\"{o}lder continuous.
 The author's recent research (\cite{T-scalar}, \cite{T-system}) is highly inspired by a paper \cite{BDGPdN}, which shows the gradient continuity for a very degenerate elliptic problem of the form $-\divx\left(\lvert{\mathcal G}_{1}(\nabla v)\rvert^{p^{\prime}-2}{\mathcal G}_{1}(\nabla v) \right)=f$.
 Here the real-valued functions $v=v(x)$ and $f=f(x)$ are respectively unknown and given.
 For this degenerate problem, the three papers \cite{BDGPdN}, \cite{MR3133426}, \cite{MR2728558} prove that a truncated gradient ${\mathcal G}_{1}(\nabla v)$, similarly defined in Section \ref{Subsect: Strategy}, is continuous, provided $f\in L^{q}$ with $q\in(n,\,\infty\rbrack$.
 Among them, \cite{BDGPdN} is based on classical methods, including De Giorgi's truncation and Campanato's perturbation argument, and this strategy works even for the vector-valued case. 
 Moreover, this regularity result was extended to parabolic systems in the recent paper \cite{BDGPdN-p}. 
 Following the spirit of \cite{BDGPdN}, the author proved the gradient continuity for $(1,\,p)$-Laplace-type problems both in the scalar-valued \cite{T-scalar} and vector-valued \cite{T-system} cases.
 The continuity of ${\mathcal G}_{0}(\nabla u)=\nabla u$ and ${\mathcal G}_{1}(\nabla v)$ is proved in a qualitative way, by showing quantitative continuity estimates of the truncated gradients ${\mathcal G}_{\delta}(\nabla u)$ and ${\mathcal G}_{1+\delta}(\nabla v)$, whose estimates depend on the truncation parameter $\delta\in(0,\,1)$.
 These continuity results are known only for elliptic cases, although some other regularity results are shown for degenerate parabolic problems in recent papers \cite{ambrosio2022regularity}, \cite{gentile2023higher}.
 In this paper, by adjusting the arguments in \cite{T-scalar} to the parabolic problem, we aim to establish a gradient continuity result for the singular parabolic equation (\ref{Eq (Section 1): (1,p)-Laplace parabolic}).

 We conclude Section \ref{Subsect: Literature} by noting a difference between this paper and the author's previous works \cite{T-scalar}, \cite{T-system}, especially on the   convergence  of approximate solutions.
 To prove the strong $L^{p}$-convergence of $\nabla u_{\varepsilon}$, we have to make use of a weak compactness argument and a compact embedding for parabolic function spaces.
 There, the condition $p>p_{\mathrm c}$ is carefully used to apply the Aubin--Lions lemma, which plays a crucial role in treating the external force terms.
 Hence, our parabolic approximation arguments need some restrictions on the exponents $p$ and $q$, while \cite{T-scalar} and \cite{T-system} require no restrictive assumptions.
 This kind of restriction strongly appears when $p\le p_{\mathrm c}$, since the compact embedding $W_{0}^{1,\,p}(\Omega)\hookrightarrow\hookrightarrow L^{2}(\Omega)$ no longer holds true.
 In particular, the Aubin--Lions lemma cannot be applied, which prevents us from dealing with any external force term (see also Remark \ref{Rmk; Subcase}).
 Although our approximation argument partly works even for $p\in(1,\,p_{\mathrm c}\rbrack$ (see \cite[Proposition 3.3]{T-subcritical} for the details), no external force term can be treated.
 \subsection{Mathematical models}\label{Subsect: Model}
 The second-order parabolic equation (\ref{Eq (Section 1): (1,p)-Laplace parabolic}) appears when one models the motion of a Bingham fluid \cite[Chapter VI]{MR0521262}, a non-Newtonian fluid that has both plasticity and viscosity properties.
 Its motion is governed by 
 \(\partial_{t}U+(U\cdot \nabla)U-b_{1}\Delta_{1}U-b_{2}\Delta_{2}U+\nabla\pi=0\) and \(\divx U=0\)
 if there is no external force.
 Here $U=U(x_{1},\,x_{2},\,x_{3},\,t)$ denotes the three-dimensional velocity of a Bingham fluid, and a scalar function $\pi=\pi(x_{1},\,x_{2},\,x_{3},\,t)$ denotes the pressure.
 The positive constants $b_{1}$ and $b_{2}$ respectively reflect a Bingham fluid's plasticity and viscosity effects. 
 When $b_{1}=0$, this model becomes the incompressible Navier--Stokes equation.
 We consider a uni-directional laminar flow in a cylinder $\Omega\times {\mathbb R}\subset {\mathbb R}^{3}$ with velocity $U=(0,\,0,\,u(x_{1},\,x_{2},\,t))$.
 Since the equalities $\divx U=0$ and $(U\cdot \nabla)U=0$ are automatically fulfilled, the equation is reduced to
 \(\partial_{t}u-b_{1}\Delta_{1}u-b_{2}\Delta_{2}u=-\partial_{x_{3}}\pi\) and \(\partial_{x_{1}}\pi=\partial_{x_{2}}\pi=0\)
 by carrying out similar arguments in \cite[Chapter VI, \S 1.3]{MR0521262}.
 Moreover, the right-hand side $\partial_{x_{3}}\pi\equiv f$ depends at most on $t$, since the left-hand side is independent of $x_{3}$.
 From our main Theorem \ref{Thm: Main theorem parabolic}, given in Section \ref{Subsect:General Setting}, we conclude that a spatial derivative $\nabla u$ is continuous, provided $f=f(t)$ admits $L^{q}$-integrability with $q>4$.

 It will be worth mentioning that another important model can be found in the evaporation dynamics of crystal surfaces, modeled by Spohn in a paper \cite[\S 2]{spohn1993surface}.
 There the evolution of the height function of the crystal, denoted by $h$, is modeled as
 \(\partial_{t}h=-\mu\cdot (-\Delta_{1}h-\Delta_{3}h)\)
 with $\mu=\mu(\nabla h)>0$ standing for the mobility.
 When $\mu$ is constant, we can conclude the continuity of $\nabla h$ by our main Theorem \ref{Thm: Main theorem parabolic}.
 However, Spohn's discussion concerning mobility tells us that $\mu$ has to be of the form $\mu=\kappa \lvert\nabla h\rvert$ for some constant $\kappa>0$.
 In other words, in his model, the speed of each level set is assumed to move by $\Delta_{1}h+\Delta_{3}h$, and another equation 
 appears.

 \subsection{Notations, main results and outline of the paper}\label{Subsect:General Setting}
 Before stating our main theorem, we fix some notations.
 
 The set of all non-negative integers and the set of all natural numbers are respectively denoted by ${\mathbb Z}_{\ge 0}\coloneqq \{\,0,\,1,\,2,\,\dots\,\,\}$, and ${\mathbb N}\coloneqq {\mathbb Z}_{\ge 0}\setminus\{ 0\}$.
 For a given real number $a\in{\mathbb R}$, we write $a_{+}\coloneqq \max\{\,a,\,0\,\}$.
 For given vectors $x=(x_{1},\,\dots\,,\,x_{n})$, $y=(y_{1},\,\dots\,,\,y_{n})\in{\mathbb R}^{n}$, the canonical inner product and the Euclidean norm are respectively written by
 $\langle x\mid y \rangle\coloneqq x_{1}y_{1}+\cdots +x_{n}y_{n}\in{\mathbb R}$ and $\lvert x\rvert\coloneqq \sqrt{\langle x\mid x\rangle}\in\lbrack 0,\,\infty)$.
 For an $n\times n$ real matrix $A=(A_{j,\,k})_{j,\,k}$, we define the standard norms as
 \(\lVert A\rVert\coloneqq \sup_{\lvert x\rvert=1}\lvert Ax\rvert\), and \(\lvert A\rvert\coloneqq \sqrt{\sum_{j,\,k=1}^{n}A_{j,\,k}^{2}},\)
 called the operator norm and the Hilbert--Schmidt norm respectively.
 For $n\times n$ symmetric real matrices $A$ and $B$, we write $A\leqslant B$ when $B-A$ is positive semi-definitive.
 The symbols $\mathrm{id}$ and $O$ stand for the identity matrix and the zero matrix respectively.

 We introduce a parabolic metric $d_{\mathrm{p}}$ in ${\mathbb R}^{n+1}$ by
 \(d_{\mathrm{p}}((x,\,t),\,(y,\,s))\coloneqq \max\left\{\,\lvert x-y\rvert,\,\sqrt{\lvert t-s\rvert}\,\right\}\) for \((x,\,t),\,(y,\,s)\in{\mathbb R}^{n}\times {\mathbb R}.\)
 With respect to this metric, $\mathop{\mathrm{dist}}_{\mathrm{p}}({\mathcal Q},\,\partial_{\mathrm{p}}\Omega_{T})$ denotes the distance from the given set ${\mathcal Q}\Subset \Omega_{T}$ to the parabolic boundary $\partial_{\mathrm{p}}\Omega_{T}\coloneqq (\Omega\times \{0\})\cup (\partial\Omega\times \lbrack0,\,T\rbrack)$.
 For given $x_{0}\in{\mathbb R}^{n}$, $t_{0}\in{\mathbb R}$ and $r\in(0,\,\infty)$, we set an open ball $B_{r}(x_{0})\coloneqq \{x\in{\mathbb R}^{n}\mid \lvert x-x_{0}\rvert<r\}$, a half interval $I_{r}(t_{0})\coloneqq (t_{0}-r^{2},\,t_{0}\rbrack$, and a parabolic cylinder $Q_{r}(x_{0},\,t_{0})\coloneqq B_{r}(x_{0})\times I_{r}(t_{0})$.
 The closure and the interior of the half interval $I_{r}(t_{0})$ are respectively denoted by $\overline{I_{r}}\coloneqq \lbrack t_{0}-r^{2},\,t_{0}\rbrack$ and $\ring{I_{r}}\coloneqq (t_{0}-r^{2},\,t_{0})$.
 The center points $x_{0}$ and $t_{0}$ are often omitted when they are clear.

 For $p\in(1,\,\infty)$, the H\"{o}lder conjugate exponent is denoted by $p^{\prime}\coloneqq p/(p-1)\in(1,\,\infty)$.
 For $k,\,m\in{\mathbb N}$, we write $W^{k,\,p}(\Omega;\,{\mathbb R}^{m})$, and $L^{p}(\Omega;\,{\mathbb R}^{m})$ respectively by the Sobolev space of $k$-th order, and the Lebesgue space, equipped with the standard norms.
 When $m=1$, we often omit the codomain ${\mathbb R}^{m}$, and write $L^{p}(\Omega)$ and $W^{k,\,p}(\Omega)$.
 For $p\in(1,\,\infty)$, the closed subspace $W_{0}^{1,\,p}(\Omega)\subset W^{1,\,p}(\Omega)$ is defined to be the closure of infinitely differentiable functions compactly supported in $\Omega$, equipped with another equivalent norm
 \(\lVert \nabla u\rVert_{L^{p}(\Omega)}\coloneqq \left(\int_{\Omega}\lvert \nabla u\rvert^{p}\,{\mathrm d}x \right)^{1/p}\) for \(u\in W_{0}^{1,\,p}(\Omega).\)
 With respect to this norm, the dual of $W_{0}^{1,\,p}(\Omega)$ is denoted by $W^{-1,\,p^{\prime}}(\Omega)$. 
 For $F\in W^{-1,\,s^{\prime}}(\Omega)$ and $v\in W_{0}^{1,\,s}(\Omega)$ with $s\in(p_{\mathrm c},\,\infty)$, the dual pairing is often denoted by $\langle F,\,v\rangle_{W^{-1,\,s^{\prime}}(\Omega),\,W_{0}^{1,\,s}(\Omega)}\coloneqq F(v)$, or more simply by $\langle F,\,v\rangle$ when $s$ and $\Omega$ are clear.
 We write $C^{0}(\lbrack 0,\,T\rbrack;\,L^{2}(\Omega))$ for the set of all $L^{2}(\Omega)$-valued functions that are strongly continuous in $\lbrack 0,\,T\rbrack$.
 For $k\in {\mathbb Z}_{\ge 0}$, $l,\,m\in{\mathbb N}$, and $U\subset{\mathbb R}^{l}$, the symbol $C^{k}(U;\,{\mathbb R}^{m})$ denotes the set of all ${\mathbb R}^{m}$-valued functions of the $C^{k}$-class in $U$, and we abbreviate $C^{k}(U)\coloneqq C^{k}(U;\,{\mathbb R})$.

 To simplify notation, we often make use of the abbreviations $X=(x,\,t)$ and ${\mathrm d}X={\mathrm d}x{\mathrm d}t$.
 For a $k$-dimensional Lebesgue measurable set $U\subset {\mathbb R}^{k}$ with $k\in{\mathbb N}$, the symbol $\lvert U\rvert$ stands for the $k$-dimensional Lebesgue measure of $U$.
 For ${\mathbb R}^{k}$-valued functions $g=g(x)$, and $h=h(x,\,t)$ that are respectively integrable in $U\subset {\mathbb R}^{n}$, and $U\times I\subset{\mathbb R}^{n+1}$ with $0<\lvert U\rvert,\,\lvert I\rvert<\infty$, the average integrals are denoted by
 \(\fint_{U}g(x)\,{\mathrm d}x\coloneqq \lvert U\rvert^{-1}\int_{U}g(x)\,{\mathrm d}x\in{\mathbb R}^{k}\), \(\fiint_{U\times I}h(x,\,t)\,{\mathrm d}X\coloneqq \lvert U\rvert^{-1}\lvert I\rvert^{-1}\iint_{U}h(x,\,t)\,{\mathrm d}X\in{\mathbb R}^{k}.\)
 The former average integral is often denoted by $f_{U}$.
 The average integral of $h$ over a parabolic cylinder $Q_{r}$ is often denoted by $(h)_{Q_{r}}$ or $(h)_{r}$.

 The function spaces
 \[X^{p}(0,\,T;\,\Omega)\coloneqq \left\{u\in L^{p}(0,\,T;\,W^{1,\,p}(\Omega))\mathrel{}\middle|\mathrel{}\partial_{t}u\in L^{p^{\prime}}(0,\,T;\,W^{-1,\,p^{\prime}}(\Omega))\right\},\]
 and 
 \[X_{0}^{p}(0,\,T;\,\Omega)\coloneqq \left\{u\in L^{p}(0,\,T;\,W_{0}^{1,\,p}(\Omega))\mathrel{}\middle|\mathrel{}\partial_{t}u\in L^{p^{\prime}}(0,\,T;\,W^{-1,\,p^{\prime}}(\Omega))\right\}\]
 are considered for $p\in(p_{\mathrm c},\,\infty)$.
 By the embeddings $ W_{0}^{1,\,p}(\Omega)\hookrightarrow\hookrightarrow L^{2}(\Omega)\hookrightarrow W^{-1,\,p^{\prime}}(\Omega)$, we can use some inclusions or embeddings concerning $X_{0}^{p}(0,\,T;\,\Omega)$.
 More precisely, by the Lions--Magenes lemma \cite[Chapter III, Proposition 1.2]{MR1422252}, the inclusion $X_{0}^{p}(0,\,T;\,\Omega)\subset C^{0}(\lbrack 0,\,T\rbrack;\, L^{2}(\Omega))$ holds. 
 Also by the Aubin--Lions lemma \cite[Chapter III, Proposition 1.3]{MR1422252}, the compact embedding $X_{0}^{p}(0,\,T;\,\Omega)\hookrightarrow\hookrightarrow L^{p}(0,\,T;\,L^{2}(\Omega))$ can be applied.

 In this paper, we require $E_{1}\in C^{0}({\mathbb R}^{n})\cap C^{2}({\mathbb R}^{n}\setminus\{0\})$, and $E_{p}\in C^{1}({\mathbb R}^{n})\cap C^{2}({\mathbb R}^{n}\setminus \{0\})$ for the smoothness of the density functions.
 We let $E_{1}$ be positively one-homogeneous.
 In other words, $E_{1}$ is assumed to satisfy
 \begin{equation}\label{Eq (Section 1): Positively one-homogeneous}
   E_{1}(kz)=kE_{1}(z)\quad \textrm{for all }k\in(0,\,\infty),\,z\in{\mathbb R}^{n}.
 \end{equation}
 We also require the existence of a concave, non-decreasing modulus of continuity $\omega_{1}\colon\lbrack 0,\,\infty)\to \lbrack0,\,\infty)$ such that $\omega_{1}(0)=0$ and
 \begin{equation}\label{Eq (Section 1): (Assumption) modulus of continuity of Hess E1}
  \left\lVert \nabla^{2}E_{1}(z_{1})-\nabla^{2}E_{1}(z_{2}) \right\rVert\le \omega_{1}(\lvert z_{1}-z_{2}\rvert)
 \end{equation}
 for all $z_{1}$, $z_{2}\in{\mathbb R}^{n}$ with $1/32\le \lvert z_{j}\rvert\le 3$ for $j\in\{\,1,\,2\,\}$.
 The other density $E_{p}$ admits the constants $0<\lambda_{0}<\Lambda_{0}<\infty$ such that
 \begin{equation}\label{Eq (Section 1): Gradient growth bounds for E-p}
  \lvert \nabla E_{p}(z)\rvert\le \Lambda_{0}\lvert z\rvert^{p-1}\quad \textrm{for all }z\in{\mathbb R}^{n},
 \end{equation}
 \begin{equation}\label{Eq (Section 1): Hessian estimates for E-p}
  \lambda_{0}\lvert z\rvert^{p-2}\mathrm{id}_{n}\leqslant \nabla^{2}E_{p}(z)\leqslant \Lambda_{0}\lvert z\rvert^{p-2}\mathrm{id}_{n}\quad \textrm{for all }z\in{\mathbb R}^{n}\setminus\{0\}.
 \end{equation}
  For the continuity of $\nabla^{2} E_{p}$,  we assume that there exists a concave, non-decreasing, and continuous function $\omega_{p}\colon \lbrack 0,\,\infty)\to\lbrack 0,\,\infty)$, such that $\omega_{p}(0)=0$ and 
 \begin{equation}\label{Eq (Section 1): modulus of continuity of Hess Ep}
  \left\lVert \nabla^{2}E_{p}(z_{1})-\nabla^{2}E_{p}(z_{2}) \right\rVert\le C_{\delta,\,M}\omega_{p}(\lvert z_{1}-z_{2}\rvert/\mu)
 \end{equation}
 for all $z_{1}$, $z_{2}\in{\mathbb R}^{n}$ with $\mu/32\le \lvert z_{j}\rvert\le 3\mu$ for $j\in\{\,1,\,2\,\}$, and $\mu\in(\delta,\,M-\delta)$.
 Here $\delta$ and $M$ are fixed constants with $0<2\delta<M<\infty$, and the constant $C_{\delta,\,M}\in(0,\,\infty)$ depends on $\delta$ and $M$.

 The definition of a weak solution to (\ref{Eq (Section 1): General Parabolic Eq}) is given as follows.
 \begin{definition}\upshape
   Let the exponents $p$ and $q$ satisfy (\ref{Eq (Section 1): p-q condition}), and  $f\in L^{q}(\Omega_{T})\cap L^{p^{\prime}}(0,\,T;\,W^{-1,\,p^{\prime}}(\Omega))$. 
   A function $u\in X^{p}(0,\,T;\,\Omega)\cap C^{0}(\lbrack 0,\,T\rbrack;\,L^{2}(\Omega))$ is called a \textit{weak} solution to (\ref{Eq (Section 1): General Parabolic Eq}) if there exists a vector field $Z\in L^{\infty}(\Omega_{T})$ such that
   \begin{equation}\label{Eq (Section 1): Weak form Definition}
    \int_{0}^{T}\langle\partial_{t}u,\,\varphi\rangle\,{\mathrm d}t+\iint_{\Omega_{T}}\left\langle Z+\nabla E_{p}(\nabla u)\mathrel{}\middle|\mathrel{}\nabla\varphi \right\rangle\,{\mathrm d}X=\iint_{\Omega_{T}}f\varphi\,{\mathrm d}X
   \end{equation}
   for all $\varphi\in X_{0}^{p}(0,\,T;\,\Omega)$, and
   \begin{equation}\label{Eq (Section 1): Subgradient parabolic}
     Z(x,\,t)\in \partial E_{1}(\nabla u(x,\,t))\quad \textrm{for a.e.~}(x,\,t)\in\Omega_{T}.
   \end{equation}
    Here $\partial E_{1}$ denotes the subdifferential of $E_{1}$, defined as \[\partial E_{1}(\zeta)\coloneqq \left\{w\in{\mathbb R}^{n} \mathrel{} \middle|\mathrel{} E_{1}(z)\ge E_{1}(\zeta)+\langle w\mid z-\zeta\rangle \text{ for all }z\in{\mathbb R}^{n} \right\}\] for each $\zeta\in{\mathbb R}^{n}$.
   
 \end{definition}
 Our main result is the following Theorem \ref{Thm: Main theorem parabolic}.
 \begin{theorem}\label{Thm: Main theorem parabolic}
   Let $p$, $q$, $E_{1}$ and $E_{p}$ satisfy (\ref{Eq (Section 1): p-q condition}) and (\ref{Eq (Section 1): Positively one-homogeneous})--(\ref{Eq (Section 1): modulus of continuity of Hess Ep}).
   Assume that $u$ is a weak solution to (\ref{Eq (Section 1): General Parabolic Eq}) with  $f\in L^{q}(\Omega_{T})\cap L^{p^{\prime}}(0,\,T;\,W^{-1,\,p^{\prime}}(\Omega))$.  
   Then, $\nabla u$ is continuous in $\Omega_{T}$. 
 \end{theorem}

 This paper is organized as follows.
 Section \ref{Sect:Approximation} aims to give an approximation scheme for the very singular parabolic equation (\ref{Eq (Section 1): General Parabolic Eq}).
 The basic strategy is to relax an energy density by convoluting this and the Friedrichs standard mollifier.
 In Section \ref{Sect:Approximation}, we verify that a weak solution to (\ref{Eq (Section 1): General Parabolic Eq}) can be constructed as a limit of the approximate parabolic equation (Proposition \ref{Prop: Convergence result}).
 By basic estimates for approximate solutions, including local gradient bounds (Theorem \ref{Proposition: Lipschitz}), De Giorgi-type oscillation estimates (Proposition \ref{Proposition: Parabolic De Giorgi}), and Campanato-type growth estimates (Proposition \ref{Proposition: Parabolic Campanato}), we prove a priori H\"{o}lder bounds for truncated gradients of approximate solutions (Theorem \ref{Thm: Key Hoelder Estimate}).
 From Proposition \ref{Prop: Convergence result} and Theorem \ref{Thm: Key Hoelder Estimate}, we complete the proof of Theorem \ref{Thm: Main theorem parabolic}.
 Sections \ref{Sect:Estimates from Weak Form}--\ref{Sect:Non-Degenerate Region} provides the proofs of Theorem \ref{Proposition: Lipschitz}, and Propositions \ref{Proposition: Parabolic De Giorgi}--\ref{Proposition: Parabolic Campanato}.
 In Section \ref{Sect:Estimates from Weak Form}, we deduce some basic energy estimates from weak formulations.
 Theorem \ref{Proposition: Lipschitz} is shown in Section \ref{Sect:Lipschitz Bounds}. 
 Sections \ref{Sect:Degenerate Region}--\ref{Sect:Non-Degenerate Region} provide various oscillation lemmata. 
 There, depending on the size of a superlevel set, we make different analyses for degenerate or non-degenerate cases.
 Section \ref{Sect:Degenerate Region}, which deals with the degenerate case, is focused on proving Proposition \ref{Proposition: Parabolic De Giorgi}.
 Section \ref{Sect:Non-Degenerate Region}, where the non-degenerate case is discussed, we appeal to comparisons with Dirichlet heat flows to show Proposition \ref{Proposition: Parabolic Campanato}.

 Finally, we would like to mention again that the remaining case $p\in(1,\,p_{\mathrm c}\rbrack$ is treated in another paper \cite{T-subcritical} by the author.
 There, the local boundedness of a solution or its gradient is mainly discussed, and Theorem \ref{Thm: Key Hoelder Estimate} in this paper, which result is valid for any $p\in(1,\,\infty)$ as long as uniform gradient bounds are guaranteed, is used without a proof.

 \section{Approximation problems}\label{Sect:Approximation}
 In Section \ref{Sect:Approximation}, we consider a parabolic approximate problem, based on the convolution of $E=E_{1}+E_{p}$ and the Friedrichs mollifiers.
 More precisely, we choose and fix a spherically symmetric and non-negative function $j\in C_{\mathrm c}^{\infty}({\mathbb R}^{n})$ such that $\lVert j\rVert_{L^{1}({\mathbb R}^{n})}=1$ holds, and its support is the closed ball $\overline{B_{1}(0)}$.
 For this smooth function $j$, we define
 \(j_{\varepsilon}(x)\coloneqq \varepsilon^{-n}j(x/\varepsilon)\) for \(\varepsilon\in(0,\,1),\,x\in{\mathbb R}^{n}\),
 often called the Friedrichs mollifier.
 For $\varepsilon\in(0,\,1)$, we define
 \begin{equation}\label{Eq (Section 2): Def of E-s-epsilon}
  E_{s,\,\varepsilon}(z)\coloneqq (j_{\varepsilon}\ast E_{s})(z)=\int_{{\mathbb R}^{n}}j_{\varepsilon}(z-y)E_{s}(y)\,{\mathrm d}y
 \end{equation}
 for $z\in{\mathbb R}^{n}$, $s\in\{\,1,\,p\,\}$.
 Section \ref{Sect:Approximation} aims to show convergence results for weak solutions to approximate problems, based on our straightforward relaxation of $E=E_{1}+E_{p}$ by $E_{\varepsilon}\coloneqq E_{1,\,\varepsilon}+E_{p,\,\varepsilon}$.
 Also, we would like to prove our main Theorem \ref{Thm: Main theorem parabolic} by making use of a priori estimates for approximate solutions, which are shown in Sections \ref{Sect:Estimates from Weak Form}--\ref{Sect:Non-Degenerate Region}.
 \subsection{Relaxation of divergence operators and some basic estimates}\label{Subsect: Relaxation Results}
 In Section \ref{Subsect: Relaxation Results}, we briefly note some basic estimates of $E_{1}$, $E_{\varepsilon}=E_{1,\,\varepsilon}+E_{p,\,\varepsilon}$, and ${\mathcal G}_{2\delta,\,\varepsilon}$.

 By (\ref{Eq (Section 1): Positively one-homogeneous}), for any \(k\in(0,\,\infty)\),\(z\in{\mathbb R}^{n}\setminus\{0\}\), $E_{1}$ satisfies 
 \(\nabla E_{1}(kz)=\nabla E_{1}(z)\) and \(\nabla^{2}E_{1}(kz)=k^{-1}\nabla^{2}E_{1}(z)\)
 In particular, by $E_{1}\in C^{2}({\mathbb R}^{n}\setminus\{0\})$ and  (\ref{Eq (Section 1): (Assumption) modulus of continuity of Hess E1}),  it is easy to find a constant $K_{0}\in(0,\,\infty)$ such that
 \begin{align}
  \label{Eq (Section 1): Gradient growth bounds for E-1}
  &\lvert \nabla E_{1}(z)\rvert\le K_{0}\quad \text{and}\quad  
  O\leqslant \nabla^{2}E_{1}(z)\leqslant K_{0}\lvert z\rvert^{-1}\mathrm{id}_{n}\quad \textrm{for all }z\in{\mathbb R}^{n}\setminus\{0\},\\ 
  \label{Eq (Section 1): modulus of continuity of Hess E1}
  &\left\lVert \nabla^{2}E_{1}(z_{1})-\nabla^{2}E_{1}(z_{2}) \right\rVert\le \mu^{-1}\omega_{1}(\lvert z_{1}-z_{2}\rvert/\mu) 
 \end{align}
 for all $\mu\in(0,\,\infty)$, $z_{1}$, $z_{2}\in{\mathbb R}^{n}\setminus \{0\}$ with $\mu/32\le \lvert z_{j}\rvert\le 3\mu$ for $j\in\{\,1,\,2\,\}$.
 
 By (\ref{Eq (Section 1): Gradient growth bounds for E-p})--(\ref{Eq (Section 1): Hessian estimates for E-p}) and (\ref{Eq (Section 1): Gradient growth bounds for E-1}), the relaxed density functions $E_{1,\,\varepsilon}$ and $E_{p,\,\varepsilon}$, given by (\ref{Eq (Section 2): Def of E-s-epsilon}), satisfies the following estimates for every $\varepsilon\in(0,\,1)$, $z\in{\mathbb R}^{n}$;
 \begin{align}
  \label{Eq (Section 2): bounds of E-1-epsilon}
  &\left\lvert \nabla E_{1,\,\varepsilon}(z) \right\rvert\le K,\quad O\leqslant \nabla^{2}E_{1,\,\varepsilon}(z)\leqslant K\left(\varepsilon^{2}+\lvert z\rvert^{2}\right)^{-\frac{1}{2}}  \mathop{\mathrm{id}_{n}},  \\ 
  \label{Eq (Section 2): bounds of E-p-epsilon}
  & \left\lvert \nabla E_{p,\,\varepsilon}(z) \right\rvert\le \Lambda\left(\varepsilon^{2}+\lvert z\rvert^{2} \right)^{\frac{p-1}{2}},\quad 
  \lambda\left(\varepsilon^{2}+\lvert z\rvert^{2} \right)^{\frac{p}{2}-1}  \mathop{\mathrm{id}_{n}}  \leqslant \nabla^{2}E_{p,\,\varepsilon}(z)\leqslant \Lambda\left(\varepsilon^{2}+\lvert z\rvert^{2} \right)^{\frac{p}{2}-1}  \mathop{\mathrm{id}_{n}}.  
\end{align}
 Here the positive constants $\lambda=\lambda(\lambda_{0},\,n,\,p)$, $\Lambda=\Lambda(\Lambda_{0},\,n,\,p)$, $K=K(K_{0},\,n)$ are independent of $\varepsilon\in(0,\,1)$.
 Moreover, by straightforward computations, we have
 \begin{equation}\label{Eq (Section 2): Strong monotoniticy of E-epsilon}
  \langle \nabla E_{\varepsilon}(z_{1})-\nabla E_{\varepsilon}(z_{2})\mid z_{1}-z_{2}\rangle
    \ge \left\{\begin{array}{cc}
    c(p)\lambda\lvert z_{1}-z_{2}\rvert^{p} & (p\ge 2),\\ 
    \lambda(\varepsilon^{2}+\lvert z_{1}\rvert^{2}+\lvert z_{2}\rvert^{2})^{p/2-1}\lvert z_{1}-z_{2}\rvert^{2}& (\textrm{otherwise}),
     \end{array}\right.
 \end{equation}
 for all $z_{1}$, $z_{2}\in{\mathbb R}^{n}$.
 Also, it is easy to verify that $E_{p}$ satisfies
 \begin{equation}\label{Eq (Section 2): Strong monotonicity of E-p}
   \left\langle \nabla E_{p}(z_{1})-\nabla E_{p}(z_{2})\mathrel{}\middle|\mathrel{} z_{1}-z_{2}\right\rangle>0
   \quad \text{for any}\quad z_{1},\,z_{2}\in{\mathbb R}^{n}\quad \text{with}\quad z_{1}\neq z_{2}.
 \end{equation}
 For the proofs of (\ref{Eq (Section 2): bounds of E-1-epsilon})--(\ref{Eq (Section 2): Strong monotoniticy of E-epsilon}), see \cite[\S A]{MR4201656} and \cite[\S 2]{T-scalar} (see also \cite{MR1634641}).

 Throughout Sections \ref{Sect:Approximation}--\ref{Sect:Non-Degenerate Region}, we deal with truncation mappings ${\mathcal G}_{\delta}$, ${\mathcal G}_{\delta,\,\varepsilon}\colon{\mathbb R}^{n}\to {\mathbb R}^{n}$, defined as
 \[ {\mathcal G}_{\delta}(z)\coloneqq (\lvert z\rvert-\delta)_{+}\frac{z}{\lvert z\rvert},\quad \text{and}\quad {\mathcal G}_{\delta,\,\varepsilon}(z)\coloneqq \left(\sqrt{\varepsilon^{2}+\lvert z\rvert^{2}}-\delta \right)_{+}\frac{z}{\lvert z\rvert}\]
 for $z\in{\mathbb R}^{n}$. 
 The truncation mapping ${\mathcal G}_{\delta,\,\varepsilon}$ makes sense as long as $0<\varepsilon<\delta$ holds, and this is Lipschitz continuous, provided $\varepsilon$ is much smaller than $\delta$.
 More precisely, if $0<\varepsilon<\delta/8$ holds, then by \cite[Lemma 2.4]{T-scalar}, we have
 \begin{equation}\label{Eq (Section 2): Lipschitz continuity of G-2delta-epsilon}
   \left\lvert {\mathcal G}_{2\delta,\,\varepsilon}(z_{1})-{\mathcal G}_{2\delta,\,\varepsilon}(z_{2}) \right\rvert\le c_{\dagger}\lvert z_{1}-z_{2}\rvert\quad \textrm{for all }z_{1},\,z_{2}\in{\mathbb R}^{n}
 \end{equation}
 with $c_{\dagger}\coloneqq 1+64/\sqrt{255}$.
 We also use some basic inequalities given in Lemma \ref{Lemma: Continuity estimates on Hessian}.
 \begin{lemma}\label{Lemma: Continuity estimates on Hessian}
   Let $E_{\varepsilon}=E_{1,\,\varepsilon}+E_{p,\,\varepsilon}$ be given by (\ref{Eq (Section 2): Def of E-s-epsilon}) for $\varepsilon\in(0,\,1)$.
   Assume that non-negative constants $\delta$, $\varepsilon$ and $M$ satisfy $\varepsilon<\delta/8$ and $2\delta<M$.
   Then, there exist constants $0<C_{1}<C_{2}<\infty$, depending at most on $n$, $p$, $\lambda$, $\Lambda$, $K$, $\delta$ and $M$, such that the following estimates (\ref{Eq (Section 2): Monotonicity outside a facet})--(\ref{Eq (Section 2): Growth outside a facet}) hold for all $z_{1}$, $z_{2}\in{\mathbb R}^{n}$ satisfying $\delta\le \lvert z_{1}\rvert\le M$ and $\lvert z_{2}\rvert\le M$.
   \begin{align}
    \label{Eq (Section 2): Monotonicity outside a facet}
    &\left\langle \nabla E_{\varepsilon}(z_{1})-\nabla E_{\varepsilon}(z_{2})\mathrel{}\middle|\mathrel{}z_{1}-z_{2} \right\rangle\ge C_{1}\lvert z_{1}-z_{2}\rvert^{2}.\\ 
    \label{Eq (Section 2): Growth outside a facet}
    &\left\lvert \nabla E_{\varepsilon}(z_{1})-\nabla E_{\varepsilon}(z_{2}) \right\rvert\le C_{2}\lvert z_{1}-z_{2}\rvert.
  \end{align}
 Moreover, there holds 
 \begin{equation}\label{Eq (Section 2): Continuity estimate on Hessian}
  \left\lvert \nabla^{2} E_{\varepsilon}(z_{1})(z_{1}-z_{2})-\left(\nabla E_{\varepsilon}(z_{1})-\nabla E_{\varepsilon}(z_{2}) \right) \right\rvert\le \lvert z_{1}-z_{2}\rvert\omega(\lvert z_{1}-z_{2}\rvert/\mu)
 \end{equation}
 for all $z_{1}$, $z_{2}\in{\mathbb R}^{n}$ satisfying $\delta+\frac{\mu}{4}\le \lvert z_{1}\rvert\le \delta+\mu$, $\lvert z_{2}\rvert\le \delta+\mu$ with $\delta<\mu<M-\delta$.
 Here the modulus of continuity $\omega=\omega_{\delta,\,M}$, which is concave, is of the form 
 \(\omega(\sigma)=C_{3}(\omega_{1}(\sigma)+\omega_{p}(\sigma)+\sigma)\)
 with $C_{3}\in (0,\,\infty)$ depending at most on $p$, $\Lambda$, $K$, $\delta$ and $M$.
 \end{lemma}
 In the special case where $\omega_{1}=\omega_{1}(\sigma)$ and $\omega_{p}=\omega_{p}(\sigma)$ behave like $\sigma^{\beta_{0}}$ for some $\beta_{0}\in(0,\,1)$, a similar continuity estimate to (\ref{Eq (Section 2): Continuity estimate on Hessian}) is shown in \cite[Lemma 2.6]{T-scalar}.
 The proofs of (\ref{Eq (Section 2): Monotonicity outside a facet})--(\ref{Eq (Section 2): Growth outside a facet}) are already given there.
 Modifying \cite[Lemma 2.6]{T-scalar}, we would like to give a proof of (\ref{Eq (Section 2): Continuity estimate on Hessian}).
 \begin{proof}
  By $\varepsilon<\delta/8$, we can check that
  \(\frac{\mu^{2}}{16}\le \varepsilon^{2}+\lvert z_{1}\rvert^{2}\le 5\mu^{2}\), and \(\varepsilon^{2}+\lvert z_{2}\rvert^{2}\le 5\mu^{2}\).
  Combining them with (\ref{Eq (Section 2): bounds of E-1-epsilon})--(\ref{Eq (Section 2): bounds of E-p-epsilon}) and $\delta<\mu<M-\delta$, we can easily check that the left-hand side of (\ref{Eq (Section 2): Continuity estimate on Hessian}) is bounded by some $C=C(p,\,\Lambda,\,K,\,\delta,\,M)\in(1,\,\infty)$,
  from which (\ref{Eq (Section 2): Continuity estimate on Hessian}) is easily concluded when $\lvert z_{1}-z_{2}\rvert>\mu/32$.
  For the remaining case $\lvert z_{1}-z_{2}\rvert\le \mu/32$, we use the triangle inequality to get 
  $\frac{7}{8}\delta+\frac{\mu}{4}\le \lvert z_{1}-y\rvert\le \frac{17}{8}\mu$, and $\frac{7}{8}\delta+\frac{\mu}{32}\le \lvert z_{1}-y+t(z_{1}-z_{2}) \rvert\le 3\mu$
  for all $y\in B_{\varepsilon}(0)$.
  By $\nabla^{2}E\in C({\mathbb R}^{n}\setminus\{0\})$, it is easy to check $\nabla^{2} E_{\varepsilon}=j_{\varepsilon}\ast \nabla^{2}E$ in ${\mathbb R}^{n}\setminus \overline{B_{\hat{\varepsilon}}(0)}$ for any fixed ${\hat\varepsilon}\in(\varepsilon,\,1)$.
  Using this identity, (\ref{Eq (Section 1): modulus of continuity of Hess Ep}), and (\ref{Eq (Section 1): modulus of continuity of Hess E1}) yields
  \begin{align*}
    &\left\lvert \nabla^{2} E_{\varepsilon}(z_{1})(z_{1}-z_{2})-\left(\nabla E_{\varepsilon}(z_{1})-\nabla E_{\varepsilon}(z_{2}) \right) \right\rvert\\ 
    &=\left\lvert \left(\int_{0}^{1} \left[\nabla^{2} E_{\varepsilon}(z_{1})-\nabla^{2}E_{\varepsilon}(z_{1}+t(z_{2}-z_{1}))\right]\,{\mathrm d}t\right)\cdot(z_{1}-z_{2}) \right\rvert\\  
    &\le \lvert z_{1}-z_{2}\rvert\sum_{s=1,\,p}\int_{0}^{1}\int_{B_{\varepsilon}(0)}\left\lVert \nabla^{2}E_{s}(z_{1}-y-t(z_{1}-z_{2}))-\nabla^{2}E_{s}(z_{1}-y) \right\rVert j_{\varepsilon}(y)\,{\mathrm d}y{\mathrm d}t\\ 
    &\le \lvert z_{1}-z_{2}\rvert\left[\delta^{-1}\omega_{1}(\lvert z_{1}-z_{2}\rvert/\mu)+C_{\delta,\,M}\omega_{p}(\lvert z_{1}-z_{2}\rvert/\mu)\right],
  \end{align*}
  from which we obtain (\ref{Eq (Section 2): Continuity estimate on Hessian}).
 \end{proof}
 
 Finally, for each $\varepsilon\in(0,\,1)$, we define the bijective mapping $G_{p,\,\varepsilon}\colon {\mathbb R}^{n}\rightarrow {\mathbb R}^{n}$ as 
 \begin{equation}\label{Eq (Section 2): Gpe}
  G_{p,\,\varepsilon}(z)\coloneqq \left(\varepsilon^{2}+\lvert z\rvert^{2} \right)^{(p-1)/2}z\quad  \text{for} \quad z\in{\mathbb R}^{n}.
 \end{equation}
 This mapping satisfies
 \begin{equation}\label{Eq (Section 6): Monotonicity of G-p-epsilon}
   \lvert G_{p,\,\varepsilon}(z_{1})-G_{p,\,\varepsilon}(z_{2}) \rvert\ge c(p) \max\left\{\,\lvert z_{1}\rvert^{p-1},\,\lvert z_{2}\rvert^{p-1} \right\}\lvert z_{1}-z_{2}\rvert
 \end{equation}
 for all $z_{1}$, $z_{2}\in{\mathbb R}^{n}$ (see \cite[Lemma 2.3 \& \S 3.5]{T-scalar}). 
 Moreover, by (\ref{Eq (Section 6): Monotonicity of G-p-epsilon}) and $G_{p,\,\varepsilon}(0)=0$, we have
 \begin{equation}\label{Eq (Section 6): Bound on G-p-epsilon-inverse}
   \left\lvert G_{p,\,\varepsilon}^{-1}(w)\right\rvert\le C(p)\lvert w\rvert^{1/p}\quad \textrm{for all}\quad w\in{\mathbb R}^{n}.
 \end{equation}
 The inequalities (\ref{Eq (Section 6): Monotonicity of G-p-epsilon})--(\ref{Eq (Section 6): Bound on G-p-epsilon-inverse}) are used in Section \ref{Sect:Non-Degenerate Region}.
 
 \subsection{Convergence of approximate solutions}\label{Subsect:Parabolic Convergence}
 For an approximation parameter $\varepsilon\in(0,\,1)$, we consider the following Dirichlet boundary value problem;
 \begin{equation}\label{Eq (Section 2): Approximated Dirichlet boundary}
  \left\{\begin{array}{rclcc}
    \partial_{t}u-\divx\left(\nabla E_{\varepsilon}(\nabla u_{\varepsilon})\right)&=&f_{\varepsilon}&\textrm{in}& \Omega_{T},\\ 
    u_{\varepsilon}& = & u_{\star} &\textrm{on}& \partial_{\mathrm{p}}\Omega_{T}.
  \end{array}\right.
 \end{equation}
 Here $u_{\star}\in X^{p}(0,\,T;\,\Omega)\cap C^{0}([0,\,T];\,L^{2}(\Omega))$ is a given function.
 For the approximation of  $f\in L^{q}(\Omega_{T})\cap L^{p^{\prime}}(0,\,T;\,W^{-1,\,p^{\prime}}(\Omega))$,  we only require the weak$^{\ast}$ convergence
 \begin{equation}\label{Eq (Section 2): Weak-star convergence of f}
   f_{\varepsilon}\stackrel{\ast}{\rightharpoonup} f\quad \textrm{in}\quad L^{q}(\Omega_{T})\quad \text{and}\quad  L^{p^{\prime}}(0,\,T;\,W^{-1,\,p^{\prime}}(\Omega)).  
  \end{equation}
   The assertion (\ref{Eq (Section 2): Weak-star convergence of f}) is easily justified, by extending $f$ to be zero outside $\Omega_{T}$, and convoluting this extended function with the $(n+1)$-dimensional Friedrichs mollifier. 
  In particular, we may let $f_{\varepsilon}\in C^{\infty}(\Omega_{T})$, and choose a constant $F\in(0,\,\infty)$ such that
   \begin{equation}\label{Eq (Section 2): Bound of f-epsilon}
     \lVert f_{\varepsilon}\rVert_{L^{q}(\Omega_{T})}+\lVert f_{\varepsilon}\rVert_{L^{p^{\prime}}(0,\,T;\,W^{-1,\,p^{\prime}}(\Omega))} \le F  \quad \textrm{for all}\quad \varepsilon\in(0,\,1).
   \end{equation}
 The weak solution of (\ref{Eq (Section 2): Approximated Dirichlet boundary}) is defined to be a function $u_{\varepsilon}\in u_{\star}+X_{0}^{p}(0,\,T;\,\Omega)$ satisfying $u_{\varepsilon}|_{t=0}=u_{\star}|_{t=0}$ in $L^{2}(\Omega)$, and
 \begin{equation}\label{Eq (Section 2): IDENTITY in the distribution}
  \partial_{t}u_{\varepsilon}-\divx\left(\nabla E_{\varepsilon}(\nabla u_{\varepsilon}) \right)=f_{\varepsilon}\quad \textrm{in}\quad L^{p^{\prime}}(0,\,T;\,L^{p^{\prime}}(\Omega)).
 \end{equation}
 In other words, the function $u_{\varepsilon}$ satisfies
 \begin{equation}\label{Eq (Section 2): Weak form for approximation problem}
   \int_{0}^{T}\left\langle \partial_{t}u_{\varepsilon},\,\varphi\right\rangle\,{\mathrm d}t+\iint_{\Omega_{T}}\left\langle \nabla E_{\varepsilon}(\nabla u_{\varepsilon})\mathrel{}\middle|\mathrel{}\nabla\varphi \right\rangle\,{\mathrm d}X=\iint_{\Omega_{T}}f_{\varepsilon}\varphi\,{\mathrm d}X.
  \end{equation}
 for all $\varphi\in L^{p}(0,\,T;\,W_{0}^{1,\,p}(\Omega))$.
The unique existence of the weak solution of (\ref{Eq (Section 2): Approximated Dirichlet boundary}) is well-established in 
 found in \cite[Chapitre 2, \S 1]{MR0259693}, \cite[\S III.4]{MR1422252}, we can show the unique existence of (\ref{Eq (Section 2): Approximated Dirichlet boundary}).
 There the continuous inclusions $ W_{0}^{1,\,p}(\Omega)\hookrightarrow L^{2}(\Omega)\hookrightarrow W^{-1,\,p^{\prime}}(\Omega)$ is used to construct a weak solution by the Faedo--Galerkin method.
 It appears that the monotonicity methods in \cite[Chapitre 2, \S 1]{MR0259693}, \cite[\S III.4]{MR1422252} cannot be directly applied to (\ref{Eq (Section 1): General Parabolic Eq}), since 
 the density $E_{1}$ lacks $C^{1}$-regularity at the origin.
 In particular, the existence of a subgradient vector field $Z$ itself appears non-trivial.
 To overcome this difficulty, we invoke a convergence lemma on vector fields (Lemma \ref{Lemma: Fundamental lemma on vector fields}).
 This lemma is a special case of \cite[Lemma 2.8]{T-scalar}, which helps to prove the convergence of weak solutions.
 \begin{lemma}\label{Lemma: Fundamental lemma on vector fields}
  Let $E_{1}\in C^{0}({\mathbb R}^{n})\cap C^{1}({\mathbb R}^{n}\setminus\{0\})$, $E_{p}\in C^{1}({\mathbb R}^{n})$ satisfy (\ref{Eq (Section 1): Positively one-homogeneous})--(\ref{Eq (Section 1): Gradient growth bounds for E-p}).
  For each $\varepsilon\in(0,\,1)$, define $E_{\varepsilon}=E_{1,\,\varepsilon}+E_{p,\,\varepsilon}$ by (\ref{Eq (Section 2): Def of E-s-epsilon}).
  Then, the following hold. 
  \begin{enumerate}
    \item\label{Item 1/2} For each fixed $v\in L^{p}(\Omega_{T})$, we have 
    \(\nabla E_{\varepsilon}(v)\to A_{0}(v)\) in \(L^{p^{\prime}}(\Omega_{T};\,{\mathbb R}^{n})\) as \(\varepsilon\to 0\).
    Here $A_{0}\colon{\mathbb R}^{n}\rightarrow {\mathbb R}^{n}$ is defined as $A_{0}(z)\coloneqq \nabla E(z)$ for $z\neq 0$, and $A_{0}(0)\coloneqq (j\ast \nabla E_{1})(0)$.
    \item\label{Item 2/2} Assume that a sequence $\{v_{\varepsilon_{k}}\}_{k}\subset L^{p}(\Omega_{T};\,{\mathbb R}^{n})$, where $\varepsilon_{k}\to 0$ as $k\to 0$, satisfies
      \(v_{\varepsilon_{k}}\to v_{0}\quad \textrm{in}\quad L^{p}(\Omega_{T};\,{\mathbb R}^{n})\) as \(k\to\infty\)
      for some $v_{0}\in L^{p}(\Omega_{T};\,{\mathbb R}^{n})$.
      Then, up to a subsequence, 
      \[\nabla E_{p,\,\varepsilon_{k}}(v_{\varepsilon_{k}})\to \nabla E_{p}(v_{0})\quad \textrm{in }L^{p^{\prime}}(\Omega_{T};\,{\mathbb R}^{n})\quad \text{and}\quad \nabla E_{1,\,\varepsilon_{k}}(v_{\varepsilon_{k}}) \stackrel{\ast}{\rightharpoonup}  Z \quad  \textrm{in } L^{\infty} (\Omega_{T};\,{\mathbb R}^{n})\]
      hold as $k\to\infty$, and the weak$^{\ast}$ limit $Z\in L^{\infty}(\Omega_{T};\,{\mathbb R}^{n})$ satisfies 
      \(Z(x,\,t)\in\partial E_{1}(v_{0}(x,\,t))\) for a.e.~\((x,\,t)\in\Omega_{T}.\)
    \end{enumerate}  
 \end{lemma}
 Applying Lemma \ref{Lemma: Fundamental lemma on vector fields}, we would like to prove that $u_{\varepsilon}$ converges to a weak solution of
 \begin{equation}\label{Eq (Section 2): Dirichlet boundary}
  \left\{\begin{array}{rclrl}
   \partial_{t}u+{\mathcal L}u & = & f&\textrm{in}&\Omega_{T-\tau},\\
   u & = & u_{\star}&\textrm{on}&\partial_{\mathrm{p}}\Omega_{T-\tau},
  \end{array} \right.
 \end{equation}
 for each fixed $\tau\in(0,\,T)$, in the sense of Definition \ref{Def: Dirichlet boundary} below.
 \begin{definition}\upshape\label{Def: Dirichlet boundary}
   Let $p$, $q$ satisfy (\ref{Eq (Section 1): p-q condition}).
   Fix $\tau\in(0,\,T)$,  $f\in L^{q}(\Omega_{T})\cap L^{p^{\prime}}(0,\,T;\,W^{-1,\,p^{\prime}}(\Omega))$,  and $u_{\star}\in X^{p}(0,\,T-\tau;\,\Omega)\cap C^{0}(\lbrack 0,\,T-\tau \rbrack;\,L^{2}(\Omega))$. 
   A function $u\in u_{\star}+X_{0}^{p}(0,\,T-\tau;\,\Omega)$ is called a weak solution of (\ref{Eq (Section 2): Dirichlet boundary}) if the identity $u|_{t=0}=u_{\star}|_{t=0}$ holds in $L^{2}(\Omega)$, and there exists $Z\in L^{\infty}(\Omega_{T-\tau})$ such that the pair $(u,\,Z)$ satisfies (\ref{Eq (Section 1): Weak form Definition})--(\ref{Eq (Section 1): Subgradient parabolic}) with $T$ replaced by $T-\tau$.
 \end{definition}
 \begin{proposition}\label{Prop: Convergence result}
   Let $p$, $q$ satisfy (\ref{Eq (Section 1): p-q condition}), and fix a function $u_{\star}\in X^{p}(0,\,T;\,\Omega)\cap C^{0}(\lbrack 0,\,T\rbrack;\,L^{2}(\Omega))$.
   Assume that $\{f_{\varepsilon}\}_{0<\varepsilon<1}\subset  L^{q}(\Omega_{T})\cap L^{p^{\prime}}(0,\,T;\,W^{-1,\,p^{\prime}}(\Omega)) $ satisfy (\ref{Eq (Section 2): Weak-star convergence of f}).
   For each $\varepsilon\in (0,\,1)$, we consider $u_{\varepsilon}\in u_{\star}+X_{0}^{p}(0,\,T;\,\Omega)$,  the  weak solution of (\ref{Eq (Section 2): Approximated Dirichlet boundary}).
   Then, there uniquely exists a function $u_{0}\in u_{\star}+X_{0}^{p}(0,\,T-\tau;\,\Omega)$ for each fixed $\tau\in(0,\,T)$, such that 
   \begin{equation}\label{Eq (Section 2): Convergence strong/pointwise}
     u_{\varepsilon_{j}}\to u_{0}\quad \textrm{in}\quad L^{p}(0,\,T-\tau;\,W^{1,\,p}(\Omega)),\quad \textrm{and}\quad \nabla u_{\varepsilon_{j}}\to \nabla u_{0}\quad \textrm{a.e. in}\quad \Omega_{T-\tau}.
   \end{equation}
   Moreover, this limit function $u_{0}$ is the unique weak solution of (\ref{Eq (Section 2): Dirichlet boundary}).
 \end{proposition}
 The strong $L^{p}(W^{1,\,p})$-convergence of weak solutions for parabolic approximate $p$-Laplace-type problems can be found in \cite[Lemma 3.1]{MR2916967}.
 The proof therein cannot be directly applied to our approximate problems, since the density $E_{1}$, unlike $E_{p}$, lacks continuous differentiability at the origin.
 In particular, we have to deduce (\ref{Eq (Section 2): Convergence strong/pointwise}) only by the weak convergence of $u_{\varepsilon}$ or $\partial_{t}u_{\varepsilon}$.
 To prove the strong $L^{p}$-convergence of a spatial gradient, we have to appeal to the Aubin--Lions lemma and Lemma \ref{Lemma: Fundamental lemma on vector fields}.
 In our proof, this bound estimate is required to make a weak compactness argument, and to apply the compact embedding $X_{0}^{p}(0,\,T-\tau/2;\,\Omega)\hookrightarrow\hookrightarrow L^{p}(0,\,T-\tau/2;\,L^{2}(\Omega))$.
 Since our convergence argument is based on a weak convergence method and a parabolic compact embedding, this is substantially different from the proof of \cite[Lemma 3.1]{MR2916967}.
 \begin{proof}
   We may choose $\varepsilon_{0}\in(0,\,1)$ and a constant $c_{\star}\in(K,\,\infty)$ such that
   \(\lvert E_{\varepsilon}(0)\rvert+\lvert \nabla E_{\varepsilon}(0)\rvert \le c_{\star}\) for all \(\varepsilon\in(0,\,\varepsilon_{0})\).
   We fix arbitrary $\tau\in(0,\,T)$, and define $0<T_{2}\coloneqq T-\tau<T_{1}\coloneqq T-\tau/2<T_{0}\coloneqq T$.
   For $j\in\{\,0,\,1,\,2\,\}$, we set a piecewise linear function
   \(\theta_{j}(t)\coloneqq \frac{(T_{j}-t)_{+}}{T_{j}}\in\lbrack 0,\,1\rbrack\) for \(t\in\lbrack 0,\,T\rbrack\).
   We first claim that there exists $C\in(1,\,\infty)$, independent of $\varepsilon\in(0,\,\varepsilon_{0})$, such that
   \begin{equation}\label{Eq (Section 2): First claim on convergence}
     \lVert \nabla u_{\varepsilon}\rVert_{L^{p}(\Omega_{T_{1}})}\le C.
   \end{equation}
   To verify this, we test $\varphi=\varphi_{0}\coloneqq (u_{\varepsilon}-u_{\star})\theta_{0}$ into (\ref{Eq (Section 2): Weak form for approximation problem}).
   As a result, we have   
   \begin{align*}
     {\mathbf L}_{1}+{\mathbf L}_{2}&\coloneqq -\frac{1}{2}\iint_{\Omega_{T}}\lvert u_{\varepsilon}-u_{\star}\rvert^{2}\partial_{t}\theta_{0}\,{\mathrm d}X+\iint_{\Omega_{T}}\left\langle \nabla E_{\varepsilon}(\nabla u_{\varepsilon})\mathrel{}\middle|\mathrel{}\nabla u_{\varepsilon}-\nabla u_{\star} \right\rangle\theta_{0}\,{\mathrm d}X\nonumber\\ 
     &=\iint_{\Omega_{T}}f_{\varepsilon}(u_{\varepsilon}-u_{\star})\theta_{0}\,{\mathrm d}X-\int_{0}^{T}\left\langle\partial_{t} u_{\star},\, \theta_{0}(u_{\varepsilon}-u_{\star})\right\rangle\,{\mathrm d}t\eqqcolon {\mathbf R}_{1}-{\mathbf R}_{2}.
   \end{align*}
   For ${\mathbf L}_{2}$, we use the convexity of $E_{\varepsilon}$ to obtain
   \[{\mathbf L}_{2}\ge \iint_{\Omega_{T}}\left[E_{\varepsilon}(\nabla u_{\varepsilon})-E_{\varepsilon}(\nabla u_{\star})\right]\theta_{0}\,{\mathrm d}X.\]
   By $\theta_{0}^{p}\le \theta_{0}\le 1$ and Minkowski's inequality, the integral ${\mathbf R}_{2}$ is estimated as follows;
   \[\lvert {\mathbf R}_{2}\rvert\le \lVert \partial_{t}u_{\star} \rVert_{L^{p^{\prime}}(0,\,T;\,W^{-1,\,p^{\prime}}(\Omega))}\left[\left(\iint_{\Omega_{T}}\lvert \nabla u_{\varepsilon}\rvert^{p} \theta_{0}\,{\mathrm d}X\right)^{1/p}+\left(\iint_{\Omega_{T}}\lvert \nabla u_{\star} \rvert^{p} \theta_{0}\,{\mathrm d}X\right)^{1/p} \right].\]
   For the remaining term ${\mathbf R}_{1}$, we use Young's inequality,  H\"{o}lder's inequality,  and (\ref{Eq (Section 2): Bound of f-epsilon}) to get 
   \[  \lvert {\mathbf R}_{1} \rvert\le \iint_{\Omega_{T}}\lvert f_{\varepsilon}\rvert\lvert u_{\varepsilon}-u_{\star} \rvert\,{\mathrm d}X \le \frac{1}{2T_{0}}\iint_{\Omega_{T}}\lvert u_{\varepsilon}-u_{\star}\rvert^{2}\,{\mathrm d}X+\frac{T_{0}}{2} \iint_{\Omega_{T}}\lvert f_{\varepsilon}\rvert^{2}\,{\mathrm d}X  \le{\mathbf L}_{1}+C(q,\,\Omega,\,T,\,T_{0})F^{2}.  \]
   By (\ref{Eq (Section 2): Strong monotoniticy of E-epsilon}) and straightforward computations found in \cite[\S 2.6]{T-scalar}, 
   we can find constants $c_{1}\in(0,\,1)$, $c_{2}\in(1,\,\infty)$, $c_{3}\in(1,\,\infty)$ such that
   \[c_{1}(p,\,\lambda)\left(\lvert z\rvert^{p}-c_{2}(p,\,c_{\star})\right)\le E_{\varepsilon}(z)\le c_{3}(p,\,\Lambda,\,K,\,c_{\star})(\lvert z\rvert^{p}+1)\]
   for all $z\in{\mathbb R}^{n}$.
   Carrying out a straightforward absorbing argument, we can find the constant $C\in(1,\,\infty)$ satisfying (\ref{Eq (Section 2): First claim on convergence}).

   By (\ref{Eq (Section 2): First claim on convergence}), it is easy to check that $\{u_{\varepsilon}-u_{\star}\}_{0<\varepsilon<1}\subset L^{p}(0,\,T;\,W_{0}^{1,\,p}(\Omega))$ is bounded.
   Moreover, by (\ref{Eq (Section 2): bounds of E-1-epsilon})--(\ref{Eq (Section 2): bounds of E-p-epsilon}), and (\ref{Eq (Section 2): Bound of f-epsilon})--(\ref{Eq (Section 2): IDENTITY in the distribution}),
   the time derivative $\partial_{t} u_{\varepsilon}\in L^{p^{\prime}}(0,\,T_{1};\,W^{-1,\,p^{\prime}}(\Omega))$ is also uniformly bounded for $\varepsilon\in(0,\,\varepsilon_{0})$.
   Therefore, by a weak compactness argument, there exists a sequence $\{\varepsilon_{j}\}_{j=1}^{\infty}\subset (0,\,\varepsilon_{0})$, which converges to $0$, such that
   \begin{align}
    &u_{\varepsilon_{j}}-u_{\star}\rightharpoonup u_{0}-u_{\star}\quad  \textrm{in}\quad L^{p}(0,\,T_{1};\,W_{0}^{1,\,p}(\Omega)),\label{Eq (Section 2): Weak convergence 1} \\ 
    &\partial_{t}u_{\varepsilon_{j}}\stackrel{\ast}{\rightharpoonup} \partial_{t}u_{0}\quad \textrm{in}\quad L^{p^{\prime}}(0,\,T_{1};\,W^{-1,\,p^{\prime}}(\Omega)),\label{Eq (Section 2): Weak convergence 2}
   \end{align}
   for some $u_{0}\in u_{\star}+X_{0}^{p}(0,\,T_{1};\,\Omega)$.
   By making use of (\ref{Eq (Section 2): Weak convergence 1})--(\ref{Eq (Section 2): Weak convergence 2}), and integration by parts for the time variable, it is easy to check that $u_{0}|_{t=0}=u_{\star}|_{t=0}$ in $L^{2}(\Omega)$.
   We also invoke the compact embedding $X_{0}^{p}(0,\,T_{1};\,\Omega)\hookrightarrow \hookrightarrow L^{p}(0,\,T_{1};\,L^{2}(\Omega))$ from the Aubin--Lions lemma.
   In particular, we are allowed to use a strong convergence
   \begin{equation}\label{Eq (Section 2): Strong convergence from Aubin-Lions}
     u_{\varepsilon_{j}}\to u_{0}\quad \textrm{in}\quad L^{p}(0,\,T_{1};\,L^{2}(\Omega))
   \end{equation} 
   by taking a subsequence if necessary. 

   From (\ref{Eq (Section 2): Weak convergence 1})--(\ref{Eq (Section 2): Strong convergence from Aubin-Lions}), we would like to show
   \begin{equation}\label{Eq (Section 2): Second claim on convergence}
    {\tilde {\mathbf L}}\coloneqq  \iint_{\Omega_{T_{1}}}\left\langle \nabla E_{\varepsilon_{j}}(\nabla u_{\varepsilon_{j}})-\nabla E_{\varepsilon_{j}}(\nabla u_{0})\mathrel{}\middle|\mathrel{}\nabla u_{\varepsilon_{j}}-\nabla u_{0} \right\rangle \theta_{1}\,{\mathrm d}X\to 0
   \end{equation}
   as $j\to\infty$.
   To prove (\ref{Eq (Section 2): Second claim on convergence}), we test $\varphi=\varphi_{1}\coloneqq (u_{\varepsilon}-u_{0})\theta_{1}$ into (\ref{Eq (Section 2): Weak form for approximation problem}) with $\varepsilon=\varepsilon_{j}$.
   Then, we have 
   \begin{align*}
   {\tilde{\mathbf L}_{1}}+{\tilde {\mathbf L}}&\coloneqq -\frac{1}{2}\iint_{\Omega_{T_{1}}}\lvert u_{\varepsilon_{j}}-u_{0}\rvert^{2}\partial_{t}\theta_{1}\,{\mathrm d}X+{\tilde {\mathbf L}}\\ 
   &=\iint_{\Omega_{T_{1}}}f_{\varepsilon_{j}}(u_{\varepsilon_{j}}-u_{0})\theta_{1}\,{\mathrm d}X-\int_{0}^{T_{1}}\left\langle \partial_{t}u_{0},\, (u_{\varepsilon_{j}}-u_{0})\theta_{1}\right\rangle_{W^{-1,\,p^{\prime}}(\Omega),\,W_{0}^{1,\,p}(\Omega)}\,{\mathrm d}t\\ 
   &\quad -\iint_{\Omega_{T_{1}}}\left\langle \theta_{1} \nabla E_{\varepsilon_{j}}(\nabla u_{0}) \mathrel{}\middle|\mathrel{} \nabla u_{\varepsilon_{j}}-\nabla u_{0} \right\rangle\,{\mathrm d}X\\ 
   &\eqqcolon {\tilde {\mathbf R}_{1}}-{\tilde {\mathbf R}_{2}}-{\tilde {\mathbf R}_{3}}.
   \end{align*}
   To estimate ${\tilde{\mathbf R}_{1}}$, we introduce $\kappa\in(0,\,1)$, and define
   $r=r(\kappa,\,q)\in (0,\,q^{\prime})$ as $r(\kappa,\,q)\coloneqq \frac{2q\kappa}{(1+\kappa)q-2}$ for finite $q$, and $r(\kappa,\,\infty)\coloneqq \frac{2\kappa}{1+\kappa}$.
   By H\"{o}lder's inequality and Young's inequality, we have
   \begin{align*}
    \lvert {\tilde {\mathbf R}_{1}}\rvert&\le \left(\iint_{\Omega_{T_{1}}}\lvert u_{\varepsilon_{j}}-u_{0} \rvert^{2}\,{\mathrm d}X \right)^{\frac{1-\kappa}{2}}\left(\iint_{\Omega_{T_{1}}}\lvert f_{\varepsilon_{j}}\rvert^{\frac{2}{1+\kappa}}\lvert u_{\varepsilon_{j}}-u_{0} \rvert^{\frac{2\kappa}{1+\kappa}}\,{\mathrm d}X\right)^{\frac{1+\kappa}{2}}\\ 
    &\le {\tilde {\mathbf L}_{1}}+C(\kappa)T_{1}^{\frac{1-\kappa}{1+\kappa}}\left(\iint_{\Omega_{T_{1}}}\lvert f_{\varepsilon_{j}}\rvert^{q}\,{\mathrm d}X\right)^{\frac{2}{q(1+\kappa)}}\left(\iint_{\Omega_{T_{1}}}\lvert u_{\varepsilon_{j}}-u_{0}\rvert^{r}  \,{\mathrm d}X \right)^{\frac{2\kappa}{(1+\kappa)r}}\eqqcolon {\tilde {\mathbf L}_{1}}+{\tilde {\mathbf R}_{4}}.
   \end{align*}
   As a result, we get
   \(0\le {\tilde{\mathbf L}}\le -({\tilde{\mathbf R}_{2}}+{\tilde{\mathbf R}_{3}})+{\tilde{\mathbf R}_{4}},\)
   and therefore it suffices to check $\tilde{\mathbf R}_{k}\to 0$ for each $k\in\{\,2,\,3,\,4\,\}$.
   The integral ${\tilde {\mathbf R}_{2}}$ is computed as follows 
   \[{\tilde{\mathbf R}_{2}}=\int_{0}^{T_{1}}\langle \theta_{1}\partial_{t}u_{0},\,(u_{\varepsilon_{j}}-u_{\star})-(u_{0}-u_{\star}) \rangle_{W^{-1,\,p^{\prime}}(\Omega),\,W_{0}^{1,\,p}(\Omega)}\,{\mathrm d}t.\]
   Noting $\theta_{1}\partial_{t}u_{0}\in L^{p^{\prime}}(0,\,T_{1};\,W^{-1,\,p^{\prime}}(\Omega))$ and (\ref{Eq (Section 2): Weak convergence 2}), we can deduce ${\tilde {\mathbf R}_{2}}\to 0$.
   For ${\tilde{\mathbf R}_{3}}$, we recall Lemma \ref{Lemma: Fundamental lemma on vector fields}, so that we are allowed to use $\theta_{1}\nabla E_{\varepsilon_{j}}(\nabla u_{0})\to \theta_{1}A_{0}(\nabla u_{0})$ in $L^{p^{\prime}}(\Omega_{T_{1}})$.
   We already know $\nabla u_{\varepsilon_{j}}\rightharpoonup \nabla u_{0}$ in $L^{p}(\Omega_{T_{1}})$ by (\ref{Eq (Section 2): Weak convergence 1}).
   Combining these convergence results, we can easily notice ${\tilde {\mathbf R}}_{3}\to 0$.
   For ${\tilde{\mathbf R}_{4}}$, we choose and fix $\kappa\in(0,\,1)$ satisfying $0<r(\kappa)<\min\{\,p,\,2\,\}$.
   Then, (\ref{Eq (Section 2): Bound of f-epsilon}), (\ref{Eq (Section 2): Strong convergence from Aubin-Lions}) and the continuous embedding $L^{p}(0,\,T_{1};\,L^{2}(\Omega))\hookrightarrow L^{r}(\Omega_{T_{1}})$ yield ${\mathbf R}_{4}\to 0$.
   This completes the proof of (\ref{Eq (Section 2): Second claim on convergence}).

   From (\ref{Eq (Section 2): Strong monotoniticy of E-epsilon}), (\ref{Eq (Section 2): First claim on convergence}), and (\ref{Eq (Section 2): Second claim on convergence}), it is easy to deduce $\nabla u_{\varepsilon_{j}}\to \nabla u_{0}$ in $L^{p}(\Omega_{T_{2}})$.
   whence we may let (\ref{Eq (Section 2): Convergence strong/pointwise}) hold, by taking a subsequence if necessary.
   This fact enables us to apply Lemma \ref{Lemma: Fundamental lemma on vector fields} \ref{Item 2/2}.
   In particular, by choosing a suitable sequence $\{\varepsilon_{j}\}_{j}\subset (0,\,\varepsilon_{0})$, and letting $j\to 0$, we conclude that the limit $u_{0}\in u_{\star}+X_{0}^{p}(0,\,T_{2};\,\Omega)$ admits a vector field $Z_{0}\in L^{\infty}(\Omega_{T_{2}})$ such that the pair $(u_{0},\,Z_{0})$ satisfies (\ref{Eq (Section 1): Weak form Definition})--(\ref{Eq (Section 1): Subgradient parabolic}) with $T$ replaced by $T_{2}$.
   In other words, the limit $u_{0}$ is a weak solution of (\ref{Eq (Section 2): Dirichlet boundary}).
   Our proof is completed by showing that the weak solution of (\ref{Eq (Section 2): Dirichlet boundary}) is unique.
   Let another pair $({\tilde u},\,{\tilde Z})\in (u_{\star}+X_{0}^{p}(0,\,T_{2};\,\Omega))\times L^{\infty}(\Omega_{T_{2}})$ satisfy (\ref{Eq (Section 1): Weak form Definition})--(\ref{Eq (Section 1): Subgradient parabolic}) with $T$ replaced by $T_{2}$. 
   Then, 
   \begin{align}\label{Eq (Section 2): Comparison u-0 vs u-tilde}
     &\int_{0}^{T_{2}}\left\langle \partial_{t}(u_{0}-{\tilde u}),\, \varphi\right\rangle_{W^{-1,\,p^{\prime}}(\Omega),\,W_{0}^{1,\,p}(\Omega)}\,{\mathrm d}t+\iint_{\Omega_{T_{2}}}\langle Z_{0}-{\tilde Z}\mid \nabla\varphi\rangle\,{\mathrm d}X\\ 
     &\quad  +\iint_{\Omega_{T_{2}}}\langle \nabla E_{p}(\nabla u_{0})-\nabla E_{p}(\nabla {\tilde u})\mid \nabla\varphi\rangle\,{\mathrm d}X=0\nonumber
   \end{align}
   holds for all $\varphi\in X_{0}^{p}(0,\,T_{2};\,\Omega)$.
   Thanks to the Steklov average, we may test $\varphi=\varphi_{2}\coloneqq (u_{0}-{\tilde u})\theta_{2}$ into (\ref{Eq (Section 2): Comparison u-0 vs u-tilde}).
   Here we may discard the first and the second integrals in the left-hand side of (\ref{Eq (Section 2): Comparison u-0 vs u-tilde}), since they are positive.
   In fact, integrating by parts in time, we easily notice that the first integral in (\ref{Eq (Section 2): Comparison u-0 vs u-tilde}) is non-negative.
   For the second integral, we use subgradient inequalities to deduce
   \(\langle Z_{0}-{\tilde Z}\mid \nabla u_{0}-\nabla {\tilde u}\rangle\ge 0\) a.e.~in $\Omega_{T_{2}}$.
   Letting $\tilde\varepsilon\to 0$ and recalling (\ref{Eq (Section 2): Strong monotonicity of E-p}), we have $\nabla (u_{0}-{\tilde u})=0$ in $L^{p}(\Omega_{T_{2}})$.
   Recalling $u_{0}-\tilde{u}\in L^{p}(0,\,T_{2};\,W_{0}^{1,\,p}(\Omega))$ yields $u_{0}=\tilde{u}$ in $L^{p}(0,\,T_{2};\,W_{0}^{1,\,p}(\Omega))$.
 \end{proof}
 \begin{remark}\label{Rmk; Subcase}\upshape
 In the proof of Proposition \ref{Prop: Convergence result}, we only use the compact embedding $X_{0}^{p}(0,\,T;\,\Omega)\hookrightarrow \hookrightarrow L^{p}(0,\,T;\,L^{2}(\Omega))$ to deal with the non-trivial external force term $f_{\varepsilon}$.
 In the simple case $f_{\varepsilon}\equiv 0$, Proposition \ref{Prop: Convergence result} is shown without this compact embedding.
  \end{remark}
 \subsection{Uniform a priori estimates and proof of Theorem \ref{Thm: Main theorem parabolic}}\label{Subsect: Holder a priori}
 By Proposition \ref{Prop: Convergence result}, our gradient continuity problem is reduced to the regularity for $u_{\varepsilon}\in X^{p}(0,\,T;\,\Omega)\cap C^{0}(\lbrack 0,\,T\rbrack;\,L^{2}(\Omega))$, a weak solution to 
 \begin{equation}\label{Eq (Section 2): Approximation equation}
   \partial_{t}u_{\varepsilon}-\divx\left(\nabla E_{\varepsilon}(\nabla u_{\varepsilon}) \right)=f_{\varepsilon} 
 \end{equation}
 in $\Omega_{T}$.
 We prove local H\"{o}lder estimates of ${\mathcal G}_{2\delta,\,\varepsilon}(\nabla u_{\varepsilon})$. 

 We consider (\ref{Eq (Section 2): Approximation equation}) in a parabolic cylinder $Q_{R}=Q_{R}(x_{\ast},\,t_{\ast})\coloneqq B_{R}(x_{\ast})\times(t_{\ast}-R^{2},\,t_{\ast}\rbrack\Subset \Omega_{T}$.
 Then, we are allowed to use the interior regularity 
 \begin{equation}\label{Eq (Section 2): Improved reg}
  \nabla u_{\varepsilon}\in L^{\infty}(Q_{R};\,{\mathbb R}^{n})\quad \text{and}\quad  \nabla^{2}u_{\varepsilon}\in L^{2}(Q_{R};\,{\mathbb R}^{n\times n}). 
 \end{equation}
 Here we should note the fact that the parabolicity ratio of the Hessian matrix $\nabla^{2} E_{\varepsilon}(\nabla u_{\varepsilon})$ is everywhere bounded for each fixed $\varepsilon\in(0,\,1)$.
 Therefore, it is possible to make standard parabolic arguments, including difference quotient methods and the De Giorgi--Nash--Moser theory (\cite[Chapters VIII]{MR1230384}, \cite[Chapters III--V]{MR0241822}; see also \cite[Appendix]{MR1135917}).
 Differentiating (\ref{Eq (Section 2): Approximation equation}) with respect to $x_{j}$ for each $j\in\{\,1,\,\dots\,,\,n\,\}$, we have 
 \begin{equation}\label{Eq (Section 3): Weak formulation differentiated}
   -\int_{t_{\ast}-R^{2}}^{t_{\ast}}\langle \partial_{t}\partial_{x_{j}}u_{\varepsilon},\,\varphi\rangle\, {\mathrm d}t+\iint_{Q_{R}}\left[\left\langle\nabla^{2}E_{\varepsilon}(\nabla u_{\varepsilon})\nabla \partial_{x_{j}}u_{\varepsilon}\mathrel{}\middle|\mathrel{}\nabla\varphi \right\rangle+f_{\varepsilon}\partial_{x_{j}}\varphi\right]\,{\mathrm d}X=0
 \end{equation}
 for all $\varphi\in X_{0}^{2}(t_{\ast}-R^{2},\,t_{\ast};\, B_{R}(x_{\ast}))$.
 \begin{remark}\upshape
   Throughout Sections \ref{Sect:Estimates from Weak Form}--\ref{Sect:Non-Degenerate Region}, we often give formal computations, in the sense that the time derivatives $\partial_{t}u_{\varepsilon}$, $\partial_{t}\partial_{x_{j}}u_{\varepsilon}$ are treated as some sort of \textit{functions}, although they should be treated in the \textit{functional} sense, including $L^{p^{\prime}}(W^{-1,\,p^{\prime}})$ and $L^{2}(W^{-1,\,2})$.
  However, we remark that all of the computations in Sections \ref{Sect:Estimates from Weak Form}--\ref{Sect:Non-Degenerate Region} make sense by appealing to the Stekrov average.  
 \end{remark}
 
  Theorem \ref{Proposition: Lipschitz} 
  states local $L^{\infty}-L^{p}$ estimates of $V_{\varepsilon}\coloneqq \sqrt{\varepsilon^{2}+\lvert \nabla u_{\varepsilon}\rvert^{2}}$.
  \begin{theorem}\label{Proposition: Lipschitz}
   Assume that the exponents $p$ and $q$ satisfy (\ref{Eq (Section 1): p-q condition}).
   Let $u_{\varepsilon}$ be a weak solution to (\ref{Eq (Section 2): Approximation equation}) in $\Omega_{T}$ with $\varepsilon\in(0,\,1)$.
   Define the positive exponents $d$ and $d_{1}$ by 
   $(d,\,d_{1})\coloneqq (p/2,\,2)$ when $p\ge 2$, and otherwise $(d,\,d_{1})\coloneqq (2p/[p(n+2)-2n],\,p^{\prime})$.
   We also set $d_{2}\coloneqq d(n+2)/p$ when $n\ge 3$, and $d_{2}\coloneqq 2d\sigma/p$ when $n=2$, 
   where $\sigma\in(2,\,q/2)$ is arbitrarily fixed.
   Then, there exists a constant $C\in(1,\,\infty)$, which depends at most on $n$, $p$, $q$, $\lambda$, $\Lambda$, $K$, and $\sigma$, such that
   \begin{equation}\label{Eq (Section 2): Local Lipschitz bounds}
     \sup_{Q_{\theta R}}\,V_{\varepsilon} \le \frac{C}{(1-\theta)^{d_{2}}}\left(1+\lVert f_{\varepsilon}\rVert_{L^{q}(Q_{R})}^{d_{1}}+\fiint_{Q_{R}} V_{\varepsilon}^{p}\,{\mathrm d}X \right)^{\frac{d}{p}}
   \end{equation}
   for every fixed $Q_{R}=Q_{R}(x_{\ast},\,t_{\ast})\Subset \Omega_{T}$ with $R\in(0,\,1)$ and $\theta\in(0,\,1)$.
 \end{theorem}
 The key estimate in this paper is Theorem \ref{Thm: Key Hoelder Estimate}, where we do not necessarily require $p\in (p_{\mathrm c},\,\infty)$, but assume local $L^{\infty}$-bounds of $V_{\varepsilon}$, uniformly for sufficiently small $\varepsilon$.
 In the case (\ref{Eq (Section 1): p-q condition}), this assumption is already guaranteed by Theorem \ref{Proposition: Lipschitz}.
 \begin{theorem}\label{Thm: Key Hoelder Estimate}
  Let $p\in(1,\,\infty)$ and $q\in(n+2,\,\infty\rbrack$. 
  Fix $\delta\in(0,\,1)$, and let $\varepsilon\in(0,\,\delta/8)$.
   Assume that $u_{\varepsilon}$ is a weak solution to (\ref{Eq (Section 2): Approximation equation}) in $Q_{R}\Subset \Omega_{T}$ with $R\in(0,\,1)$ and (\ref{Eq (Section 2): Improved reg}).
   Also, let the positive numbers $F$ and $\mu_{0}$ satisfy (\ref{Eq (Section 2): Bound of f-epsilon}) and
   \begin{equation}
     \esssup_{Q_{R}} V_{\varepsilon}\le \delta+\mu_{0}
   \end{equation}
   respectively.
    Then, the estimates
    \begin{align}
      \label{Eq (Section 2): Main estimate Bounds}
     &\left\lvert {\mathcal G}_{2\delta,\,\varepsilon}(\nabla u_{\varepsilon}(X)) \right\rvert\le \mu_{0}\quad \textrm{for all}\quad X\in Q_{\rho_{0}}(x_{\ast},\,t_{\ast}),\\ 
     \label{Eq (Section 2): Main estimate Holder}
     &\left\lvert {\mathcal G}_{2\delta,\,\varepsilon}(\nabla u_{\varepsilon}(X_{1}))-{\mathcal G}_{2\delta,\,\varepsilon}(\nabla u_{\varepsilon}(X_{2})) \right\rvert\le C\left(\frac{d_{\mathrm{p}}(X_{1},X_{2})}{\rho_{0}}\right)^{\alpha}\mu_{0}
    \end{align}
   hold for all $X_{1}$, $X_{2}\in Q_{\rho_{0}/2}(x_{\ast},\,t_{\ast})$.
   Here the radius  $\rho_{0}\in(0,\,R/4\rbrack$,  the exponent $\alpha\in(0,\,1)$, and the constant $C\in(1,\,\infty)$ depend at most on $n$, $p$, $q$, $\lambda$, $\Lambda$, $K$, $F$, $\omega_{1}$, $\omega_{p}$, $\mu_{0}$, and $\delta$.
 \end{theorem}
 
 To prove Theorem \ref{Thm: Key Hoelder Estimate}, we always assume
 \begin{equation}\label{Eq (Section 2): Bounds of G-delta-epsilon-nabla u}
   \sup_{Q_{2\rho}}\,\left\lvert {\mathcal G}_{\delta,\,\varepsilon}(\nabla u_{\varepsilon})\right\rvert\le \mu\le \mu+\delta\le M,
 \end{equation}
 or equivalently
 \begin{equation}\label{Eq (Section 2): Bounds of V-epsilon}
   \sup_{Q_{2\rho}}\,V_{\varepsilon} \le \mu+\delta\le M,
 \end{equation}
 for $Q_{2\rho}=Q_{2\rho}(x_{0},\,t_{0})\subset Q_{R}$.
 Here $M\coloneqq \mu_{0}+\delta$ is a fixed constant, and $\mu\in(0,\,\mu_{0}\rbrack$ denotes a parameter that stands for a local bound of ${\mathcal G}_{\delta,\,\varepsilon}(\nabla u_{\varepsilon})$.
 We introduce a sufficiently small ratio $\nu\in(0,\,1/4)$, which is determined later in Section \ref{Sect:Non-Degenerate Region}, and define a superlevel set
 \[ S_{\rho,\,\mu,\,\nu}\coloneqq \{(x,\,t)\in Q_{\rho}\mid V_{\varepsilon}(x,\,t)-\delta>(1-\nu)\mu\}, \]
 which is often denoted by $S_{\rho}$ when $\mu$ and $\nu$ are clear. 
 We may often let
 \begin{equation}\label{Eq (Section 2): Delta vs Mu}
   0<\delta<\mu,
 \end{equation}
 since otherwise ${\mathcal G}_{2\delta,\,\varepsilon}(\nabla u_{\varepsilon})\equiv 0$ holds in $Q_{2\rho}$, and therefore oscillation estimates for ${\mathcal G}_{2\delta,\,\varepsilon}(\nabla u_{\varepsilon})$ become trivial.
 We introduce an exponent $\beta\in(0,\,1)$ by $\beta\coloneqq 1-(n+2)/q$ when $q\in(n+2,\,\infty)$.
 For $q=\infty$, we let $\beta$ be an arbitrarily fixed number in $(0,\,1)$.
 Dividing the possible cases by measuring the size of $S_{\rho,\,\mu,\,\nu}\subset Q_{\rho}$, we prove Propositions \ref{Proposition: Parabolic De Giorgi}--\ref{Proposition: Parabolic Campanato} as below.
 \begin{proposition}\label{Proposition: Parabolic De Giorgi}
  In addition to the assumptions of Theorem \ref{Thm: Key Hoelder Estimate}, let the positive numbers $\mu$, $M$, and a cylinder $Q_{2\rho}=Q_{2\rho}(x_{0},\,t_{0})\Subset \Omega_{T}$ satisfy (\ref{Eq (Section 2): Bounds of G-delta-epsilon-nabla u}) and (\ref{Eq (Section 2): Delta vs Mu}). 
   If 
   \begin{equation}\label{Eq (Section 2): Measure Assumption De Giorgi}
     \lvert S_{\rho,\,\mu,\,\nu}\rvert\le (1-\nu)\lvert Q_{\rho}\rvert
   \end{equation}
   holds for some $\nu\in(0,\,1/4)$, then there exist an exponent $\kappa\in \left((\sqrt{\nu}/6)^{\beta},\,1\right)$ and a radius ${\tilde \rho}\in(0,\,1)$, depending at most on $n$, $p$, $q$, $\lambda$, $\Lambda$, $K$, $F$, $M$, $\delta$, and $\nu$, such that 
   \begin{equation}\label{Eq (Section 2): De Giorgi Oscillation result}
     \sup_{Q_{\sqrt{\nu}\rho/3}}\,\left\lvert {\mathcal G}_{\delta,\,\varepsilon}(\nabla u_{\varepsilon}) \right\rvert\le \kappa\mu,
   \end{equation}
   holds, provided $\rho\le {\tilde\rho}$.
 \end{proposition}
  \begin{proposition}\label{Proposition: Parabolic Campanato}
  In addition to the assumptions of Theorem \ref{Thm: Key Hoelder Estimate}, let the positive numbers $\mu$, $M$, and a cylinder $Q_{2\rho}=Q_{2\rho}(x_{0},\,t_{0})\Subset Q_{R}$ satisfy (\ref{Eq (Section 2): Bounds of V-epsilon})--(\ref{Eq (Section 2): Delta vs Mu}).
   Then, there exist ${\hat \rho}\in(0,\,1)$ and $\nu\in(0,\,1/4)$, depending at most on $n$, $p$, $q$, $\lambda$, $\Lambda$, $K$, $F$, $M$, $\delta$, and $\omega=\omega_{\delta,\,M}$ given by Lemma \ref{Lemma: Continuity estimates on Hessian}, such that if there hold 
   \begin{equation}\label{Eq (Section 2): Measure Assumption Campanato}
     \lvert S_{\rho,\,\mu,\,\nu}\rvert>(1-\nu)\lvert Q_{\rho}\rvert
   \end{equation}
   and $\rho\le {\hat\rho}$, then the limit
    \[\Gamma_{2\delta,\,\varepsilon}(x_{0},\,t_{0})\coloneqq \lim\limits_{r\to 0}\left({\mathcal G}_{2\delta,\,\varepsilon}(\nabla u_{\varepsilon}) \right)_{Q_{r}(x_{0},\,t_{0})}\in{\mathbb R}^{n}  \]
   exists. Moreover, this limit satisfies 
   \[\fiint_{ Q_{r}(x_{0},\,t_{0})  }\left\lvert {\mathcal G}_{2\delta,\,\varepsilon}(\nabla u_{\varepsilon})- \Gamma_{2\delta,\,\varepsilon}(x_{0},\,t_{0})  \right\rvert^{2}\,{\mathrm d}X\le \left(\frac{r}{\rho} \right)^{2\beta}\mu^{2}\quad \textrm{for all }r\in(0,\,\rho\rbrack.\]
 \end{proposition}
 The desired exponent $\alpha$ in Theorem \ref{Thm: Key Hoelder Estimate} is given by 
 \begin{equation}\label{Eq (Section 2): Def of Holder exp}
   \alpha\coloneqq \frac{\log \kappa}{\log (\sqrt{\nu}/6)}\in(0,\,\beta),
 \end{equation}
 where $\nu$ and $\kappa=\kappa(\nu)$ are respectively determined by Propositions \ref{Proposition: Parabolic Campanato} and \ref{Proposition: Parabolic De Giorgi}.
 To see our proofs of Propositions \ref{Proposition: Parabolic De Giorgi}--\ref{Proposition: Parabolic Campanato}, given in Sections \ref{Sect:Degenerate Region}--\ref{Sect:Non-Degenerate Region}, we notice that $\nu\to 0$ and $\kappa\to 1$, as $\delta\to 0$.
 Therefore, the exponent $\alpha$ determined by (\ref{Eq (Section 2): Def of Holder exp}) will vanish as $\delta\to 0$.

 We make different analyses, depending on whether a gradient may vanish or not.
 This will be judged by measuring the ratio $\lvert S_{\rho,\,\mu,\,\nu}\rvert/\lvert Q_{\rho}\rvert$.
 It is worth mentioning that (\ref{Eq (Section 2): Measure Assumption Campanato}), the assumption in Proposition \ref{Proposition: Parabolic Campanato}, roughly states that the modulus $V_{\varepsilon}$ will stay close to its upper bound $\mu+\delta$.
 This indicates that the average integral of $\nabla u_{\varepsilon}$ will never vanish, which is rigorously verified by non-trivial energy estimates.
 Under the other condition (\ref{Eq (Section 2): Measure Assumption De Giorgi}) in Proposition \ref{Proposition: Parabolic De Giorgi}, we will probably face the case where a gradient may vanish.
 For this degenerate case, we should make suitable truncation to avoid analysis over facets.
 Such divisions by cases, based on the size of the superlevel set of the modulus, can be found in fundamental works in the $p$-Laplace regularity theory (e.g., \cite{MR1230384}, \cite{MR2635642}).
 In the proof of Propositions \ref{Proposition: Parabolic De Giorgi}--\ref{Proposition: Parabolic Campanato}, however, we fully appeal to (\ref{Eq (Section 2): Delta vs Mu}), and do not use the intrinsic scaling methods found in the $p$-Laplace regularity theory.

 Theorem \ref{Proposition: Lipschitz}, Propositions \ref{Proposition: Parabolic De Giorgi} and \ref{Proposition: Parabolic Campanato} are respectively shown in Sections \ref{Sect:Lipschitz Bounds}, \ref{Sect:Degenerate Region}, and \ref{Sect:Non-Degenerate Region}.
 We conclude Section \ref{Sect:Approximation} by giving the proofs of Theorem \ref{Thm: Key Hoelder Estimate} and Theorem \ref{Thm: Main theorem parabolic}.
 \begin{proof}[Proof of Theorem \ref{Thm: Key Hoelder Estimate}]
   We set $M\coloneqq \mu_{0}+\delta$, and choose $\nu\in(0,\,1/4)$ and ${\hat \rho}\in(0,\,1)$ as in Proposition \ref{Proposition: Parabolic Campanato}.
   Corresponding to this $\nu$, we set $\sigma\coloneqq \sqrt{\nu}/6$, and choose $\kappa\in(\sigma^{\beta},\,1)$ and $\tilde{\rho}\in(0,\,1)$ as in Proposition \ref{Proposition: Parabolic De Giorgi}.
   We set $\rho_{0}\coloneqq \min\{\,{\hat\rho},\,\tilde{\rho},\,R/4\,\}\in (0,\,R/4\rbrack$.   
   Then, (\ref{Eq (Section 2): Bounds of G-delta-epsilon-nabla u})--(\ref{Eq (Section 2): Bounds of V-epsilon}) hold with $(\rho,\,\mu)\coloneqq (\rho_{0},\,\mu_{0})$ and for every center point $(x_{0},\,t_{0})\in Q_{\rho_{0}}(x_{\ast},\,t_{\ast})$.
   
  We claim that the limit $\Gamma_{2\delta,\,\varepsilon}(x_{0},\,t_{0})\coloneqq \lim\limits_{r\to 0}\left({\mathcal G}_{2\delta,\,\varepsilon}(\nabla u_{\varepsilon}) \right)_{Q_{r}(x_{0},\,t_{0})}\in{\mathbb R}^{n}$ exists for each $(x_{0},\,t_{0})\in Q_{\rho_{0}}(x_{\ast},\,t_{\ast})$.
  Moreover, this limit satisfies
   \begin{equation}\label{Eq (Section 2): Alpha-Growth Campanato}
     \fiint_{Q_{r}(x_{0},\,t_{0})}\left\lvert {\mathcal G}_{2\delta,\,\varepsilon}(\nabla u_{\varepsilon})-\Gamma_{2\delta,\,\varepsilon}(x_{0},\,t_{0})\right\rvert^{2}\,{\mathrm d}X\le C\left(\frac{r}{\rho_{0}}\right)^{2\alpha}\mu_{0}^{2}
   \end{equation}
   for all $r\in(0,\,\rho_{0}\rbrack$.
   Here the exponent $\alpha\in(0,\,\beta)$ is defined by (\ref{Eq (Section 2): Def of Holder exp}), and $C\in(1,\,\infty)$ depends at most on $\nu$ and $\alpha$.
   To show (\ref{Eq (Section 2): Alpha-Growth Campanato}), we define $\mu_{k}\coloneqq \kappa^{k}\mu_{0}$, $\rho_{k}\coloneqq \sigma^{k}\rho_{0}$ for $k\in{\mathbb N}$.
   For these decreasing sequences $\{\mu_{k}\}_{k=0}^{\infty}$, $\{\rho_{k}\}_{k=0}^{\infty}$, we define
   \[{\mathcal N}\coloneqq \{k\in{\mathbb Z}_{\ge 0}\mid \mu_{k}>\delta,\textrm{ and }\lvert S_{\rho_{k},\,\mu_{k},\,\nu}(x_{0},\,t_{0}) \rvert\le (1-\nu)\lvert Q_{\rho_{k}}(x_{0},\,t_{0})\rvert\},\]
   The proper inclusion ${\mathcal N}\subsetneq {\mathbb Z}_{\ge 0}$ is clear, since $\mu_{k}\to 0$ as $k\to\infty$.
   Thus, we can define $k_{\star}\in{\mathbb Z}_{\ge 0}$ by the minimum number satisfying $k_{\star}\not\in {\mathcal N}$.
   By repeatedly applying Proposition \ref{Proposition: Parabolic De Giorgi} with $(\rho,\,\mu)=(\rho_{k},\,\mu_{k})$ for $k\in\{\,0,\,\dots\,,\,k_{\star}-1\,\}$, we have
   \begin{equation}\label{Eq (Section 2): De Giorgi Iteration}
    \sup_{Q_{2\rho_{k}}(x_{0},\,t_{0})}\,\lvert {\mathcal G}_{2\delta,\,\varepsilon}(\nabla u_{\varepsilon})\rvert\le\sup_{Q_{2\rho_{k}}(x_{0},\,t_{0})}\,\lvert {\mathcal G}_{\delta,\,\varepsilon}(\nabla u_{\varepsilon})\rvert\le \mu_{k}
   \end{equation}
   for every $k\in\{\,0,\,\dots\,,\,k_{\star}\,\}$.
   For $k_{\star}\not\in {\mathcal N}$, there are two possible cases.

   If $\mu_{k_{\star}}>\delta$, then $\lvert S_{\rho_{k_{\star}},\,\mu_{k_{\star}},\,\nu}\rvert>(1-\nu)\lvert Q_{\rho_{k_{\star}}}\rvert$ must hold.
   By Proposition \ref{Proposition: Parabolic Campanato}, the limit $\Gamma_{2\delta,\,\varepsilon}(x_{0},\,t_{0})\in{\mathbb R}^{n}$ exists. 
   Moreover, we have $\lvert\Gamma_{2\delta,\,\varepsilon}(x_{0},\,t_{0})\rvert\le \mu_{k_{\star}}$, and
   \[\fiint_{Q_{r}(x_{0},\,t_{0})}\left\lvert {\mathcal G}_{2\delta,\,\varepsilon}(\nabla u_{\varepsilon})-\Gamma_{2\delta,\,\varepsilon}(x_{0},\,t_{0}) \right\rvert^{2}\,{\mathrm d}X\le \left(\frac{r}{\rho_{k_{\star}}} \right)^{2\beta}\mu_{k_{\star}}^{2}    \le \left(\frac{r}{\rho_{0}}\right)^{2\alpha}\mu_{0}^{2}\]
   for all $r\in (0,\,\rho_{k_{\star}}\rbrack$, where we have used $\alpha\le \beta$ and $\kappa=\sigma^{\alpha}$.
   For $r\in(\rho_{k_{\star}},\,\rho_{0}\rbrack$, there corresponds a unique $k\in\{\,0,\,\dots\,,\,k_{\star}-1 \,\}$ such that $\rho_{k+1}<r\le\rho_{k}$.
   Then, it is obvious that $\sigma^{k}<\sigma^{-1}\cdot \frac{r}{\rho_{0}}=\frac{6}{\sqrt{\nu}}\cdot\frac{r}{\rho_{0}}$.
   Hence, we can compute
   \begin{align*}
    \fiint_{Q_{r}(x_{0},\,t_{0})}\left\lvert {\mathcal G}_{2\delta,\,\varepsilon}(\nabla u_{\varepsilon})-\Gamma_{2\delta,\,\varepsilon}(x_{0},\,t_{0}) \right\rvert^{2}\,{\mathrm d}X&\le 2\left(\lvert \Gamma_{2\delta,\,\varepsilon}(x_{0},\,t_{0}) \rvert^{2}+ \sup_{Q_{\rho_{k}}}\left\lvert {\mathcal G}_{2\delta,\,\varepsilon}(\nabla u_{\varepsilon}) \right\rvert^{2}\right) \\ 
    &\le 4\mu_{k}^{2}\le 4\left(\frac{6}{\sqrt{\nu}}\right)^{2\alpha}\left(\frac{r}{\rho_{0}}\right)^{2\alpha}\mu_{0}^{2}.
   \end{align*}
   This completes the proof of (\ref{Eq (Section 2): Alpha-Growth Campanato}) in the case $\mu_{k_{\star}}>\delta$.
   
   In the remaining case $0<\mu_{k_{\star}}\le \delta$, it is clear that ${\mathcal G}_{2\delta,\,\varepsilon}(\nabla u_{\varepsilon})\equiv 0$ in $Q_{\rho_{k_{\star}}}$.
   In particular, the identity $\Gamma_{2\delta,\,\varepsilon}(x_{0},\,t_{0})=0$ holds, and it is easy to check that (\ref{Eq (Section 2): De Giorgi Iteration}) is satisfied for all $k\in{\mathbb Z}_{\ge 0}$.
   For each $r\in(0,\,\rho_{0}\rbrack$, there corresponds a unique number $k\in{\mathbb Z}_{\ge 0}$ such that $\rho_{k+1}<r\le \rho_{k}$.
   By (\ref{Eq (Section 2): De Giorgi Iteration}) and $\sigma^{k}<\frac{6}{\sqrt{\nu}}\cdot\frac{r}{\rho_{0}}$, we have 
   \[\fiint_{Q_{r}(x_{0},\,t_{0})}\left\lvert {\mathcal G}_{2\delta,\,\varepsilon}(\nabla u_{\varepsilon})-\Gamma_{2\delta,\,\varepsilon}(x_{0},\,t_{0}) \right\rvert^{2}\,{\mathrm d}X\le \mu_{k}^{2}=\sigma^{2\alpha k}\mu_{0}^{2}\le \left(\frac{6}{\sqrt{\nu}} \right)^{2\alpha}\left(\frac{r}{\rho_{0}} \right)^{2\alpha}\mu_{0}^{2}.\]
   In any possible cases, we conclude (\ref{Eq (Section 2): Alpha-Growth Campanato}) with $C\coloneqq 4\cdot(6/\sqrt{\nu})^{2\alpha}$.

   The mapping $\Gamma_{2\delta,\,\varepsilon}$ is a Lebesgue representative of ${\mathcal G}_{2\delta,\,\varepsilon}(\nabla u_{\varepsilon})$ by Lebesgue's differentiation theorem.
   Therefore, (\ref{Eq (Section 2): Main estimate Bounds}) is obvious since $\lvert\Gamma_{2\delta,\,\varepsilon}(x_{0},\,t_{0})\rvert\le \mu_{0}$ for all $(x_{0},\,t_{0})\in Q_{\rho_{0}}(x_{\ast},\,t_{\ast})$.
   The proof of (\ref{Eq (Section 2): Main estimate Holder}) is completed by showing that $\Gamma_{2\delta,\,\varepsilon}$ satisfies
   \(\left\lvert \Gamma_{2\delta,\,\varepsilon}(X_{1})-\Gamma_{2\delta,\,\varepsilon}(X_{2}) \right\rvert\le C\left(\frac{d_{\mathrm{p}}(X_{1},\,X_{2})}{\rho_{0}}\right)^{\alpha}\mu_{0}\)
   for all $X_{1}=(x_{1},\,t_{1})$, $X_{2}=(x_{2},\,t_{2})\in Q_{\rho_{0}/2}(x_{\ast},\,t_{\ast})$.
   Firstly, we consider the case $r\coloneqq d_{\mathrm{p}}(X_{1},\,X_{2})\le 2\rho_{0}/3$. We fix a point
   \(X_{3}=(x_{3},\,t_{3})\coloneqq\left(\frac{x_{1}+x_{2}}{2},\,\min\{\,t_{1},\,t_{2}\,\}\right).\)
   The inclusions $Q_{r}(X_{3})\subset Q_{3r/2}(X_{j})\subset Q_{\rho_{0}}(X_{j})\subset Q_{2\rho_{0}}(X_{j})\subset Q_{4\rho_{0}}(X_{\ast})$ are easily checked for each $j\in\{\,1,\,2\,\}$.
   Applying (\ref{Eq (Section 2): Alpha-Growth Campanato}) with $Q_{r}(x_{0},\,t_{0})$ replaced by $Q_{3r/2}(X_{j})$ ($j\in\{\,1,\,2\,\}$), we obtain
   \[\left\lvert \Gamma_{2\delta,\,\varepsilon}(X_{1})-\Gamma_{2\delta,\,\varepsilon}(X_{2}) \right\rvert^{2}\le 2\sum_{j=1}^{2}\fiint_{Q_{r}(X_{3})}\left\lvert {\mathcal G}_{2\delta,\,\varepsilon}(\nabla u_{\varepsilon})-\Gamma_{2\delta,\,\varepsilon}(X_{j}) \right\rvert^{2}\,{\mathrm d}X\le C_{n,\,\nu}\left(\frac{r}{\rho_{0}}\right)^{2\alpha}\mu_{0}^{2},\]
   from which (\ref{Eq (Section 2): Main estimate Holder}) is clear.
   The remaining case $r>2\rho_{0}/3$ is easier, since this condition allows us to compute
   \(\left\lvert \Gamma_{2\delta,\,\varepsilon}(X_{1})-\Gamma_{2\delta,\,\varepsilon}(X_{2})\right\rvert\le 2\mu_{0}\le 2\cdot\left(\frac{3r}{2\rho_{0}} \right)^{\alpha}\mu_{0},\)
   which completes the proof of (\ref{Eq (Section 2): Main estimate Holder}).
  \end{proof}
 \begin{proof}[Proof of Theorem \ref{Thm: Main theorem parabolic}]
   Fix subdomains ${\mathcal Q}\Subset  \tilde{\mathcal Q}\Subset \Omega_{T-\tau}$ with $\tau\in(0,\,T)$, and let $\delta\in(0,\,1)$ and $\varepsilon_{k}\in(0,\,\delta/8)$.
   For each $\varepsilon=\varepsilon_{k}$, we consider $u_{\varepsilon_{k}}$, the unique weak solution of (\ref{Eq (Section 2): Approximated Dirichlet boundary}).
   Proposition \ref{Prop: Convergence result} allows us to assume that 
   \(\nabla u_{\varepsilon_{k}}\to \nabla u\) a.e.~in \(\Omega_{T-\tau}\),
   by taking a subsequence if necessary.
   Moreover, we can choose the positive constants $U_{\tau}$ and $F$, satisfying \(\lVert\nabla u_{\varepsilon_{k}} \rVert_{L^{p}(\Omega_{T-\tau})}\le U_{\tau}\) and \(\lVert f_{\varepsilon_{k}}\rVert_{L^{q}(\Omega_{T})}\le F\) respectively.
   By Theorem \ref{Proposition: Lipschitz}, we can find a constant $\mu_{0}\in(1,\,\infty)$, which depends at most on $n$, $p$, $q$, $\lambda$, $\Lambda$, $K$, $U_{\tau}$, $F$, and $\mathop{\mathrm{dist}}_{\mathrm{p}}(\tilde{\mathcal Q},\,\partial_{\mathrm{p}}\Omega_{T-\tau})$, such that $V_{\varepsilon}\le \delta+\mu_{0}$ holds a.e.~in $\tilde{\mathcal Q}$.
   Thanks to this uniform bound and Theorem \ref{Thm: Key Hoelder Estimate}, we can apply the Arzel\`{a}--Ascoli theorem to the sequence $\{{\mathcal G}_{2\delta,\,\varepsilon_{k}}(\nabla u_{\varepsilon_{k}}) \}_{k}\subset C^{0}({\mathcal Q};\,{\mathbb R}^{n})$.
   In particular, we may let ${\mathcal G}_{2\delta,\,\varepsilon_{k}}(\nabla u_{\varepsilon_{k}})$ uniformly converge to a H\"{o}lder continuous mapping $v_{2\delta}$ in ${\mathcal Q}$.
   Meanwhile, we have ${\mathcal G}_{2\delta,\,\varepsilon_{k}}(\nabla u_{\varepsilon_{k}})\to {\mathcal G}_{2\delta}(\nabla u)$ a.e.~in $\Omega_{T-\tau}$.
   These convergence results imply that ${\mathcal G}_{2\delta}(\nabla u)=v_{2\delta}$ a.e.~in ${\mathcal Q}\Subset \Omega_{T-\tau}$.
   Thus, we conclude ${\mathcal G}_{2\delta}(\nabla u)\in C^{0}(\Omega_{T};\,{\mathbb R}^{n})$ for each fixed $\delta\in(0,\,1)$.

   By the definition of ${\mathcal G}_{\delta}$, it is obvious that 
   \(\sup_{\Omega_{T}}\,\lvert{\mathcal G}_{\delta_{1}}(\nabla u)-{\mathcal G}_{\delta_{2}}(\nabla u) \rvert\le \lvert \delta_{1}-\delta_{2} \rvert\) for all \(\delta_{1},\,\delta_{2}\in(0,\,1).\)
   Therefore, ${\mathcal G}_{\delta}(\nabla u)\in C^{0}(\Omega_{T};\,{\mathbb R}^{n})$ uniformly converges as $\delta\to 0$.
   Since ${\mathcal G}_{\delta}(\nabla u)\to \nabla u$ a.e.~in $\Omega_{T}$ as $\delta \to 0$, this implies $\nabla u\in C^{0}(\Omega_{T};\,{\mathbb R}^{n})$.
 \end{proof}

\section{Basic estimates of approximate solutions}\label{Sect:Estimates from Weak Form}
 Section \ref{Sect:Estimates from Weak Form} aims to give weak formulations and local estimates of $u_{\varepsilon}$, which are used in Sections \ref{Sect:Degenerate Region}--\ref{Sect:Non-Degenerate Region}.
 \subsection{A basic weak formulation}\label{Subsect: Basic Weak Form}
 From (\ref{Eq (Section 3): Weak formulation differentiated}), we deduce a basic weak formulation concerning $V_{\varepsilon}=\sqrt{\varepsilon^{2}+\lvert \nabla u_{\varepsilon}\rvert^{2}}$ in a systematic approach. 
 \begin{lemma}\label{Lemma: Basic Weak Form}
   Let $u_{\varepsilon}$ be a weak solution to (\ref{Eq (Section 2): Approximation equation}) in $Q=Q_{R}(x_{\ast},\,t_{\ast})\Subset \Omega_{T}$ with $\varepsilon\in(0,\,1)$ and (\ref{Eq (Section 2): Improved reg}).
   Assume that $\psi$ is a non-decreasing, non-negative, and locally Lipschitz function that is differentiable except at finitely many points. 
   For this $\psi$, we define a convex function
   \begin{equation}\label{Eq (Section 3): Def of PSI}
     \Psi(s)\coloneqq \int_{0}^{s}\sigma\psi(\sigma)\,{\mathrm{d}}\sigma\quad \textrm{for}\quad s\in\lbrack 0,\,\infty).
   \end{equation}
   Let $\zeta\in  X_{0}^{2}(t_{\ast}-R^{2},\,t_{\ast};\,B_{R}(x_{\ast}))\cap L^{\infty}(\Omega_{T}) $ be a non-negative function that is compactly supported in the interior of $Q$. 
   Then, there holds
   \begin{equation}\label{Eq (Section 3): Weak form of V-epsilon}
   2J_{0}+2J_{1}+J_{2}+J_{3}\le n\lambda^{-1}J_{4}+2J_{5},
   \end{equation}
   where $J_{0},\,\dots\,,\,J_{5}$ are the integrals defined as
   \begin{align*}
    &J_{0} \coloneqq \displaystyle\int_{t_{\ast}-R^{2}}^{t_{\ast}} \langle \partial_{t}\left[\Psi(V_{\varepsilon})\right],\, \zeta\rangle\,{\mathrm d}t,\quad  J_{1} \coloneqq \displaystyle\iint_{Q}\mleft\langle \nabla^{2}E_{\varepsilon}(\nabla u_{\varepsilon})\nabla[\Psi(V_{\varepsilon})] \mathrel{}\middle|\mathrel{}\nabla\zeta \mright\rangle\,{\mathrm{d}}X,  \\ 
    &J_{2} \coloneqq \displaystyle\iint_{Q}\mleft\langle \nabla^{2}E_{\varepsilon}(\nabla u_{\varepsilon})\nabla V_{\varepsilon}\mathrel{} \middle| \mathrel{}\nabla V_{\varepsilon}\mright\rangle\zeta\psi^{\prime}(V_{\varepsilon})V_{\varepsilon}\,\mathrm{d}X,\\ 
    &J_{3} \coloneqq \displaystyle\sum_{j=1}^{n}\displaystyle\iint_{Q}\mleft\langle \nabla^{2}E_{\varepsilon}(\nabla u_{\varepsilon})\nabla\partial_{x_{j}}u_{\varepsilon} \mathrel{}\middle|\mathrel{}\nabla\partial_{x_{j}}u_{\varepsilon} \mright\rangle\zeta\psi(V_{\varepsilon})\,\mathrm{d}X,\\
    &J_{4} \coloneqq \displaystyle\iint_{Q}\lvert f_{\varepsilon}\rvert^{2}V_{\varepsilon}^{2-p}\left(\psi(V_{\varepsilon})+\psi^{\prime}(V_{\varepsilon})V_{\varepsilon} \right)\zeta\,{\mathrm{d}}X,\quad J_{5}\coloneqq \displaystyle\iint_{Q}\lvert f_{\varepsilon}\rvert\lvert \nabla\zeta\rvert\psi(V_{\varepsilon})V_{\varepsilon}\,{\mathrm{d}}X.
  \end{align*}
   
   \end{lemma}
   \begin{proof}
          We note  that the integrals $J_{2}$ and $J_{3}$  satisfies 
   \begin{equation}\label{Eq (Section 3): Ellipticity of J-2 and J-3}
   J_{2}\ge \lambda\iint_{Q}V_{\varepsilon}^{p-1}\lvert \nabla V_{\varepsilon}\rvert^{2}\zeta\psi^{\prime}(V_{\varepsilon})\,{\mathrm{d}}X\quad \textrm{and}\quad J_{3}\ge \lambda\iint_{Q}V_{\varepsilon}^{p-2}\mleft\lvert\nabla^{2} u_{\varepsilon}\mright\rvert^{2}\zeta\psi(V_{\varepsilon})\,{\mathrm{d}}X,
   \end{equation}
    which inequalities  follow from (\ref{Eq (Section 2): bounds of E-1-epsilon})--(\ref{Eq (Section 2): bounds of E-p-epsilon}) and $\psi,\,\psi^{\prime},\,\zeta\ge 0$. 
   For each $j\in\{\,1,\,\dots\,,\,n\,\}$, we test $\varphi\coloneqq \zeta\psi(V_{\varepsilon})\partial_{x_{j}}u_{\varepsilon}\in   X_{0}^{2}(t_{\ast}-R^{2},\,t_{\ast};\,B_{R}(x_{\ast})) $ into (\ref{Eq (Section 3): Weak formulation differentiated}). 
   Summing over $j\in\{\,1,\,\dots\,,\,n\,\}$ and computing similarly to  \cite[Lemma 3.5]{T-scalar},  we have
   \begin{align*}
   J_{0}+J_{1}+J_{2}+J_{3}&=-\iint_{Q}f_{\varepsilon}\left(\psi(V_{\varepsilon})\langle \nabla u_{\varepsilon}\mid \nabla\zeta\rangle+\zeta\psi^{\prime}(V_{\varepsilon})\langle\nabla u_{\varepsilon}\mid\nabla V_{\varepsilon}\rangle+\zeta\psi(V_{\varepsilon})\Delta u_{\varepsilon}\right)\,{\mathrm d}X  \\ 
   &\le J_{5}+\frac{1}{2}(J_{2}+J_{3})+\frac{n}{2\lambda}J_{4}\nonumber
   \end{align*}
   by Young's inequality and (\ref{Eq (Section 3): Ellipticity of J-2 and J-3}). From this, (\ref{Eq (Section 3): Weak form of V-epsilon}) immediately follows.
  \end{proof}
  The weak formulation (\ref{Eq (Section 3): Weak form of V-epsilon}) is viewed in two ways. 
  The first is that the composite function $\Psi(V_{\varepsilon})$ becomes a subsolution to a certain parabolic equation.
  This is easy to deduce by discarding the non-negative terms $J_{2}$ and $J_{3}$ in (\ref{Eq (Section 3): Weak form of V-epsilon}).
  In Section \ref{Subsect:De Giorgi Truncation}, we suitably choose $\psi$ to apply De Giorgi's truncation.
  The second is that we can obtain local $L^{2}$-estimates for the Hessian matrix $\nabla^{2}u_{\varepsilon}$ from the integral $J_{3}$.
  The detailed computations are given in Section \ref{Subsect:Hessian Energy}.
 \subsection{De Giorgi's truncation}\label{Subsect:De Giorgi Truncation}
 We prove that the convex composite function \[ U_{\delta,\,\varepsilon}\coloneqq (V_{\varepsilon}-\delta)_{+}^{2}=\lvert {\mathcal G}_{\delta,\,\varepsilon}(\nabla u_{\varepsilon})\rvert^{2}  \] is a weak subsolution to a uniformly parabolic equation.
 This fact implies that as in Lemma \ref{Lemma: De Giorgi truncation lemma} below, $U_{\delta,\,\varepsilon}$ belongs to a certain parabolic De Giorgi class (see \cite[Chapter II, \S 7]{MR0241822}, \cite[Chapter VI, \S 13]{MR1465184} as related items).
 \begin{lemma}\label{Lemma: De Giorgi truncation lemma}
   Let the assumptions of Proposition \ref{Proposition: Parabolic De Giorgi} be verified.
   Fix a subcylinder $Q_{0}\coloneqq B_{r}\times (T_{0},\,T_{1}\rbrack\subset Q_{2\rho}$.
   Then, for all $k\in(0,\,\infty)$, $\eta\in C_{\mathrm c}^{1}(B_{r};\,\lbrack 0,\,1\rbrack)$, $\phi_{\mathrm c}\in C^{1}([T_{0},\,T_{1}];\,\lbrack 0,\,1\rbrack)$ satisfying $\phi_{\mathrm c}(T_{0})=0$, we have 
   \begin{align}\label{Eq (Section 3): De Giorgi Truncation 1}
    &\esssup_{\tau\in (T_{0},\,T_{1})} \,\int_{B_{r}\times\{\tau\}}(U_{\delta,\,\varepsilon}-k)_{+}^{2}\eta^{2}\phi_{\mathrm c}\,{\mathrm d}x+\iint_{Q_{0}}\lvert \nabla (U_{\delta,\,\varepsilon}-k)_{+} \rvert^{2}\eta^{2}\phi_{\mathrm c}\,{\mathrm d}X\\ 
    & \le C\left[\iint_{Q_{0}} (U_{\delta,\,\varepsilon}-k)_{+}^{2}\left(\lvert\nabla\eta\rvert^{2}+\lvert\partial_{t}\phi_{\mathrm c}\rvert\right) \,{\mathrm d}X+\mu^{4}F^{2}\lvert A_{k}\rvert^{1-2/q} \right],\nonumber\\ 
    \label{Eq (Section 3): De Giorgi Truncation 2}
    &\int_{B_{r}\times\{\tau\}}  (U_{\delta,\,\varepsilon}-k)_{+}^{2}  \eta^{2}\,{\mathrm d}x -\int_{B_{r}\times\{T_{0}\}}   (U_{\delta,\,\varepsilon}-k)_{+}^{2}  \eta^{2}\,{\mathrm d}x\\ 
    &\le C\left[  \iint_{Q_{0}}  (U_{\delta,\,\varepsilon}-k)_{+}^{2}\lvert\nabla\eta\rvert^{2}\,{\mathrm d}X+\mu^{4}F^{2}\lvert A_{k}\rvert^{1-2/q}\right]\quad \text{for a.e.~} \tau\in (T_{0},\,T_{1}). \nonumber
   \end{align}
   Here $A_{k}\coloneqq \{(x,\,t)\in Q_{0}\mid U_{\delta,\,\varepsilon}(x,\,t)>k\}$, and the constant $C\in(1,\,\infty)$ depends at most on $n$, $p$, $q$, $\lambda$, $\Lambda$, $K$, $M$, and $\delta$.
  \end{lemma}
 \begin{proof}
   We apply Lemma \ref{Lemma: Basic Weak Form} with $\psi(\sigma)\coloneqq 2(1-\delta/\sigma)_{+}$, so that $\Psi(V_{\varepsilon})=U_{\delta,\,\varepsilon}$.
   Then, it is easy to compute $V_{\varepsilon}\psi(V_{\varepsilon})=2U_{\delta,\,\varepsilon}^{1/2}$, and
   \(V_{\varepsilon}^{2-p}\left(\psi(V_{\varepsilon})+V_{\varepsilon}\psi^{\prime}(V_{\varepsilon})\right)\le 2\delta^{1-p}V_{\varepsilon}\).
   Remarking that all the integrands in (\ref{Eq (Section 3): Weak form of V-epsilon}) vanish in the place $\{V_{\varepsilon}\le \delta\}$, we may replace $\nabla^{2} E_{\varepsilon}(\nabla u_{\varepsilon})$ by 
   a measurable matrix ${\mathcal A}_{\delta,\,\varepsilon}(\nabla u_{\varepsilon})$, which coincides with $\nabla^{2}E_{\varepsilon}(\nabla u_{\varepsilon})$ on $\{V_{\varepsilon}>\delta\}$, but ${\mathcal A}_{\delta,\,\varepsilon}(\nabla u_{\varepsilon})\equiv \mathrm{id}_{n}$ on $\{V_{\varepsilon}\le \delta\}$.
   This matrix satisfies
   \(
     \lambda_{\ast}\mathrm{id}_{n}\leqslant {\mathcal A}_{\delta,\,\varepsilon}(\nabla u_{\varepsilon})\leqslant \Lambda_{\ast}\mathrm{id}_{n}
   \)
   for some constants $\lambda_{\ast}\in(0,\,1)$, $\Lambda_{\ast}\in(1,\,\infty)$, depending at most on $\lambda$, $\Lambda$, $K$, $\delta$ and $M$.
   Discarding $J_{2}\ge 0$ and $J_{3}\ge 0$, we obtain
   \[\int_{T_{0}}^{T_{1}}\langle \partial_{t}U_{\delta,\,\varepsilon},\,\zeta\rangle\,{\mathrm d}t+\iint_{Q_{0}}\left\langle{\mathcal A}_{\delta,\,\varepsilon}(\nabla u_{\varepsilon})\nabla U_{\delta,\,\varepsilon} \mathrel{}\middle|\mathrel{}\nabla\zeta\right\rangle\,{\mathrm d}X\le C\mu\left(\iint_{Q_{0}}\left[\lvert f_{\varepsilon}\rvert^{2}\zeta+\lvert f_{\varepsilon}\rvert\lvert \nabla\zeta\rvert\right] \,{\mathrm d}X\right)\]  
   for any non-negative $\zeta\in X_{0}^{2}(T_{0},\,T_{1};\,B_{r})\cap L^{\infty}(Q_{0})$ that is compactly supported in the interior of $Q_{0}$.
   We test $\zeta\coloneqq (U_{\delta,\,\varepsilon}-k)_{+}\eta^{2}\phi$, where $\phi\colon \lbrack T_{0},\,T_{1}\rbrack\rightarrow \lbrack 0,\,1\rbrack$ is an arbitrary Lipschitz function satisfying $\phi(T_{0})=\phi(T_{1})=0$.
   Then, we have 
   \begin{align*}
    &-\frac{1}{2}\iint_{Q_{0}}(U_{\delta,\,\varepsilon}-k)_{+}^{2}\eta^{2}\partial_{t}\phi\,{\mathrm d}X+\lambda_{\ast}\iint_{Q_{0}}\lvert \nabla (U_{\delta,\,\varepsilon}-k)_{+}\rvert^{2}\eta^{2}\phi\,{\mathrm d}X  \\ 
    &\le \frac{\lambda_{\ast}}{2}\iint_{Q_{0}}\lvert \nabla (U_{\delta,\,\varepsilon}-k)_{+}\rvert^{2}\eta^{2}\phi \,{\mathrm d}X+C\left[\iint_{Q_{0}}(U_{\delta,\,\varepsilon}-k)_{+}^{2}\lvert\nabla\eta\rvert^{2}\phi \,{\mathrm d}X+\mu^{4}\iint_{A_{k}}\lvert f_{\varepsilon}\rvert^{2}\eta^{2}\phi\,{\mathrm d}X\right],
   \end{align*}
   where we have used Young's inequality and (\ref{Eq (Section 2): Delta vs Mu}).
   By H\"{o}lder's inequality, we obtain
   \begin{align*}
    &-\frac{1}{2}\iint_{Q_{0}}(U_{\delta,\,\varepsilon}-k)_{+}^{2}\eta^{2}\partial_{t}\phi\,{\mathrm d}X+\frac{\lambda_{\ast}}{2}\iint_{Q_{0}}\lvert \nabla (U_{\delta,\,\varepsilon}-k)_{+}\rvert^{2}\eta^{2}\phi\,{\mathrm d}X \nonumber \\  
    &\le C\left[\iint_{Q_{0}}(U_{\delta,\,\varepsilon}-k)_{+}^{2} \lvert \nabla\eta\rvert^{2}\phi\,{\mathrm d}X+\mu^{4}F^{2}\lvert A_{k}\rvert^{1-2/q}\right].
   \end{align*}
   The desired claims (\ref{Eq (Section 3): De Giorgi Truncation 1})--(\ref{Eq (Section 3): De Giorgi Truncation 2}) are easily deduced by suitably choosing $\phi$. 
  \end{proof}

 \subsection{$L^{2}$-energy estimates}\label{Subsect:Hessian Energy}
  Under the assumptions (\ref{Eq (Section 2): Delta vs Mu}) and $0<\nu<1/4$, we deduce energy estimates for $G_{p,\,\varepsilon}(\nabla u_{\varepsilon})=V_{\varepsilon}^{p-1}\nabla u_{\varepsilon}$, where $G_{p,\,\varepsilon}$ is given by (\ref{Eq (Section 2): Gpe}).

  
 \begin{lemma}\label{Lemma: Energy estimates from weak form}
  In addition to the assumptions of Theorem \ref{Thm: Key Hoelder Estimate}, let the positive numbers $\mu$, $M$, and a cylinder $Q_{2\rho}=Q_{2\rho}(x_{0},\,t_{0})\Subset \Omega_{T}$ satisfy (\ref{Eq (Section 2): Bound of f-epsilon}) and (\ref{Eq (Section 2): Bounds of V-epsilon})--(\ref{Eq (Section 2): Delta vs Mu}).
   Then, the following estimates hold for any $\sigma\in(0,\,1)$, $\nu\in(0,\,1/4)$.
   \begin{align}
    \fiint_{Q_{\sigma\rho}}\left\lvert \nabla\left[G_{p,\,\varepsilon}(\nabla u_{\varepsilon}) \right] \right\rvert^{2}\,{\mathrm d}X\le \frac{C\mu^{2p}}{\sigma^{n+2}\rho^{2}}\left[\frac{1}{(1-\sigma)^{2}}+ F^{2}\rho^{2\beta}\right].\label{Eq (Section 3): Hessian L-2 energy on Q}\\ 
    \lvert Q_{\sigma\rho}\rvert^{-1}  \iint_{S_{\sigma\rho,\,\mu,\,\nu}}  \left\lvert \nabla\left[G_{p,\,\varepsilon}(\nabla u_{\varepsilon}) \right] \right\rvert^{2}\,{\mathrm d}X\le \frac{C\mu^{2p}}{\sigma^{n+2}\rho^{2}}\left[\frac{\nu}{(1-\sigma)^{2}}+\frac{F^{2}\rho^{2\beta}}{\nu}\right].\label{Eq (Section 3): Hessian L-2 energy on S}
  \end{align}
   %
   %
   Here $C\in(0,\,\infty)$ in (\ref{Eq (Section 3): Hessian L-2 energy on Q})--(\ref{Eq (Section 3): Hessian L-2 energy on S}) depends at most on $n$, $p$, $\lambda$, $\Lambda$, $K$, $\delta$, and $M$.
 \end{lemma}
 \begin{proof}
   For each fixed $\sigma\in(0,\,1)$, we choose cut-off functions $\eta\in C_{\mathrm c}^{1}(B_{\rho};\,\lbrack 0,\,1\rbrack)$ and $\phi_{\mathrm c}\in C^{1}(\overline{I_{\rho}};\,\lbrack 0,\,1\rbrack)$ satisfying $\phi_{\mathrm c}(t_{0}-\rho^{2})=0$, $\eta|_{B_{\sigma\rho}}\equiv 1$, and $\phi_{\mathrm c}|_{I_{\sigma\rho}}\equiv 1$.
   For sufficiently small ${\tilde\varepsilon}>0$, which tends to $0$ later,
   we choose $\phi_{\mathrm h}\colon \overline{I_{\rho}}\to \lbrack 0,\,1\rbrack$ satisfying $\phi_{\mathrm h}\equiv 1$ on $\lbrack t_{0}-\rho^{2},\,t_{0}-{\tilde \varepsilon}\rbrack$, and linearly interpolated on the other domain with $\phi_{\mathrm h}(t_{0})=0$.
   By $\lvert \nabla V_{\varepsilon}\rvert^{2}\le \lvert \nabla^{2}u_{\varepsilon}\rvert^{2}$, it is easy to check that $\lvert \nabla [G_{p,\,\varepsilon}(\nabla u_{\varepsilon})] \rvert^{2}\le c_{p}V_{\varepsilon}^{2p-2}\lvert \nabla^{2}u_{\varepsilon}\rvert^{2}$ for some $c_{p}\in(1,\,\infty)$.
   With this in mind, we apply Lemma \ref{Lemma: Basic Weak Form} with $\psi(s)\coloneqq {\tilde \psi}(s) \sigma^{p}$, and $\zeta\coloneqq \eta^{2}\phi$ with $\phi\coloneqq \phi_{\mathrm c}\phi_{\mathrm h}$.
   Here ${\tilde \psi}\colon [0,\,\infty)\rightarrow \lbrack0,\,\infty)$ is a non-decreasing function that is chosen later. 
   The weak formulation (\ref{Eq (Section 3): Weak form of V-epsilon}) becomes
   \begin{align*}
      {\mathbf L}_{1}+{\mathbf L}_{2}+{\mathbf L}_{3}
     &\coloneqq -2\iint_{Q_{\rho}}\eta^{2}\Psi(V_{\varepsilon})\phi_{\mathrm c}\partial_{t}\phi_{\mathrm h}\,{\mathrm d}X+\iint_{Q_{\rho}}\left\langle \nabla^{2}E_{\varepsilon}(\nabla u_{\varepsilon})\nabla V_{\varepsilon}\mathrel{}\middle|\mathrel{}\nabla V_{\varepsilon} \right\rangle \eta^{2}\phi V_{\varepsilon}\psi^{\prime}(V_{\varepsilon})\,{\mathrm d}X \\ 
     &\quad +\sum_{j=1}^{n}\iint_{Q_{\rho}}\left\langle \nabla^{2}E_{\varepsilon}(\nabla u_{\varepsilon})\nabla\partial_{x_{j}} u_{\varepsilon}\mathrel{}\middle|\mathrel{}\nabla\partial_{x_{j}}u_{\varepsilon}  \right\rangle \eta^{2}\phi\psi(V_{\varepsilon})\,{\mathrm d}X\\ 
     &\le 2\iint_{Q_{\rho}}\Psi(V_{\varepsilon})\eta^{2}\phi_{\mathrm h}\partial_{t}\phi_{\mathrm c}\,{\mathrm d}X+4\iint_{Q_{\rho}}\left\lvert \left\langle \nabla^{2}E_{\varepsilon}(\nabla u_{\varepsilon})\nabla V_{\varepsilon} \mathrel{}\middle|\mathrel{}\nabla\eta \right\rangle\right\rvert\eta\phi\psi(V_{\varepsilon})V_{\varepsilon}\,{\mathrm d}X\\ 
     & \quad +\frac{n}{2\lambda}\iint_{Q_{\rho}}\lvert f_{\varepsilon}\rvert^{2}V_{\varepsilon}^{2-p}\left(\psi(V_{\varepsilon})+V_{\varepsilon}\psi^{\prime}(V_{\varepsilon}) \right)\eta^{2}\phi \,{\mathrm d}X+4\iint_{Q_{\rho}}\lvert f_{\varepsilon}\rvert\lvert\nabla\eta\rvert\psi(V_{\varepsilon})V_{\varepsilon}\eta\phi\,{\mathrm d}X\\ 
     &\eqqcolon 2{\mathbf R}_{1}+4{\mathbf R}_{2}+\frac{n}{2\lambda}{\mathbf R}_{3}+4{\mathbf R}_{4}.
    \end{align*} 
   We may discard ${\mathbf L}_{1}\ge 0$.
   By the Cauchy--Schwarz inequality and Young's inequality, ${\mathbf R}_{2}$ and ${\mathbf R}_{4}$ are estimated as follows:
   \[\begin{array}{rcl}4{\mathbf R}_{2}&\le& {\mathbf L}_{2}+4\displaystyle\iint_{Q_{\rho}}\left\langle \nabla^{2}E_{\varepsilon}(\nabla u_{\varepsilon})\nabla\eta\mathrel{}\middle|\mathrel{}\nabla\eta \right\rangle V_{\varepsilon}\phi\displaystyle\frac{\psi(V_{\varepsilon})^{2}}{\psi^{\prime}(V_{\varepsilon})} \,{\mathrm d}X,\\ 
    4{\mathbf R}_{4}&\le& 2{\mathbf R}_{3}+2\displaystyle\iint_{Q_{\rho}}V_{\varepsilon}^{p-1}\lvert \nabla\eta \rvert^{2}\phi\displaystyle\frac{\psi(V_{\varepsilon})^{2}}{\psi^{\prime}(V_{\varepsilon})}\,{\mathrm d}X.
   \end{array}\]
   Also, we note that our choice of $\psi$ yields
   \[\frac{\psi(V_{\varepsilon})^{2}}{\psi^{\prime}(V_{\varepsilon})}=\frac{V_{\varepsilon}^{p+1}{\tilde \psi}(V_{\varepsilon})^{2}}{p{\tilde \psi}(V_{\varepsilon})+V_{\varepsilon}{\tilde \psi}^{\prime}(V_{\varepsilon})},\quad 
   \text{and}
   \quad \psi(V_{\varepsilon})+V_{\varepsilon}\psi^{\prime}(V_{\varepsilon})=V_{\varepsilon}^{p}\left((p+1){\tilde \psi}(V_{\varepsilon})+V_{\varepsilon}{\tilde\psi}^{\prime}(V_{\varepsilon}) \right).\]
   Letting ${\tilde\varepsilon}\to 0$, recalling our choice of $\eta$ and $\phi_{\mathrm c}$, and using (\ref{Eq (Section 2): Bounds of V-epsilon})--(\ref{Eq (Section 2): Delta vs Mu}) and (\ref{Eq (Section 3): Ellipticity of J-2 and J-3}), we have
   \begin{align}\label{Eq (Section 3): Claim on L-2-energy bounds}
     \iint_{Q_{\sigma\rho}}\left\lvert \nabla \left[ G_{p,\,\varepsilon}(\nabla u_{\varepsilon}) \right] \right\rvert^{2}{\tilde\psi}(V_{\varepsilon})\,{\mathrm d}X & 
     \le \frac{C}{[(1-\sigma)\rho]^{2}}\iint_{Q_{\rho}}\left(\frac{\mu^{2p}{\tilde \psi}(V_{\varepsilon})^{2}}{p{\tilde \psi}(V_{\varepsilon})+V_{\varepsilon}{\tilde\psi}^{\prime}(V_{\varepsilon})}+\Psi(V_{\varepsilon})\right) \,{\mathrm d}X\\ &\quad +C\mu^{2p}\iint_{Q_{\rho}}\lvert f_{\varepsilon}\rvert^{2}\left({\tilde \psi}(V_{\varepsilon})+V_{\varepsilon}{\tilde\psi}^{\prime}(V_{\varepsilon}) \right)\,{\mathrm d}X.\nonumber
    \end{align}
   When ${\tilde \psi}(s)\equiv p+2$, which yields $\Psi(s)=s^{p+2}$, we simply use (\ref{Eq (Section 2): Delta vs Mu}) to get $\Psi(V_{\varepsilon})\le C_{p,\,\delta,\,M}\mu^{2p}$.
   By (\ref{Eq (Section 3): Claim on L-2-energy bounds}) and H\"{o}lder's inequality, we have
   \[\iint_{Q_{\sigma\rho}}\left\lvert \nabla \left[ G_{p,\,\varepsilon}(\nabla u_{\varepsilon}) \right] \right\rvert^{2}\,{\mathrm d}X\le C\mu^{2p}\left[\frac{\lvert Q_{\rho}\rvert}{[(1-\sigma)\rho]^{2}}+F^{2}\lvert Q_{\rho}\rvert^{1-2/q}\right].\]
   Dividing each side by $\lvert Q_{\sigma\rho}\rvert=\sigma^{n+2}\lvert Q_{\rho}\rvert$, we obtain (\ref{Eq (Section 3): Hessian L-2 energy on Q}).
   When ${\tilde \psi}(s)=(s-\delta-k)_{+}^{2}$ with $k\coloneqq (1-2\nu)\mu>\mu/2$, we note that for every $s\in(0,\,\mu+\delta)$
   \[
     \Psi(s)\le (\mu+\delta)^{p+1}\int_{0}^{\mu+\delta}(\sigma-\delta-k)_{+}^{2}\,{\mathrm d}\sigma \le \frac{(\mu-k)^{3}(\mu+\delta)^{p+1}}{3}, 
   \]
   which yields $\Psi(V_{\varepsilon})\le C_{p,\,\delta,\,M}\nu^{3}\mu^{2p+2}$ by  (\ref{Eq (Section 2): Bounds of V-epsilon})--(\ref{Eq (Section 2): Delta vs Mu}). 
   We also note that $k\in(\mu/2,\,\mu)$ yields
   \begin{align*}
    &\frac{{\tilde \psi}(V_{\varepsilon})^{2}}{p{\tilde \psi}(V_{\varepsilon})+V_{\varepsilon}{\tilde\psi}^{\prime}(V_{\varepsilon})}=\frac{(V_{\varepsilon}-\delta-k)_{+}^{3}}{p(V_{\varepsilon}-\delta-k)_{+}+2V_{\varepsilon}}\le \frac{(2\nu\mu)^{3}}{2(\delta+k)}\le 8\nu^{3}\mu^{2},\\ 
    &{\tilde{\psi}}(V_{\varepsilon})+V_{\varepsilon}{\tilde \psi}^{\prime}(V_{\varepsilon})=(V_{\varepsilon}-\delta-k)_{+}(3V_{\varepsilon}-\delta-k)\le 2\nu\mu\cdot 2(\delta+\mu+\nu\mu).
   \end{align*}
   Combining these inequalities with (\ref{Eq (Section 2): Delta vs Mu}), and \({\tilde\psi}(V_{\varepsilon})\ge (\nu\mu)^{2}\) on $S_{\sigma\rho,\,\mu,\,\nu}\subset Q_{\sigma\rho}$, we get 
   \[(\nu\mu)^{2}\iint_{S_{\sigma\rho,\,\mu,\,\nu}}\left\lvert \nabla\left[G_{p,\,\varepsilon}(\nabla u_{\varepsilon})\right]\right\rvert^{2}\,  {\mathrm d}X\le C\mu^{2p+2}\left[\frac{\nu^{3}\lvert Q_{\rho}\rvert}{[(1-\sigma)\rho]^{2}}+\nu F^{2}\lvert Q_{\rho}\rvert^{1-2/q}\right],\]
   from (\ref{Eq (Section 3): Claim on L-2-energy bounds}).   This completes the proof of (\ref{Eq (Section 3): Hessian L-2 energy on S}).
  \end{proof}
\section{Local gradient bounds}\label{Sect:Lipschitz Bounds} In Section \ref{Sect:Lipschitz Bounds}, we prove local uniform bounds of $\nabla u_{\varepsilon}$.
 \subsection{Preliminary}\label{Subsect: Wk}
 For $p$-harmonic flows with $p_{\mathrm c}<p<\infty$, local gradient bounds are proved in \cite{MR743967} by De Giorgi's truncation (see also \cite[Chapter VIII]{MR1230384}).
 Another proof that is fully based on Moser's iteration is given by \cite[\S 4]{BDLS}. 
 However, the method in \cite[\S 4]{BDLS} will not work in our problem, since the test functions therein may intersect with a facet of a solution.
 Here we note that our problem can be seen as uniformly parabolic when a spatial gradient does not vanish.
 With this in mind, we carefully choose test functions supported in the place 
 \(D_{k}\coloneqq \{(x,\,t)\in Q_{R}\mid \lvert \nabla u_{\varepsilon}(x,\,t)\rvert>k\}\)
 for some constant $k\ge 1$.
 This type of strategy can be found in various elliptic problems whose uniform ellipticity may break on some degenerate regions (see e.g.,  \cite{MR4078712}, \cite{MR2733227}, \cite{MR1925022}, \cite{MR4201656}).
 Among these papers, our proof, based on Moser's iteration, is substantially a modification of \cite{MR4201656}. 
  It is worth mentioning that local gradient bounds for a $p$-Poisson flow are discussed in \cite{MR2916967}, provided $f$ is in the Lorentz space $L(n+2,\,1)$, which is properly larger than $L^{q}$ with $q\in(n+2,\,\infty)$. 
 
 To clarify our strategy, we introduce another truncation of a partial derivative $\partial_{x_{j}}u_{\varepsilon}$. 
 The truncation function $g_{k}$ is given by $g_{k}(\sigma)\coloneqq (\sigma-k)_{+}-(-\sigma-k)_{+}$ for $\sigma\in{\mathbb R}$,
 and the corresponding truncated partial derivative is of the form
 \(v_{j,\,k}\coloneqq g_{k}(\partial_{x_{j}}u_{\varepsilon})\) for each \(j\in\{\,1,\,\dots\,,\,n\,\}\),
 where the constant $k\ge 1$ will be determined by $\lVert f_{\varepsilon}\rVert_{L^{q}(Q_{R})}$.
 We do not directly show reversed H\"{o}lder estimates for $V_{\varepsilon}$.
 Instead, we consider another scalar function 
  $W_{k}\coloneqq \sqrt{k^{2}+\sum_{j=1}^{n}v_{j,\,k}^{2}}$.
 In Section \ref{Sect:Lipschitz Bounds}, we mainly aim to deduce local $L^{\infty}$-bounds of $W_{k}$.
 Here we note
 \begin{equation}\label{Eq (Section 4): Compatibility}
  V_{\varepsilon} \le c_{n}W_{k} \quad \textrm{in} \quad Q=Q_{R},\quad \text{and}\quad W_{k}\le \sqrt{2} V_{\varepsilon} \quad \textrm{in}\quad D_{k}\subset Q_{R}.
\end{equation}
since we have let $k\ge 1$ (see \cite[\S 4.1]{MR4201656}).
In particular, we are allowed to use
\begin{equation}\label{Eq (Section 4): Uniform elliptic structure on D-k}
 {\tilde \lambda}W_{k}^{p-2}\mathrm{id}_{n} \leqslant \nabla^{2}E_{\varepsilon}(\nabla u_{\varepsilon}) \leqslant {\tilde \Lambda}W_{k}^{p-2}\mathrm{id}_{n} \quad \textrm{in } D_{k},
\end{equation}
where the constants $0<\tilde{\lambda}<\tilde{\Lambda}<\infty$ depend at most on $n$, $p$, $\lambda$, $\Lambda$, $K$.
 With this in mind, we can deduce a weak formulation of $W_{k}$, similarly to Lemma \ref{Lemma: Basic Weak Form}.
 We consider the composite functions $\psi$ and $\Psi$ given in Lemma \ref{Lemma: Basic Weak Form}, and fix an arbitrary non-negative Lipschitz function $\zeta=\zeta(x,\,t)$ that is compactly supported in the interior of $Q$.
 Then, $W_{k}$ satisfies 
 \begin{equation}\label{Eq (Section 4): Weak Form for Wk} 
  2{\tilde J}_{0}+2{\tilde J}_{1}+{\tilde J}_{2}+{\tilde J}_{3}\le {\tilde \lambda}^{-1}n{\tilde J}_{4}+2J_{5},
 \end{equation}
 where for each $l\in\{\,0,\,\dots\,,\,6\,\}$, ${\tilde J}_{l}$ is similarly defined as in Lemma \ref{Lemma: Basic Weak Form} with $V_{\varepsilon}$ and $\nabla\partial_{x_{j}}u_{\varepsilon}$ replaced by $W_{k}$ and $\nabla v_{j,\,k}$ respectively.
 The inequality (\ref{Eq (Section 4): Weak Form for Wk}) is deduced by testing $\phi\coloneqq \zeta \psi(W_{k})v_{j,\,k}$ into (\ref{Eq (Section 3): Weak formulation differentiated}).
 Here it should be mentioned that this test function is supported in $Q\cap \{\lvert \partial_{x_{j}}u_{\varepsilon}\rvert>k\}\subset D_{k}$.
 Therefore we may replace $\nabla\partial_{x_{j}}u_{\varepsilon}$ and $\partial_{t}\partial_{x_{j}}u_{\varepsilon}$ by $\nabla v_{j,\,k}$ and $\partial_{t}v_{j,\,k}$ respectively, and make use of (\ref{Eq (Section 4): Uniform elliptic structure on D-k}).
 Summing over $j\in\{\,1,\,\dots\,,\,n\,\}$, we can obtain (\ref{Eq (Section 4): Weak Form for Wk}), similarly to (\ref{Eq (Section 3): Weak form of V-epsilon}).
 \begin{lemma}\label{Lemma: W-k energy estimate}
   Let $u_{\varepsilon}$ be a weak solution to (\ref{Eq (Section 2): Approximation equation}) in $Q_{R}=Q_{R}(x_{\ast},\,t_{\ast})\Subset\Omega_{T}$ with $\varepsilon\in(0,\,1)$ and (\ref{Eq (Section 2): Improved reg}).
   Fix $\eta\in C_{\mathrm c}^{1}(B_{R};\,\lbrack 0,\, 1\rbrack)$, $\phi_{\mathrm c}\in C^{1}(\overline{I_{R}};\,\lbrack 0,\,1\rbrack)$ satisfying $\phi_{\mathrm c}(t_{0}-R^{2})=0$.
   Then, there exists a constant $C\in(0,\,\infty)$, depending at most on $n$, $p$, $\lambda$, $\Lambda$ and $K$, such that
   \begin{align}\label{Eq (Section 4): Energy Estimate for Wk}
    &  \esssup_{\tau\in \ring{I_{R}}} \,\int_{B_{R}\times \{\tau\}} \left(\eta W_{k}^{(1+\alpha)}\right)^{2}  \phi_{\mathrm c}\,{\mathrm d}x+\iint_{Q_{R}}  \left\lvert \nabla \left(\eta W_{k}^{p/2+\alpha}\right) \right\rvert^{2} \phi_{\mathrm c} \,{\mathrm d}X \nonumber \\
    & \le C(1+\alpha)^{2} \left[\iint_{Q_{R}}\left(W_{k}^{p+2\alpha}  \left(\eta^{2}+\lvert \nabla\eta\rvert^{2}\right) +W_{k}^{2(1+\alpha)}\lvert \partial_{t}\phi_{\mathrm c}\rvert\right)\,{\mathrm d}X+\iint_{Q_{R}}\lvert f_{\varepsilon}\rvert^{2}W_{k}^{2+2\alpha-p} \eta^{2}\phi_{\mathrm c}\,{\mathrm d}X\right]
   \end{align}
   for all $\alpha\in\lbrack 0,\,\infty)$.
 \end{lemma}
 \begin{proof}
  Let $\phi_{\mathrm h}\colon I_{r}\to \lbrack 0,\,1\rbrack$ be an arbitrary Lipschitz function that is non-increasing and $\phi_{\mathrm h}(t_{0})=0$, and we set $\phi\coloneqq \phi_{\mathrm c}\phi_{\mathrm h}$.
  We test $\zeta\coloneqq \eta^{2}\phi$ into (\ref{Eq (Section 4): Weak Form for Wk}) with \(\psi(s)\coloneqq s^{2\alpha}\), and therefore \(\Psi(s)=[2(1+\alpha)]^{-1}s^{2(1+\alpha)}\).
  Then, we note that (\ref{Eq (Section 4): Uniform elliptic structure on D-k}) yields
  \[
    {\tilde J}_{2}+{\tilde J}_{3}
    \ge (1+2\alpha){\tilde \lambda}\iint_{Q_{R}}\lvert\nabla W_{k}\rvert^{2}W_{k}^{p+2\alpha-2}\eta^{2}\phi\,{\mathrm d}X=\frac{4(1+2\alpha)}{(p+2\alpha)^{2}} {\tilde \lambda}\iint_{Q_{R}}\left\lvert \nabla W_{k}^{p/2+\alpha} \right\rvert^{2}\eta^{2}\phi\,{\mathrm d}X,
  \]
  where we have used
  \(\lvert\nabla W_{k}\rvert^{2}\le \sum_{j=1}^{n} \lvert\nabla v_{j,\,k}\rvert^{2}.\)
  By (\ref{Eq (Section 4): Uniform elliptic structure on D-k}) and Young's inequality, we have
  \begin{align*}
    2\lvert {\tilde J}_{1}\rvert+n{\tilde\lambda}^{-1}\lvert{\tilde J}_{4}\rvert+2\lvert {\tilde J}_{5}\rvert
    &\le\frac{(1+2\alpha){\tilde \lambda}}{2}\iint_{Q_{R}}\lvert\nabla W_{k}\rvert^{2}W_{k}^{p+2\alpha-2}\eta^{2}\phi\,{\mathrm d}X+C\iint_{Q_{R}}W_{k}^{p+2\alpha}\lvert\nabla\eta\rvert^{2}\phi\,{\mathrm d}X\\
    &\quad +C(1+2\alpha)\iint_{Q_{R}}\lvert f_{\varepsilon}\rvert^{2}W_{k}^{2+2\alpha-p}\eta^{2}\phi\,{\mathrm d}X.
  \end{align*}
  Hence, we obtain
  \begin{align*}
    &(1+\alpha)^{-1}\left[-\iint_{Q_{R}}W_{k}^{2+2\alpha}\eta^{2}\phi_{\mathrm c}\partial_{t}\phi_{\mathrm h}\,\mathrm{d}X+\iint_{Q_{R}}\left\lvert \nabla W_{k}^{p/2+\alpha} \right\rvert^{2}\eta^{2}\phi_{\mathrm c}\phi_{\mathrm h}\,{\mathrm d}X\right] \\ 
    &\le C\left[\iint_{Q_{R}}\left(W_{k}^{p+2\alpha}\lvert \nabla\eta\rvert^{2}+ W_{k}^{2(1+\alpha)}\lvert \partial_{t}\phi_{\mathrm c}\rvert\right) \,{\mathrm d}X+(1+\alpha)\iint_{Q_{R}}\lvert f_{\varepsilon}\rvert^{2}W_{k}^{2+2\alpha-p}\eta^{2}\phi_{\mathrm c}\,{\mathrm d}X\right],\nonumber
  \end{align*}
  where $C\in(0,\,\infty)$ depends at most on $p$, $\tilde{\lambda}$, and $\tilde{\Lambda}$.
  From this estimate, we deduce (\ref{Eq (Section 4): Energy Estimate for Wk}) by suitably choosing  $\phi_{\mathrm h}$.
 \end{proof} 
 
 We infer an iteration lemma, often used implicitly in Moser's iteration.
 \begin{lemma}\label{Lemma: Geometric Convergence on Seq}
   Fix the constants $A,\,B,\,\kappa\in(1,\,\infty)$ and $\mu\in(0,\,\infty)$.
   Let the sequences $\{Y_{l}\}_{l=0}^{\infty}\subset \lbrack 0,\,\infty)$ and $\{p_{l}\}_{l=0}^{\infty}\subset (0,\,\infty)$ satisfy
   \(Y_{l+1}^{p_{l+1}}\le \left(AB^{l}Y_{l}^{p_{l}}\right)^{\kappa}\), and \(p_{l}\ge \mu\left(\kappa^{l}-1\right)\)
   for all $l\in{\mathbb Z}_{\ge 0}$, and $\kappa^{l}p_{l}^{-1}\to \mu^{-1}$ as $l\to \infty$. Then, we have 
   \(
     \limsup\limits_{l\to\infty} Y_{l}\le A^{\frac{\kappa^{\prime}}{\mu}}B^{\frac{(\kappa^{\prime})^{2}}{\mu}}Y_{0}^{\frac{p_{0}}{\mu}}.
   \)
 \end{lemma}
 \begin{proof}
  By induction, it is easy to check that
  \(Y_{l}^{p_{l}}\le \prod_{k=1}^{l}A^{\kappa^{l-k+1}} \prod_{k=1}^{l}B^{k\kappa^{l-k+1}} Y_{0}^{\kappa^{l}p_{0}}\)
  holds for all $l\in{\mathbb N}$. 
  By \cite[Lemma 2.3]{MR4389309} and the assumptions in Lemma \ref{Lemma: Geometric Convergence on Seq}, we obtain
  \[Y_{l}\le \prod_{k=1}^{l}A^{\frac{\kappa^{l-k+1}}{\mu\left(\kappa^{l}-1\right)}}\prod_{k=1}^{l}B^{\frac{k\kappa^{l-k+1}}{\mu(\kappa^{l}-1)}} Y_{0}^{\frac{\kappa^{l}p_{0}}{p_{l}}} \le A^{\frac{\kappa^{\prime}}{\mu}} B^{\frac{(\kappa^{\prime})^{2}}{\mu}}Y_{0}^{\frac{\kappa^{l}p_{0}}{p_{l}}}.\]
  Letting $l\to\infty$ completes the proof.  
  \end{proof}
   We also recall a well-known lemma without a proof (see  \cite[Chapter V, Lemma 3.1]{MR717034} for the proof).
  \begin{lemma}\label{Lemma (Section 4): ABSORB LEMMA}
  Fix $R_{1},\,R_{2}\in{\mathbb R}_{>0}$ with $R_{1}<R_{2}$.
Assume that a bounded function $F\colon \lbrack R_{1},\,R_{2} \rbrack\to{\mathbb R}_{\ge 0}$ admits the constants $\alpha\in {\mathbb R}_{>0}$, $A,\,B\in{\mathbb R}_{\ge 0}$, and $\theta\in (0,\,1)$, such that there holds
\[F(r_{1})\le \theta F(r_{2})+\frac{A}{(r_{2}-r_{1})^{\alpha}}+B\]
for any $r_{1},\,r_{2}\in\lbrack R_{1},\,R_{2}\rbrack$ with $r_{1}<r_{2}$.
Then, $F$ satisfies
\[F(R_{1})\le C(\alpha,\,\theta)\left[\frac{A}{(R_{2}-R_{1})^{\alpha}}+B \right].\]
  \end{lemma}
  
  \subsection{Moser's iteration}
  We give the proof of Theorem \ref{Proposition: Lipschitz} by Moser's iteration.
  \begin{proof}[Proof of Theorem \ref{Proposition: Lipschitz}]
    We define $\kappa \in(1,\,2)$ to be $\kappa\coloneqq 1+2/n$ for $n\ge 3$, and $\kappa\coloneqq \sigma^{\prime}$ for $n\ge 2$. 
    We also set ${\tilde\kappa}\coloneqq 1+2/n-\kappa \in\lbrack 0,\,1)$.
   By our choice of $\sigma$ and $q$, we have $q/(q-2)<\kappa$.
   We define $\pi\coloneqq \max\{\,p-1,\,p/2\,\}>0$, and $k\coloneqq 1+\lVert f_{\varepsilon}\rVert_{L^{q}(Q_{R})}^{1/\pi}\ge 1$. 
   Then, the non-negative function ${\tilde f}_{\varepsilon}\coloneqq W_{k}^{-2\pi}\lvert f_{\varepsilon}\rvert^{2}\in L^{q/2}(Q_{R})$ satisfies  $\lVert {\tilde f}_{\varepsilon} \rVert_{L^{q/2}(Q_{R})}\le k^{-2\pi}\lVert f_{\varepsilon} \rVert_{L^{q}(Q_{R})}^{2} \le 1$. 
   The proof is completed by showing 
   \begin{equation}\label{Eq (Section 4): Claim on W-k}
     \esssup_{Q_{\theta R}} \,W_{k}\le \frac{C}{(1-\theta)^{d_{2}}}\left[\fiint_{Q_{R}}W_{k}^{p}\,{\mathrm d}X \right]^{d/p}
   \end{equation}
   for some $C=C(n,\,p,\,q,\,\lambda,\,\Lambda,\,K)\in(1,\,\infty)$.
   Indeed, (\ref{Eq (Section 4): Compatibility}), (\ref{Eq (Section 4): Claim on W-k}) and $W_{k}\le k+V_{\varepsilon}$ yield
   \[\esssup_{Q_{\theta R}}\,V_{\varepsilon}\le C_{n}\esssup_{Q_{\theta R}}\,W_{k}\le \frac{C}{(1-\theta)^{d_{2}}}\left(k^{p}+\fiint_{Q_{R}}V_{\varepsilon}^{p}\,{\mathrm d}X \right)^{d/p}.\]
   Recalling the definition of $k$, we conclude (\ref{Eq (Section 2): Local Lipschitz bounds}).

   To prove (\ref{Eq (Section 4): Claim on W-k}), we firstly deduce the reversed H\"{o}lder inequality 
   \begin{equation}\label{Eq (Section 4): REV H}
   \iint_{Q_{R_{1}}} W_{k}^{2\kappa \alpha+2(\kappa-1)+p}\,{\mathrm d}X \le \left[\frac{C(n,\,p,\,q,\,\lambda,\,\Lambda,\,K)}{(R_{2}-R_{1})^{2}}(1+\alpha)^{\gamma} \iint_{Q_{R_{2}}}\left(W_{k}^{p+2\alpha}+W_{k}^{2+2\alpha}\right) \,{\mathrm d}X \right]^{\kappa},
   \end{equation}
   which hold for any $\alpha\in\lbrack 0,\,\infty)$ and $0<R_{1}<R_{2}\le R$.
   Here $\gamma=\gamma(n,\,q)\in\lbrack 2,\,\infty)$ is a constant.
   To prove (\ref{Eq (Section 4): REV H}), we let the functions $\eta=\eta(x)$ and $\phi_{\mathrm c}=\phi_{\mathrm c}(t)$ satisfy all the assumptions in Lemma \ref{Lemma: W-k energy estimate}.
    We fix $\alpha\ge 0$, and introduce the non-negative functions \(\varphi_{1}\coloneqq \eta W_{k}^{1+\alpha} \phi_{\mathrm c}^{1/2}\), \(\varphi_{2}\coloneqq \eta W_{k}^{p/2+\alpha}\phi_{\mathrm c}^{1/2}\), and \(\varphi_{3}\coloneqq \min\{\,\varphi_{1},\,\varphi_{2} \,\}\), where we note $\varphi_{3}=\varphi_{1}$ when $p\ge2$ and $\varphi_{3}=\varphi_{2}$ when $p_{\mathrm c}<p\le 2$. By the definition of $\pi$ and ${\tilde f}_{\varepsilon}$, we rewrite (\ref{Eq (Section 4): Energy Estimate for Wk}) as 
   \begin{align}\label{Eq (Section 4): Mid ENERGY}
   & \esssup_{\tau\in \ring{I_{R}}}\,\int_{B_{R}\times \{\tau\}}\varphi_{1}^{2}\,{\mathrm d}x+\iint_{Q_{R}}\lvert \nabla\varphi_{2} \rvert^{2}\,{\mathrm d}X\nonumber \\
   & \le C(1+\alpha)^{2}\left[\iint_{Q_{R}}\left(W_{k}^{p+2\alpha} \left(\eta^{2}+\lvert\nabla\eta\rvert^{2}\right)+W_{k}^{2+2\alpha}\lvert\partial_{t}\phi_{\mathrm c}\rvert \right) \,{\mathrm d}X\right] +C(1+\alpha)^{2}\iint_{Q_{r}}{\tilde f}_{\varepsilon}\varphi_{3}^{2}\,{\mathrm d}X.
   \end{align} 
   From (\ref{Eq (Section 4): Mid ENERGY}), we can find a constant $C\in(1,\,\infty)$, depending at most on $n$, $\kappa$, $q$, $\lambda$, $\Lambda$, and $K$, such that
   \begin{equation}\label{Eq (Section 4): Reversed Hoelder}
     \iint_{Q_{R}}\varphi_{1}^{2(\kappa-1)}\varphi_{2}^{2} \,  {\mathrm d}X  \le CR^{2{\tilde \kappa}}(1+\alpha)^{\kappa\gamma}\left(\iint_{Q_{R}}\left(W_{k}^{p+2\alpha}(\eta^{2}+\lvert \nabla \eta\rvert^{2})+W_{k}^{2+2\alpha}\lvert \partial_{t}\phi\rvert\right)\,{\mathrm d}X\right)^{\kappa}
   \end{equation}
   for some constant 
   $\gamma=  \gamma(n,\,q)  \in\lbrack 2,\,\infty)$.
   To obtain (\ref{Eq (Section 4): Reversed Hoelder}), we use the parabolic Poincar\'{e}--Sobolev embedding (see also \cite[Chapter I, Proposition 3.1]{MR1230384}), given by 
   \[
     \iint_{Q_{R}}\varphi_{1}^{2(\kappa-1)}\varphi_{2}^{2} \,{\mathrm d}X \le CR^{2{\tilde \kappa}}\left(\iint_{Q_{R}}\left\lvert \nabla \varphi_{2} \right\rvert^{2}\,{\mathrm d}X\right)\left(\sup_{\tau\in \ring{I_{R}}}\int_{B_{R}\times \{\tau\}} \varphi_{1}^{2}(x,\,t)\,{\mathrm d}x \right)^{\kappa-1} 
   \]
   for some $C=C(\kappa,\,q)\in(0,\,\infty)$, independent of $m$.
   This is easily deduced by applying H\"{o}lder's inequality to the two functions  $\varphi_{1}^{2}(\,\cdot\,,\,t)\in L^{\frac{1}{\kappa-1}}(B_{R})$ and $\varphi_{2}^{2}(\,\cdot\,,\,t)\in L^{\frac{1}{2-\kappa}}(B_{R})$, and the Sobolev embedding to $\varphi_{2}(\,\cdot\,,\,t)\in W_{0}^{1,\,2}(B_{R})\hookrightarrow L^{\frac{2}{2-\kappa}}(B_{R})$. 
   Combining with (\ref{Eq (Section 4): Mid ENERGY}), we have 
   \begin{align*}
      \iint_{Q_{R}}\varphi_{1}^{2(\kappa-1)}\varphi_{2}^{2}  \, {\mathrm d}X  &  \le C(1+\alpha)^{2\kappa}R^{2{\tilde \kappa}}\left(\iint_{Q_{R}}\left( W_{k}^{p+2\alpha}\left(\eta^{2}+\lvert\nabla \eta\rvert^{2}\right)+W_{k}^{2+2\alpha}\lvert\partial_{t}\phi_{\mathrm c}\rvert\right)\, {\mathrm d}X\right)^{\kappa}\\    &  \quad +C(1+\alpha)^{2\kappa}R^{2{\tilde \kappa}}\left(\iint_{Q_{R}} \varphi_{3}^{\frac{2q}{q-2}}\,{\mathrm d}X\right)^{\kappa(1-2/q)}. 
   \end{align*}
   Here we have used $\lVert {\tilde f}_{\varepsilon}\rVert_{L^{q/2}(Q_{R})}\le 1$.
   If $q=\infty$, then the last estimate obviously yields (\ref{Eq (Section 4): Reversed Hoelder}) with $\gamma\coloneqq 2$.
    For $q\in(n+2,\,\infty)$, we interpolate $\varphi_{3}$ among the Lebesgue spaces $L^{2}(Q_{R})\subset L^{\frac{2q}{q-2}}(Q_{R})\subset L^{2\kappa}(Q_{R})$. 
   Combining with Young's inequality  and recalling the definition of $\varphi_{3}$,  we obtain
   \begin{align*}
    \left(\iint_{Q_{R}} \varphi_{3}^{\frac{2q}{q-2}}\,{\mathrm d}X\right)^{\kappa\left(1-\frac{2}{q}\right)}& \le \left(\iint_{Q_{R}} \varphi_{3}^{2}\,{\mathrm d}X\right)^{\kappa^{\prime}\cdot\frac{\kappa q-q-2\kappa}{q}}\left(\iint_{Q_{R}}\varphi_{3}^{2\kappa}\,{\mathrm d}X \right)^{\kappa^{\prime}\cdot\frac{2}{q}}\\ & \le \left(\iint_{Q_{R}} \varphi_{1}^{2}\,{\mathrm d}X\right)^{\kappa^{\prime}\cdot\frac{\kappa q-q-2\kappa}{q}}\left(\iint_{Q_{R}}\varphi_{1}^{2(\kappa-1)}\varphi_{2}^{2} \,{\mathrm d}X \right)^{\kappa^{\prime}\cdot\frac{2}{q}}\\
     & \le \sigma^{-\frac{2\kappa^{\prime}}{q-2\kappa^{\prime}}}\left(\iint_{Q_{R}}\varphi_{1}^{2}\,{\mathrm d}X \right)^{\kappa}+\sigma\iint_{Q_{R}}\varphi_{1}^{2(\kappa-1)}\varphi_{2}^{2}  \, {\mathrm d}X
   \end{align*}
   for all $\sigma\in(0,\,\infty)$.
   Choosing sufficiently small $\sigma>0$,  we deduce  (\ref{Eq (Section 4): Reversed Hoelder}) with $\gamma\coloneqq \frac{2q}{q-2\kappa^{\prime}}\in(2,\,\infty)$.  The desired estimate (\ref{Eq (Section 4): REV H}) is clear by suitably choosing $\eta$ and $\phi_{\mathrm c}$ in (\ref{Eq (Section 4): Reversed Hoelder}). 
   
    We first consider $p\in\lbrack 2,\,\infty)$, where $W_{k}\ge 1$ and (\ref{Eq (Section 4): REV H}) yield
   \[   \iint_{Q_{R_{1}}} W_{k}^{\kappa\beta+(2-p)(\kappa-1)}\,{\mathrm d}X \le \left[\frac{C(n,\,p,\,q,\,\lambda,\,\Lambda,\,K)}{(R_{2}-R_{1})^{2}}\beta^{\gamma} \iint_{Q_{R_{2}}}W_{k}^{\beta} \,{\mathrm d}X \right]^{\kappa}\]
   for any $\beta\in\lbrack p,\,\infty)$ and $0<R_{1}<R_{2}\le R$. 
   For given $\theta\in(0,\,1)$, we set \(r_{l}\coloneqq \theta R+2^{-l}(1-\theta)R\), $B_{l}\coloneqq B_{r_{l}}(x_{0})$, $I_{l}\coloneqq I_{r_{l}}(t_{0})$ and $Q_{l}\coloneqq B_{l}\times I_{l}=Q_{r_{l}}(x_{0},\,t_{0})$ for every $l\in{\mathbb Z}_{\ge 0}$.    
   We define the sequence $\{p_{l}\}_{l=0}^{\infty}\subset\lbrack p,\,\infty)$ as $p_{l}\coloneqq 2\kappa^{l}+p-2$, which satisfies the recurrence formula $p_{l+1}=\kappa p_{l}+(2-p)(\kappa-1)$ with $p_{0}=p$.  Noting $1/2\le r_{l+1}/r_{l}\le 1$, we easily notice that the sequence \(Y_{l}\coloneqq \left(\fiint_{Q_{l}}W_{k}^{p_{l}}\,{\mathrm d}X \right)^{1/p_{l}}\) for $l\in{\mathbb Z}_{\ge 0}$ satisfies all of the assumptions in Lemma \ref{Lemma: Geometric Convergence on Seq}  with $\mu\coloneqq 2$,  $A=(1-\theta)^{-2}C$ and $B\coloneqq 4\kappa^{\kappa\gamma}$ for some sufficiently large $C=C(n,\,p,\,q,\,\lambda,\,\Lambda,\,K)\in (1,\,\infty)$. Lemma \ref{Lemma: Geometric Convergence on Seq} yields $\limsup\limits_{l\to\infty}Y_{l}\le A^{\frac{\kappa^{\prime}}{2}}B^{\frac{(\kappa^{\prime})^{2}}{2}}Y_{0}^{\frac{p}{2}}$, which implies (\ref{Eq (Section 4): Claim on W-k}).
   
    In the remaining case $p\in(p_{\mathrm c},\,2)$, which clearly yields $n\ge 3$ and $\kappa=1+2/n$, from (\ref{Eq (Section 4): REV H}) we have
   \[   \iint_{Q_{R_{1}}} W_{k}^{\kappa\beta+p-2}\,{\mathrm d}X \le \left[\frac{C(n,\,p,\,q,\,\lambda,\,\Lambda,\,K)}{(R_{2}-R_{1})^{2}}\beta^{\gamma} \iint_{Q_{R_{2}}}W_{k}^{\beta} \,{\mathrm d}X \right]^{\kappa}\]
   for any $\beta\in\lbrack 2,\,\infty)$ and $0<R_{1}<R_{2}\le R$. 
   For given $\theta\in(0,\,1)$, we arbitrarily take $\theta\le \theta_{1}<\theta_{2}\le 1$. For every $l\in{\mathbb Z}_{\ge 0}$, we choose \(r_{l}\coloneqq \theta_{1} R+2^{-l}(\theta_{2}-\theta_{1})R\). Corresponding to this $r_{l}$, we define $B_{l}$, $I_{l}$, $Q_{l}$ in the same manner. However, we choose $p_{l}\coloneqq n(1-p/2)+[2-n(1-p/2)]\kappa^{l}\in\lbrack 2,\,\infty)$ for $l\in{\mathbb Z}_{\ge 0}$, which satisfies the recurrence formula $p_{l+1}=\kappa p_{l}+p-2$ with $p_{0}=2$. 
   Then, the sequence $Y_{l}$, defined in the same manner, satisfies all of the assumptions in Lemma \ref{Lemma: Geometric Convergence on Seq} with $\mu\coloneqq 2-n(1-p/2)\in(0,\,\infty)$,  $A=(\theta_{2}-\theta_{1})^{-2}C$ and $B\coloneqq 4\kappa^{\kappa\gamma}$ for some sufficiently large $C=C(n,\,p,\,q,\,\lambda,\,\Lambda,\,K)\in (1,\,\infty)$. By the resulting inequality $\limsup\limits_{l\to\infty}Y_{l}\le A^{\frac{\kappa^{\prime}}{\mu}}B^{\frac{(\kappa^{\prime})^{2}}{\mu}}Y_{0}^{\frac{2}{\mu}}$ and Young's inequality, we have
   \[
     \esssup_{Q_{\theta_{1} R}}\,W_{k}
      \le \left[\frac{C}{(\theta_{2}-\theta_{1})^{2\kappa^{\prime}}}\fiint_{Q_{\theta_{1}R}}W_{k}^{2}\,{\mathrm d}X \right]^{\frac{1}{2-n(1-p/2)}}
       \le \frac{1}{2}\esssup_{Q_{\theta_{2}R}}\,W_{k}+ \left[\frac{C}{(\theta_{2}-\theta_{1})^{2\kappa^{\prime}}}\fiint_{Q_{\theta_{2}R}} W_{k}^{p}\,{\mathrm d}X\right]^{\frac{d}{p}}
   \]
   for any $\theta\le \theta_{1}<\theta_{2}\le 1$. The desired estimate (\ref{Eq (Section 4): Claim on W-k}) follows from Lemma \ref{Lemma (Section 4): ABSORB LEMMA}.
  \end{proof}
\section{Degenerate case}\label{Sect:Degenerate Region}
 In Section \ref{Sect:Degenerate Region}, we aim to prove Proposition \ref{Proposition: Parabolic De Giorgi} from the estimates (\ref{Eq (Section 3): De Giorgi Truncation 1})--(\ref{Eq (Section 3): De Giorgi Truncation 2}).
 The former (\ref{Eq (Section 3): De Giorgi Truncation 1}) is used to deduce oscillation lemmata (Lemmata \ref{Lemma: Superlevelset decay}--\ref{Lemma: Density Lemma}).
 There we make use of a result of the expansion of posivity (Lemma \ref{Lemma: Expansion of Positivity}), which is verified by the latter (\ref{Eq (Section 3): De Giorgi Truncation 2}).
 As related items, see \cite[Chapters II--III]{MR0241822}, \cite[Chapters 3--4]{MR2865434}.
 \subsection{Levelset estimates}
 Firstly, we prove various estimates for the level sets of $U_{\delta,\,\varepsilon}$.
 For given $\tau_{0}\in(t_{0}-4\rho^{2},\,t_{0})$ and $\gamma\in(0,\,1)$, we define
 \(I_{\frac{3}{2}\rho}(\gamma;\,\tau_{0})\coloneqq (\tau_{0},\,\tau_{0}+\gamma(\frac{3}{2}\rho)^{2}\rbrack\subset {\tilde I}_{\frac{3}{2}\rho}(\gamma;\,\tau_{0})\coloneqq (\tau_{0}-\gamma(\frac{3}{2}\rho)^{2},\,\tau_{0}+\gamma(\frac{3}{2}\rho)^{2}\rbrack,\)
 and $Q_{\frac{3}{2}\rho}(\gamma;\,\tau_{0})\coloneqq B_{\frac{3}{2}\rho}(x_{0})\times I_{\frac{3}{2}\rho}(\gamma;\,\tau_{0})$.
 \begin{lemma}\label{Lemma: Superlevelset decay}
  In addition to the assumptions of Proposition \ref{Proposition: Parabolic De Giorgi}, let
  \begin{align}
    \label{Eq (Section 5): Inclusion on intervalis}
    &I_{\frac{3}{2}\rho}(\gamma;\,\tau_{0})\subset {\tilde I}_{\frac{3}{2}\rho}(\gamma;\,\tau_{0}) \subset (t_{0}-4\rho^{2},\, t_{0}\rbrack,\\ 
    \label{Eq (Section 5): Positivity assumption}
    &\left\lvert\left\{x\in B_{\rho}\mathrel{}\middle|\mathrel{}U_{\delta,\,\varepsilon}(x,\,t)\le (1-{\hat\nu})\mu^{2} \right\}\right\rvert\ge \tilde{\nu}\lvert B_{\rho}\rvert \quad \textrm{for a.e.~}t\in I_{\frac{3}{2}\rho}(\gamma;\,\tau_{0}),
  \end{align}
  and $\rho^{\beta}\le 2^{-i_{\star}}\hat{\nu}$ hold for some $i_{\star}\in{\mathbb N}$.
  Then, we have
  \begin{equation}\label{Eq (Section 5): Measure result}
    \left\lvert Q_{\frac{3}{2}\rho}(\gamma;\,\tau_{0})\cap\left\{ U_{\delta,\,\varepsilon}\ge\left(1-2^{-i_{\star}}{\hat\nu}\right)\mu^{2} \right\}\right\rvert\le \frac{C_{\dagger}}{\tilde{\nu}\sqrt{\gamma i_{\star}}}\left\lvert Q_{\frac{3}{2}\rho}(\gamma;\,\tau_{0})\right\rvert,
  \end{equation}
  for some $C_{\dagger}\in(1,\,\infty)$ depending at most on $n$, $p$, $q$, $\lambda$, $\Lambda$, $K$, $F$, $\delta$, and $M$.
\end{lemma}

We invoke an isoperimetric inequality for the functions in $W^{1,\,1}(B_{\rho})$ (see e.g., \cite[Chapter 10, \S 5.1]{MR2566733}); 
for any fixed numbers $-\infty<k<l<\infty$ and $v\in W^{1,\,1}(B_{\rho})$, there holds
\begin{equation}\label{Eq (Section 5): Isoperimetric Ineq}
  (l-k)\lvert \{x\in B_{\rho}\mid v(x)>l \} \rvert\le \frac{C_{n}\rho^{n+1}}{\lvert\{x\in B_{\rho}\mid v(x)<k \}\rvert}\int_{B_{\rho}\cap \{k<v<l\}}v\,{\mathrm d}x.
\end{equation}
\begin{proof}
  We set $k_{i}\coloneqq (1-2^{-i}{\hat\nu})\mu^{2}$ and $A_{i}\coloneqq Q\cap \{U_{\delta,\,\varepsilon}>k_{i}\}$ for each $i\in{\mathbb Z}_{\ge 0}$, where $Q\coloneqq Q_{\frac{3}{2}\rho}(\gamma;\,\tau_{0})$.
  By the definition of $k_{i}$, $k_{i+1}-k_{i}=2^{-i-1}{\hat\nu}\mu^{2}$ is clear for every $i\in{\mathbb Z}_{\ge 0}$.
  Note that (\ref{Eq (Section 5): Positivity assumption}) enables us to apply (\ref{Eq (Section 5): Isoperimetric Ineq}) to the sliced function $U_{\delta,\,\varepsilon}(\,\cdot\,,\,t)$.
  Hence, we obtain
  \[
    \frac{{\hat \nu}\mu^{2}}{2^{i+1}}\left\lvert\left\{x\in B_{\frac{3}{2}\rho}\mathrel{} \middle|\mathrel{} U_{\delta,\,\varepsilon}(x,\,t)>k_{i+1} \right\}\right\rvert
    \le \frac{C(n)\rho}{\tilde{\nu}}\int_{B_{\frac{3}{2}\rho}\cap\{k_{i}<U_{\delta,\,\varepsilon}(\,\cdot\,,\,t) <k_{i+1}\}}\lvert \nabla (U_{\delta,\,\varepsilon}(x,\,t)-k_{i})_{+} \rvert\,{\mathrm d}x
  \]
  for  a.e.~$t\in I_{\frac{3}{2}\rho}(\gamma;\,\tau_{0})$, and $i\in{\mathbb Z}_{\ge 0}$.
  Integrating with respect to $t\in I_{\frac{3}{2}\rho}(\gamma;\,\tau_{0})$ and applying the Cauchy--Schwarz inequality, we get
  \[
    \frac{{\hat\nu}\mu^{2}}{2^{i+1}}\lvert A_{i+1}\rvert\le \frac{C(n)\rho}{\tilde\nu}\left(\iint_{Q}\lvert \nabla (U_{\delta,\,\varepsilon}-k_{i})_{+}\rvert^{2}\,{\mathrm d}X \right)^{1/2}\left(\lvert A_{i}\rvert-\lvert A_{i+1}\rvert \right)^{1/2}
  \]
  for every $i\in{\mathbb Z}_{\ge 0}$.
  To estimate the integration on the right-hand side,
  we make use of (\ref{Eq (Section 3): De Giorgi Truncation 1}) with $Q_{0}\coloneqq B_{2\rho}\times {\tilde I}_{\frac{3}{2}\rho}(\gamma;\,\tau_{0})$ and $k\coloneqq k_{i}$.
  Choosing suitable cut-off functions $\eta\in C_{\mathrm c}^{1}(B_{2\rho};\,\lbrack 0,\,1\rbrack)$ and $\phi_{\mathrm c}\in C^{1}([t_{0}-\gamma(3\rho/2)^{2},\,t_{0}+\gamma(3\rho/2)^{2}];\,\lbrack 0,\,1\rbrack)$, and 
  noting $(U_{\delta,\,\varepsilon}-k_{i})_{+}\le 2^{-i}{\hat \nu}\mu^{2}$, we compute
  \[
    \iint_{Q}\lvert \nabla (U_{\delta,\,\varepsilon}-k_{i})_{+} \rvert^{2}\,{\mathrm d}X\le C\left[\frac{4^{-i}\hat{\nu}\mu^{4}}{\gamma\rho^{2}}\lvert Q_{0}\rvert+\mu^{4}F^{2}\lvert Q_{0}\rvert^{1-2/q}\right]\le \frac{C({\hat \nu}\mu^{2})^{2}}{4^{i}\gamma \rho^{2}}\left[ 1+ \left(\frac{2^{i}\rho^{\beta}}{{\hat \nu}} \right)^{2}  \right]\lvert Q\rvert
  \]
  for some constant $C\in(1,\,\infty)$ depending at most on $n$, $p$, $q$, $\lambda$, $\Lambda$, $K$, $\delta$ and $M$.
  The proof is completed by showing (\ref{Eq (Section 5): Measure result}), provided $2^{i_{\star}}\rho^{\beta}\le {\hat\nu}$ and (\ref{Eq (Section 5): Positivity assumption}).
  If $2^{i_{\star}}\rho^{\beta}\le {\hat \nu}$ holds for some fixed $i_{\star}\in{\mathbb N}$, then we can find a constant $C_{\dagger}\in(1,\,\infty)$ satisfying 
  \[\lvert A_{i+1}\rvert^{2}\le \frac{C_{\dagger}^{2}}{{\tilde\nu}^{2}\gamma}\lvert Q\rvert(\lvert A_{i}\rvert-\lvert A_{i+1}\rvert)\]
  for each $i\in\{\,0,\,1,\,\dots\,,\,i_{\star}-1\}$. From this, we obtain
  \[i_{\star}\lvert A_{i_{\star}}\rvert^{2}\le \sum_{i=0}^{i_{\star}-1}\lvert A_{i+1}\rvert^{2}\le \frac{C_{\dagger}^{2}}{\tilde{\nu}^{2}\gamma}\lvert Q\rvert\cdot \lvert A_{0}\rvert\le \frac{C_{\dagger}^{2}}{\tilde{\nu}^{2}\gamma}\lvert Q\rvert^{2},\quad \text{whence}\quad \lvert A_{i_{\star}}\rvert\le \frac{C_{\dagger}}{{\tilde\nu}\sqrt{\gamma i_{\star}} }\lvert Q\rvert.\qedhere \]
\end{proof}
 \begin{lemma}\label{Lemma: Density Lemma}
 Under the assumptions of Proposition \ref{Proposition: Parabolic De Giorgi},  there exists $c_{\ast}\in(0,\,1)$, depending at most on $n$, $p$, $q$, $\lambda$, $\Lambda$, $K$, $F$, $\delta$, and $M$, such that
 \begin{equation}\label{Eq (Section 5): Density Assumption}
   \left\lvert Q_{\frac{3}{2}\rho}(\gamma;\,\tau_{0})\cap \left\{ U_{\delta,\,\varepsilon}\ge (1-\nu_{0})\mu^{2}\right\}\right\rvert\le \alpha_{\star}\left\lvert Q_{\frac{3}{2}\rho}(\gamma;\,\tau_{0}) \right\rvert \quad \text{with}\quad \alpha_{\star}(\gamma)\coloneqq c_{\ast}\gamma^{\frac{n+2}{2\beta}}
 \end{equation}
 and $ \rho^{\beta} <\nu_{0}$ imply
 \begin{equation}\label{Eq (Section 5): Density Result}
   \sup_{{\tilde Q}(\gamma;\,\tau_{0})}\,U_{\delta,\,\varepsilon}\le \left(1-\frac{\nu_{0}}{2}\right)\mu^{2},\quad\text{where}\quad  {\tilde Q}(\gamma;\,\tau_{0})\coloneqq B_{\rho}\times \left(\tau_{0}+\frac{5\gamma}{4}\rho^{2},\, \tau_{0}+\frac{9\gamma}{4}\rho^{2}\right\rbrack.
 \end{equation} 
 \end{lemma}
 Before the proof, we infer a well-known lemma without a proof (see \cite[Chapter 9, Lemma 15.1]{MR2566733}).
 \begin{lemma}\label{Lemma (Section 5): REC}
 Let $\{Y_{l}\}_{l=0}^{\infty}$ be a sequence of positive numbers that satisfies the recursive inequalities 
 \(Y_{l+1}\le AB^{l}Y_{l}^{1+\varsigma}\) for all \(l\in{\mathbb Z}_{\ge 0}\),
 where $A,\,B\in(1,\,\infty)$, and $\varsigma\in(0,\,\infty)$ are constant. If $Y_{0}\le A^{-\varsigma^{-1}}B^{-\varsigma^{-2}}$, then $Y_{l}\to 0$ as $l\to\infty$.
 \end{lemma} We also invoke the parabolic Poincar\'{e}--Sobolev inequality
 \begin{equation}\label{Eq (Section 5): Parabolic Poincare-Sobolev}
  \iint_{B\times I}\lvert v\rvert^{2+\frac{4}{n}}\,{\mathrm d}X\le C(n)\left(\esssup_{\tau\in I}\int_{B\times \{\tau\}}\lvert v(x,\,\tau)\rvert^{2}\,{\mathrm d}x \right)^{2/n}\iint_{B\times I}\lvert \nabla v\rvert^{2}\,{\mathrm d}X
 \end{equation}
 for all $v\in L^{2}(I;\,W_{0}^{1,\,2}(B))\cap L^{\infty}(I;\,L^{2}(B))$, where $I\subset {\mathbb R}$ is a bounded open interval, and $B\subset {\mathbb R}^{n}$ is a bounded open ball (see \cite[Chapter I, Proposition 3.1]{MR1230384}).
 \begin{proof}
   We set $\rho_{l}\coloneqq (1+2^{-l-1})\rho \in(\rho,\,3\rho/2\rbrack$, $\tau_{l}\coloneqq \tau_{0}+\frac{9}{4}\gamma\rho^{2}-\gamma\rho_{l}^{2}\in (-4\rho^{2},\,0)$, $B_{l}\coloneqq B_{\rho_{l}}\subset B_{2\rho}$, $I_{l}\coloneqq \left(\tau_{l},\, \tau_{0}+\frac{9}{4}\gamma\rho^{2}\right\rbrack\subset (t_{0}-4\rho^{2},\,t_{0}\rbrack$ and $Q_{l}\coloneqq B_{l}\times I_{l}\subset Q_{2\rho}$ for each $l\in{\mathbb Z}_{\ge 0}$.
   We also define a superlevel set $A_{l}\coloneqq Q_{l}\cap \{ U_{\delta,\,\varepsilon}>k_{l}\}$, where $k_{l}\coloneqq (1-(2^{-1}+2^{-l-1})\nu_{0})\mu^{2}$ and a ratio $Y_{l}\coloneqq \lvert A_{l}\rvert/\lvert B_{l}\rvert\in(0,\,1\rbrack$.
   Applying H\"{o}lder's inequality and (\ref{Eq (Section 5): Parabolic Poincare-Sobolev}), and  choosing $\phi_{\mathrm c}$ and $\eta$ that satisfies $\eta_{l}|_{B_{l+1}}\equiv1$, $\phi_{\mathrm c}|_{I_{l+1}}\equiv 1$, and $\lVert \nabla \eta \rVert_{L^{\infty}(B_{l})}^{2}+\lVert \partial_{t}\phi_{\mathrm c}\rVert_{L^{\infty}(I_{l})}\le c\cdot 4^{l}/(\gamma\rho^{2})$ for some constant $c\in(0,\,\infty)$, from (\ref{Eq (Section 3): De Giorgi Truncation 1}), we obtain  
   \begin{align*}
     &\frac{\nu_{0}\mu^{2}}{2^{l+2}}\lvert A_{l+1}\rvert \le\left(\iint_{Q_{l}}\left[(U_{\delta,\,\varepsilon}-k_{l})_{+}\eta\phi^{1/2}\right]^{2+\frac{4}{n}}\,{\mathrm d}X \right)^{\frac{n}{2(n+2)}}  \lvert A_{l}\rvert^{1-\frac{n}{2(n+2)}}\\ 
     &\le C(n)\left(\esssup_{\tau\in I_{l}}\int_{B_{l}\times \{\tau\}}(U_{\delta,\,\varepsilon}-k_{l})_{+}^{2}\left(\eta^{2}\phi\right)\,{\mathrm d}x \right)^{\frac{1}{n+2}}\left(\iint_{Q_{l}}\lvert \nabla [(U_{\delta,\,\varepsilon}-k_{l})_{+}\eta]\rvert^{2}\phi\,{\mathrm d}X \right)^{\frac{n}{2(n+2)}} \lvert A_{l}\rvert^{\frac{1}{2}+\frac{1}{n+2}}\\ 
     &
     \le \frac{C\nu_{0}\mu^{2}\cdot 2^{l}}{\sqrt{\gamma}}\left[\frac{\lvert A_{l}\rvert^{1/q}}{\rho}+\frac{1}{\nu_{0}} \right]\lvert A_{l}\rvert^{\frac{1}{2}-\frac{1}{q}}\cdot \lvert A_{l}\rvert^{\frac{1}{2}+\frac{1}{n+2}} \le \frac{C\nu_{0}\mu^{2}\cdot 2^{l}}{\sqrt{\gamma}}\left[\frac{\lvert Q_{l}\rvert^{1/q}}{\rho}+\frac{1}{\nu_{0}} \right]\lvert A_{l}\rvert^{1+\frac{\beta}{n+2}}
   \end{align*}
   for some $C\in(1,\,\infty)$ depending at most on $n$, $p$, $q$, $\lambda$, $\Lambda$, $K$, $F$, $\delta$, and $M$.
   Noting $\lvert Q_{l}\rvert/\lvert Q_{l+1}\rvert\le (\rho_{l}/\rho_{l+1})=[1+(1+2^{l+2})^{-1}]^{n+2}\le (6/5)^{n+2}$, we have
   \[
     Y_{l+1}\le \left(\frac{6}{5}\right)^{n+2}\frac{\lvert A_{l+1}\rvert}{\lvert Q_{l}\rvert^{1+\frac{\beta}{n+2}}}\cdot\lvert Q_{l}\rvert^{\frac{\beta}{n+2}}\le \frac{{\tilde C_{\ast}}\cdot 4^{l}}{2\sqrt{\gamma}}\left[1+\frac{  \rho^{\beta}  }{\nu_{0}} \right]Y_{l}^{1+\frac{\beta}{n+2}}
   \]
   for every $l\in{\mathbb Z}_{\ge 0}$.
   Here ${\tilde C}_{\ast}$ depends at most on $n$, $p$, $q$, $\lambda$, $\Lambda$, $K$, $F$, $\delta$, and $M$.
   We choose \(c_{\ast}\coloneqq {\tilde C_{\ast}}^{-\frac{n+2}{\beta}}4^{-\frac{(n+2)^{2}}{\beta^{2} }}\), and note that (\ref{Eq (Section 5): Density Assumption}) implies $Y_{0}\le \alpha_{\star}$.  
   In particular, if $\rho^{\beta}<\nu_{0}$ holds, then we can apply Lemma \ref{Lemma (Section 5): REC} to conclude $Y_{l}\to 0$ as $l\to \infty$, which completes the proof.
 \end{proof}

 Finally, we prove Lemma \ref{Lemma: Expansion of Positivity}, which is often called the expansion of positivity.
 \begin{lemma}\label{Lemma: Expansion of Positivity}
   In addition to the assumptions of Proposition \ref{Proposition: Parabolic De Giorgi}, let $U_{\delta,\,\varepsilon}$ satisfy
   \begin{equation}\label{Eq (Section 5): Time Propagation Assumption} 
     \left\lvert\left\{x\in B_{\rho}\mathrel{}\middle|\mathrel{} U_{\delta,\,\varepsilon}(x,\,\tau_{0})\le (1-\nu)\mu^{2} \right\} \right\rvert\ge \frac{\nu}{2}\lvert B_{\rho}\rvert
   \end{equation}
   for some $\tau_{0}\in(t_{0}-\rho^{2},\,t_{0}-\nu\rho^{2}/2\rbrack$.
   Then, there exist sufficiently small numbers $\gamma_{0}\in(0,\,\nu/9)$, $\theta_{0}\in(0,\,1/2)$ and $\rho_{\star}\in(0,\,1)$, which depend at most on $n$, $p$, $q$, $\lambda$, $\Lambda$, $K$, $\delta$, $M$ and $\nu$, such that (\ref{Eq (Section 5): Inclusion on intervalis})--(\ref{Eq (Section 5): Positivity assumption}) hold with $(\gamma,\,{\tilde{\nu}},\,{\hat \nu})\coloneqq (\gamma_{0},\,\nu/4,\,\theta_{0}\nu)$, provided $\rho\le \rho_{\star}$.
 \end{lemma}
 \begin{proof}
   We first note that (\ref{Eq (Section 5): Inclusion on intervalis}) always holds provided $\gamma_{0}\in(0,\,\nu/9)$.
   For $\sigma\in(0,\,1)$, we choose a suitable $\eta\in C_{\mathrm c}^{1}(B_{\rho};\,\lbrack 0,\,1\rbrack)$ satisfying $\eta|_{B_{(1-\sigma)\rho}}\equiv 1$,
   and apply (\ref{Eq (Section 3): De Giorgi Truncation 2}) with $Q_{0}=B_{\rho}\times (\tau_{0},\,\tau_{1}\rbrack$, where $\tau_{1}\coloneqq \tau_{0}+\gamma_{0}(\frac{3}{2}\rho)^{2}$.
   Define
   \(A_{\theta_{0}\nu,\,r}(\tau)\coloneqq \left\{x\in B_{r}\mid U_{\delta,\,\varepsilon}(x,\,\tau)>(1-\theta_{0}\nu)\mu^{2} \right\}\)
   for $r\in(0,\,\rho\rbrack$, $\tau\in(\tau_{0},\,\tau_{1}\rbrack$.
   We choose $k\coloneqq (1-\nu)\mu^{2}$, so that $(U_{\delta,\,\varepsilon}-k)_{+}\le \nu\mu^{2}$.
   Also, $U_{\delta,\,\varepsilon}-k\ge (1-\theta_{0})\nu\mu^{2}$ holds in $A_{\theta_{0}\nu,\,r}$.
   By (\ref{Eq (Section 3): De Giorgi Truncation 2}), $\lvert Q_{0}\rvert\le \frac{9}{4}  \gamma_{0}\rho^{2}\lvert B_{\rho}\rvert$ and $q>n+2$, we obtain
   \begin{align*}
     &(1-\theta_{0})^{2}(\nu\mu^{2})^{2}\lvert A_{\theta_{0}\nu,\,(1-\sigma)\rho}(\tau) \rvert\\ 
     &\le (\nu\mu^{2})^{2}\left\lvert \left\{x\in B_{\rho}\mathrel{}\middle|\mathrel{} U_{\delta,\,\varepsilon}(x,\,\tau_{0})>(1-\nu)\mu^{2} \right\}\right\rvert+C\left[\frac{(\nu\mu^{2})^{2}}{(\sigma\rho)^{2}}\lvert Q_{0}\rvert+\mu^{4}F^{2}\lvert Q_{0}\rvert^{1-2/q} \right]\\ 
     &\le (\nu\mu^{2})^{2}\lvert B_{\rho}\rvert\left(1-\frac{\nu}{2}+\frac{C\gamma_{0}}{\sigma^{2}}+\frac{CF^{2}\rho^{2\beta}}{\nu^{2}}\right).
    \end{align*}
   Noting $\lvert A_{\theta_{0}\nu,\,\rho}(\tau)\rvert\le
   \lvert B_{\rho}\setminus B_{(1-\sigma)\rho}\rvert+\lvert A_{\theta_{0}\nu,\,(1-\sigma)\rho}(\tau)\rvert$, we have
   \[
   \lvert A_{\theta_{0}\nu,\,\rho}(\tau)\rvert
    \le\left(n\sigma+\frac{1-\nu/2}{(1-\theta_{0})^{2}}+\frac{C_{0}\gamma_{0}}{\sigma^{2}(1-\theta_{0})^{2}}+\frac{C_{0}\rho^{2\beta}}{(1-\theta_{0})^{2}\nu^{2}} \right) \lvert B_{\rho}\rvert
   \]
   for some $C_{0}=C_{0}(n,\,p,\,q,\,\lambda,\,\Lambda,\,K,\,F,\,\delta,\,M)\in(1,\,\infty)$.
   Finally, we choose $\sigma\in(0,\,1)$, $\theta_{0}\in(0,\,1/2)$, $\gamma_{0}\in(0,\,\nu/9)$ and $\rho_{\star}\in(0,\,1)$ satisfying
   \[n\sigma\le\frac{\nu}{24},\quad \frac{1-\nu/2}{(1-\theta_{0})^{2}}\le 1-\frac{\nu}{8},\quad \frac{C_{0}\gamma_{0}}{\sigma^{2}(1-\theta_{0})^{2}}\le \frac{\nu}{24},\quad \text{and}\quad \frac{C_{0}\rho_{\star}^{2\beta}}{\nu^{2}(1-\theta_{0})^{2}}\le\frac{\nu}{24},\]
   so that $\lvert A_{\theta_{0}\hat{\nu},\,\rho}(\tau)\rvert\le (1-\nu/4)\lvert B_{\rho}\rvert$ holds for  a.e.~$\tau\in(\tau_{0},\,\tau_{1}\rbrack$, provided $\rho\le \rho_{\star}$.
 \end{proof}

 \subsection{Proof of the De Giorgi-type oscillation estimates}
 We conclude Section \ref{Sect:Degenerate Region} by giving the proof of Proposition \ref{Proposition: Parabolic De Giorgi}.
 \begin{proof}[Proof of Proposition \ref{Proposition: Parabolic De Giorgi}]
   By (\ref{Eq (Section 2): Measure Assumption De Giorgi}), there exists $\tau_{0}\in(t_{0}-\rho^{2},\,t_{0}-\nu\rho^{2}/2\rbrack$ such that (\ref{Eq (Section 5): Time Propagation Assumption}) holds.
   In fact, if such $\tau_{0}$ never exists, we can easily compute
   \(\lvert Q_{\rho}\setminus S_{\rho,\,\nu,\,\mu}\rvert\le \frac{\nu}{2}+(1-\frac{\nu}{2})\cdot\frac{\nu}{2}<\nu\lvert Q_{\rho}\rvert\), 
   which contradicts with (\ref{Eq (Section 2): Measure Assumption De Giorgi}).
   We fix such $\tau_{0}$ that satisfies (\ref{Eq (Section 5): Time Propagation Assumption}).

   Let $\gamma_{0}\in(0,\,\nu/9)$ be as in Lemma \ref{Lemma: Expansion of Positivity}.
   We set $A\coloneqq (t_{0}-\tau_{0})\rho^{-2}\in \lbrack \nu/2,\,1)$, and choose $l_{0}\in{\mathbb N}$, depending at most on $A=A(\nu)$ and $\gamma_{0}$, as the unique number satisfying \(9\gamma_{0}l_{0}\ge 4A>9\gamma_{0}(l_{0}-1)\).
   We also define $\gamma_{1}\coloneqq 4A/(9l_{0})$, and $T_{k}\coloneqq \tau_{0}+(2+\frac{k}{4})\gamma_{1}\rho^{2}$ for each $k\in\{\,1,\,\dots\,,\,9l_{0}-8\,\}$.
   By the definitions, we have $l_{0}\le 1+\frac{4}{9\gamma_{0}}\eqqcolon L_{0}$, and $T_{9l_{0}-8}=t_{0}$.
   We fix $i_{\star}\in {\mathbb N}$ satisfying
   \(\kappa\coloneqq \sqrt{1-2^{-(i_{\star}+1)(9L_{0}-8)}}>(\sqrt{\nu}/6)^{\beta}\), \(\frac{C_{\dagger}}{(\nu/4)\sqrt{i_{\star}\gamma_{1}}}\le \alpha_{\star}(\gamma_{1})\), and \(\frac{C_{\dagger}}{\sqrt{i_{\star}(\gamma_{1}/3)}}\le \alpha_{\star}(\gamma_{1}/3)\),
   where $C_{\dagger}\in(1,\,\infty)$ and $\alpha_{\star}\in(0,\,1)$ are respectively given by Lemma \ref{Lemma: Superlevelset decay} and Lemma \ref{Lemma: Density Lemma}.
   By $\gamma_{0}\ge \gamma_{1}\ge 2\nu/(9L_{0})$, $i_{\star}$ is determined by $\nu$ and $\gamma_{0}$, as well as $n$, $p$, $q$, $\lambda$, $\Lambda$, $K$, $F$, $M$, and $\delta$.
   We choose the radius ${\tilde\rho}\in(0,\,\rho_{\star})$ satisfying
   \(
     {\tilde{\rho}}^{\beta}<2^{-[(9l_{0}-8)i_{\star}+9(l_{0}-1)]}\theta_{0}\nu.
   \) 
   We claim that
   \begin{equation}\label{Eq (Section 5): De Giorgi Induction}
     U_{\delta,\,\varepsilon}\le \left(1-2^{-(i_{\star}+1)k} \right)\mu^{2}\quad \textrm{in}\quad B_{\frac{3}{2}\rho}\times I_{k}
   \end{equation}
   for every $k\in\{\,1,\,\dots\,,\,9l_{0}-8\,\}$, where $I_{k}\coloneqq \left(t_{0}+\frac{5}{4}\gamma_{1}\rho^{2},\, T_{k}\right\rbrack$.
   To prove (\ref{Eq (Section 5): De Giorgi Induction}) for $k=1$, we are allowed to apply Lemmata \ref{Lemma: Superlevelset decay}--\ref{Lemma: Density Lemma} with $(\gamma,\,{\tilde{\nu}},\, {\hat \nu},\,\nu_{0})\coloneqq (\gamma_{1},\,\nu/4,\,\theta_{0}\nu,\,2^{-i_{\star}}\theta_{0}\nu)$ by our choice of ${\tilde \rho}$.
   Therefore, the proof of (\ref{Eq (Section 5): De Giorgi Induction}) with $l_{0}=1$ is completed.
   For $l_{0}\ge 2$, since 
   our choice of ${\tilde \rho}\in(0,\,1)$ yields ${\tilde\rho}^{\beta}<2^{-(i_{\star}+1)k-i_{\star}}$ for every $k\in\{\,1,\,\dots\,,\,9l_{0}-9\,\}$, this allows us to repeatedly apply Lemmata \ref{Lemma: Superlevelset decay}--\ref{Lemma: Density Lemma} with 
   \((\gamma,\,{\tilde{\nu}},\, {\hat \nu},\,\nu_{0})\coloneqq (\gamma_{1}/3,\,1,\,2^{-(i_{\star}+1)k}\theta_{0}\nu,\,2^{-(i_{\star}+1)k-i_{\star}}\theta_{0}\nu)\)
   for each $k\in\{\,1,\,\dots\,9l_{0}-9\,\}$.
   Thus, we conclude (\ref{Eq (Section 5): De Giorgi Induction}) for every $k\in\{\,1,\,\dots\,,\,9l_{0}-8\,\}$.
   Noting 
   \(\left(\tau_{0}+\frac{5}{4}\gamma_{1}\rho^{2}\right)-t_{0}
   \le \frac{4}{9}(\tau_{0}-t_{0})\le -\left(\frac{\sqrt{\nu}}{3}\rho\right)^{2}\), and using (\ref{Eq (Section 5): De Giorgi Induction}) with $k=9l_{0}-8$, and recalling $U_{\delta,\,\varepsilon}=\lvert {\mathcal G}_{\delta,\,\varepsilon}(\nabla u_{\varepsilon}) \rvert^{2}$, we finally obtain (\ref{Eq (Section 2): De Giorgi Oscillation result}) with 
   \(\kappa>(\sqrt{\nu}/6)^{\beta}\).
  \end{proof}
\section{Non-degenerate case}\label{Sect:Non-Degenerate Region}
 Section \ref{Sect:Non-Degenerate Region} aims to give the proof of Proposition \ref{Proposition: Parabolic Campanato}.
 Here we mainly deal with an oscillation energy
 \(\Phi(r)\coloneqq \fiint_{Q_{r}}\lvert \nabla u_{\varepsilon}-(\nabla u_{\varepsilon})_{r}\rvert^{2}\,{\mathrm d}X\) for \(r\in(0,\,\rho\rbrack,\)
 where $Q_{2\rho}=Q_{2\rho}(x_{0},\,t_{0})$ satisfies (\ref{Eq (Section 2): Bounds of V-epsilon}).
 To make it successful, we carefully use (\ref{Eq (Section 2): Delta vs Mu}) and (\ref{Eq (Section 2): Measure Assumption Campanato}) to verify that the average integral of a gradient will not degenerate.


 Throughout Section \ref{Sect:Non-Degenerate Region}, we often use a well-known fact that, for $l\in{\mathbb N}$, there holds 
 \begin{equation}\label{Eq (Section 6): L-2 oscillation minimizer}
   \fint_{U}\lvert f-f_{U}\rvert^{2}\,{\mathrm d}y=\min_{\xi\in{\mathbb R}^{n}} \fint_{U}\lvert f-\xi\rvert^{2}\,{\mathrm d}y 
 \end{equation}
 for any $f\in L^{2}(U,\,{\mathbb R}^{n})$, where $U\subset{\mathbb R}^{l}$ is a Lebesgue measurable set with $0<\lvert U\rvert<\infty$.
 \subsection{Energy estimates}
 From Lemma \ref{Lemma: Energy estimates from weak form}, we aim to prove Lemma \ref{Lemma: Oscillation from Hessian} below, which plays an important role in making our comparison arguments successful.
 \begin{lemma}\label{Lemma: Oscillation from Hessian}
  For every $\theta\in (0,\,1/16)$, there exist sufficiently small $\nu\in(0,\,1/4)$ and $\rho_{\ast}\in(0,\,1)$, depending at most on $n$, $p$, $q$, $\lambda$, $\Lambda$, $K$, $F$, $M$, $\delta$, and $\theta$, such that the following estimates (\ref{Eq (Section 6): Mean bounds by below})--(\ref{Eq (Section 6): Energy oscillation bounds from Hessian}) hold, if the assumptions of Proposition \ref{Proposition: Parabolic Campanato} and $\rho\le \rho_{\ast}$ are verified.
  \begin{align}
    \label{Eq (Section 6): Mean bounds by below}
    \lvert (\nabla u_{\varepsilon})_{\rho} \rvert\ge \delta+\frac{\mu}{2}.\\ 
    \label{Eq (Section 6): Energy oscillation bounds from Hessian}
    \fiint_{Q_{\rho}}\lvert \nabla u_{\varepsilon}-(\nabla u_{\varepsilon})_{\rho}\rvert^{2}\,{\mathrm d}X\le \theta\mu^{2}.
  \end{align}
 \end{lemma}
 In addition to Lemma \ref{Lemma: Energy estimates from weak form}, we use
 \begin{equation}\label{Eq (Section 6): Lower bounds on superlevel set}
   \lvert \nabla u_{\varepsilon}\rvert\ge \frac{7}{8}\delta+(1-\nu)\mu\quad \textrm{a.e. in }S_{\rho,\,\mu,\,\nu},
 \end{equation}
 which is easy to deduce by $\varepsilon<\delta/8$.
 In fact, we have
 \(\delta+(1-\nu)\mu\le V_{\varepsilon}\le \varepsilon+\lvert \nabla u_{\varepsilon} \rvert\le \frac{\delta}{8}+\lvert \nabla u_{\varepsilon}\rvert\) a.e.~in \(S_{\rho,\,\mu,\,\nu}.\)
 The inequality (\ref{Eq (Section 6): Lower bounds on superlevel set}) also allows us to apply (\ref{Eq (Section 6): Monotonicity of G-p-epsilon}), which yields Lemma \ref{Lemma: Oscillation Integral lemma}.
 \begin{lemma}\label{Lemma: Oscillation Integral lemma}  
 Under the assumptions of Proposition \ref{Proposition: Parabolic Campanato}, there exists a constant $C\in(1,\,\infty)$, depending at most on $n$, $p$, $q$, $\lambda$, $\Lambda$, $K$, $F$, $M$ and $\delta$, such that
  \begin{equation}\label{Eq (Section 6): Energy oscillation 0}
    \Phi(\sigma\rho)\le \frac{C\mu^{2}}{\sigma^{n+2}}\left[(1-\sigma)+\frac{\sqrt{\nu}}{(1-\sigma)^{3}}+\frac{F\rho^{\beta}}{(1-\sigma)^{2}\sqrt{\nu}}\right]
  \end{equation}
  for all $\sigma\in(0,\,1)$.
 \end{lemma}  
 \begin{proof}
   We fix $\sigma\in (0,\,1)$, and set ${\mathbb R}^{n}$-valued functions
   \[H_{1}(t)\coloneqq \fint_{B_{\sigma\rho}}\nabla u_{\varepsilon}(x,\,t)\,{\mathrm d}x,\quad H_{p}(t)\coloneqq\fint_{B_{\sigma\rho}}G_{p,\,\varepsilon}(\nabla u_{\varepsilon}(x,\,t))\,{\mathrm d}x,\quad \text{and} \quad \Xi(t)\coloneqq G_{p,\,\varepsilon}^{-1}(H_{p}(t)),\]
   defined for $t\in I_{\rho}$. By the definition of $G_{p,\,\varepsilon}$, (\ref{Eq (Section 2): Bounds of V-epsilon})--(\ref{Eq (Section 2): Delta vs Mu}) and (\ref{Eq (Section 6): Bound on G-p-epsilon-inverse}), it is easy to check $H_{1}(t)\le 2\mu$ and $\Xi(t)\le C_{p}\mu$.
   We apply (\ref{Eq (Section 6): L-2 oscillation minimizer}) to $\nabla u_{\varepsilon}(\,\cdot\,,\,t)\in L^{2}(B_{\sigma\rho};\,{\mathbb R}^{n})$ to deduce
   \begin{align*}
     \Phi(\sigma\rho)&\le 2\fiint_{Q_{\sigma\rho}}\lvert \nabla u_{\varepsilon}(x,\,t)-H_{1}(t) \rvert^{2}\,{\mathrm d}x{\mathrm d}t+2\fiint_{Q_{\sigma\rho}}\lvert H_{1}(t)-(\nabla u_{\varepsilon})_{\sigma\rho} \rvert^{2}\,{\mathrm d}x{\mathrm d}t\\ 
    &\le  2\fiint_{Q_{\sigma\rho}}\lvert \nabla u_{\varepsilon}-\Xi(t) \rvert^{2}\,{\mathrm d}x{\mathrm d}t+C\mu\fiint_{Q_{\sigma\rho}}\lvert H_{1}(t)-(\nabla u_{\varepsilon})_{\sigma\rho}\rvert\,{\mathrm d}x{\mathrm d}t\\ 
    & \le C\mu\left[\fiint_{Q_{\sigma\rho}}\lvert \nabla u_{\varepsilon}-\Xi(t) \rvert\,{\mathrm d}x{\mathrm d}t+ \esssup_{t_{1},\,t_{2}\in I_{\sigma\rho}} \,\lvert H_{1}(t_{1})-H_{1}(t_{2}) \rvert \right].
   \end{align*}
   Therefore, (\ref{Eq (Section 6): Energy oscillation 0}) is shown by proving following (\ref{Eq (Section 6): Energy oscillation 1})--(\ref{Eq (Section 6): Energy oscillation 2});
   \begin{align}
    &\fiint_{Q_{\sigma\rho}}\left\lvert \nabla u_{\varepsilon}(x,\,t)-\Xi(t) \right\rvert\,{\mathrm d}x{\mathrm d}t\le \frac{C\mu}{\sigma^{n+2}}\left[\frac{\sqrt{\nu}}{1-\sigma}+\frac{F\rho^{\beta}}{\sqrt{\nu}} \right].\label{Eq (Section 6): Energy oscillation 1} \\ 
    & \esssup_{t_{1},\,t_{2}\in I_{\sigma\rho}}  \lvert H(t_{1})-H(t_{2}) \rvert\le \frac{C\mu}{\sigma^{n}}\left[(1-\sigma)+\frac{\sqrt{\nu}}{(1-\sigma)^{3}}+\frac{F\rho^{\beta}}{(1-\sigma)^{2}\sqrt{\nu}}\right]. \label{Eq (Section 6): Energy oscillation 2}
   \end{align}
   
   To prove (\ref{Eq (Section 6): Energy oscillation 1}) from (\ref{Eq (Section 3): Hessian L-2 energy on Q})--(\ref{Eq (Section 3): Hessian L-2 energy on S}), we decompose $Q_{\sigma\rho}=(Q_{\sigma\rho}\setminus S_{\sigma\rho})\cup S_{\sigma\rho}$, and note
   \begin{equation}\label{Eq (Section 6): S-sigma-rho is small domain}
    \lvert Q_{\sigma\rho}\setminus S_{\sigma\rho}\rvert \le \lvert Q_{\rho}\setminus S_{\rho}\rvert\le \nu\lvert Q_{\rho}\rvert=\sigma^{-(n+2)}\nu\lvert Q_{\sigma\rho}\rvert,
   \end{equation}
    which is easily checked by replacing the sublevel set $Q_{\sigma\rho}\setminus S_{\sigma\rho}$ by a larger one $Q_{\rho}\setminus S_{\rho}$, and using (\ref{Eq (Section 2): Measure Assumption Campanato}). 
   We also recall (\ref{Eq (Section 6): Lower bounds on superlevel set}), so that (\ref{Eq (Section 6): Monotonicity of G-p-epsilon}) can be applied to $\lvert \nabla u_{\varepsilon}(x,\,t)-\Xi(t)\rvert$ for $(x,\,t)\in S_{\sigma\rho}$.
   Combining with (\ref{Eq (Section 6): S-sigma-rho is small domain}) and the Poincar\'{e}--Wirtinger inequality, we get
   \begin{align*}
     \fiint_{Q_{\sigma\rho}}\left\lvert \nabla u_{\varepsilon}(x,\,t)-\Xi_{1}(t) \right\rvert\,{\mathrm d}X&\le \frac{C\nu\mu}{\sigma^{n+2}}+\frac{C}{\mu^{p-1}}\fiint_{Q_{\sigma\rho}}\left\lvert G_{p,\,\varepsilon}(\nabla u_{\varepsilon}(x,\,t))-H_{p}(t) \right\rvert\,{\mathrm d}X\\
     &\le \frac{C\nu\mu}{\sigma^{n+2}}+\frac{C}{\mu^{p-1}}(\sigma\rho)\fiint_{Q_{\sigma\rho}}\left\lvert \nabla\left[G_{p,\,\varepsilon}(\nabla u_{\varepsilon}) \right] \right\rvert\,{\mathrm d}X.
   \end{align*}
   To estimate the last integral, we use the decomposition $Q_{\sigma\rho}=(Q_{\sigma\rho}\setminus S_{\sigma\rho})\cup S_{\sigma\rho}$ and (\ref{Eq (Section 6): S-sigma-rho is small domain}) again.
   By H\"{o}lder's inequality and (\ref{Eq (Section 3): Hessian L-2 energy on Q})--(\ref{Eq (Section 3): Hessian L-2 energy on S}), we have
   \[
      \fiint_{Q_{\sigma\rho}}\left\lvert \nabla u_{\varepsilon}(x,\,t)-\Xi_{1}(t) \right\rvert\,{\mathrm d}x{\mathrm d}t
        \le \frac{C\nu\mu}{\sigma^{n+2}}+\frac{C\mu}{\sigma^{n+1}}\left[\frac{\left(1+\sigma^{n/2+1}\right)\sqrt{\nu}}{1-\sigma}+\left(\sqrt{\nu}+\frac{\sigma^{n/2+1}}{\sqrt{\nu}}\right)F\rho^{\beta} \right].
   \]
   By $\nu\in(0,\,1/4)$ and $\sigma\in(0,\,1)$, it is easy to deduce (\ref{Eq (Section 6): Energy oscillation 1}).

   To prove (\ref{Eq (Section 6): Energy oscillation 2}), we set ${\tilde \sigma}\coloneqq (1+\sigma)/2\in(1/2,\,1)$, and choose $\eta\in C_{\mathrm c}^{2}(B_{{\tilde\sigma}\rho};\,\lbrack 0,\,1\rbrack)$ satisfying  $\eta|_{B_{\sigma\rho}}\equiv 1$, and $\lVert \nabla^{2}\eta \rVert_{L^{\infty}(B_{\tilde{\sigma}\rho})}+\lVert \nabla\eta \rVert_{L^{\infty}(B_{\tilde{\sigma}\rho})}^{2}\le c(1-\sigma)^{-2}\rho^{-2}$ for some constant $c\in(0,\,\infty)$. 
   Without loss of generality, we may let $t_{1}$, $t_{2}\in I_{\sigma\rho}$ satisfy $t_{1}<t_{2}$. 
   We choose sufficiently small $\tilde\varepsilon>0$, and define a piecewise linear function $\chi\colon \overline{I_{{\tilde\sigma}\rho}}\rightarrow \lbrack 0,\,1\rbrack$ such that $\chi\equiv 0$ on $\lbrack t_{0}-({\tilde \sigma}\rho)^{2},\,t_{1}\rbrack\cup \lbrack t_{2},\,t_{0}\rbrack$, $\chi\equiv 1$ on $\lbrack t_{1}+{\tilde\varepsilon},\,t_{2}-{\tilde\varepsilon}\rbrack$, and $\chi$ is linearly interpolated in $(t_{1},\,t_{1}+{\tilde \varepsilon})\cup (t_{2}-{\tilde\varepsilon},\,t_{2})$.
   We will later let ${\tilde\varepsilon}\to 0$. 
   We test $\varphi(x,\,t)\coloneqq \eta(x)\chi(t)$, which is compactly supported in $Q_{{\tilde\sigma}\rho}$, into the weak formulation
   \[-\iint_{Q_{{\tilde\sigma}\rho}}\partial_{x_{j}}u_{\varepsilon}\partial_{t}\varphi\,{\mathrm d}X+\iint_{Q_{{\tilde \sigma}\rho}}\left\langle\nabla E_{\varepsilon}(\nabla u_{\varepsilon})-\nabla E_{\varepsilon}(\Xi(t))\mathrel{}\middle|\mathrel{}\nabla\partial_{x_{j}}\varphi \right\rangle\,{\mathrm d}X=\iint_{Q_{{\tilde \sigma}\rho}}f_{\varepsilon}\partial_{x_{j}}\varphi\,{\mathrm d}X.\]
   Letting ${\tilde \varepsilon}\to 0$, 
   we have
   \[
    \left\lvert \int_{B_{{\tilde \sigma}\rho}}\eta\left(\partial_{x_{j}} u_{\varepsilon}(x,t_{1})-\partial_{x_{j}} u_{\varepsilon}(x,t_{2})\right){\mathrm d}x \right\rvert
    \le 
    \iint_{Q_{\tilde{\sigma}\rho}}\left[\left\lvert \nabla E_{\varepsilon}(\nabla u_{\varepsilon})-\nabla E_{\varepsilon}(\Xi(t))\right\rvert \left\lvert \nabla^{2}\eta \right\rvert
    +
    \lvert f_{\varepsilon}\rvert\lvert\nabla\eta\rvert\right]{\mathrm d}X,
    \]
    which holds for a.e.~$t_{1},\,t_{2}\in I_{\sigma\rho}$.
   To estimate the first integral on the right-hand side, we use $Q_{{\tilde \sigma}\rho}=(Q_{{\tilde\sigma}\rho}\setminus S_{\tilde{\sigma}\rho})\cup S_{{\tilde\sigma}\rho}$ and (\ref{Eq (Section 6): S-sigma-rho is small domain}) with $\sigma$ replaced by $\tilde\sigma$.
   By applying (\ref{Eq (Section 2): Growth outside a facet}) with $(z_{1},\,z_{2})=(\nabla u_{\varepsilon}(x,\,t),\,\Xi(t))$ for $(x,\,t)\in S_{\tilde{\sigma}\rho}$,  and noting $\lVert \nabla^{2}\eta\rVert_{L^{\infty}(B_{\tilde{\sigma}\rho})} \lvert B_{\sigma\rho}\rvert^{-1}\le C(n)(1-\sigma)^{-2}\sigma^{-n} \tilde{\sigma}^{n+2} \lvert Q_{\tilde{\sigma}\rho}\rvert^{-1}$,  and using (\ref{Eq (Section 6): Energy oscillation 1}) with $\sigma$ replaced by $\tilde\sigma$, we get
   \begin{align*}
    &\frac{1}{\lvert B_{\sigma\rho}\rvert}\left\lvert \int_{B_{{\tilde \sigma}\rho}}\eta\left( \partial_{x_{j}} u_{\varepsilon} (x,\,t_{1}) - \partial_{x_{j}} u_{\varepsilon} (x,\,t_{2})\right)\,{\mathrm d}x \right\rvert\\ 
    &\le  C(K,\,\Lambda,\,M)\lVert \nabla^{2}\eta\rVert_{L^{\infty}(B_{\tilde{\sigma}\rho})}\cdot \frac{\nu\lvert Q_{\tilde{\sigma}\rho}\rvert}{\lvert B_{\sigma\rho}\rvert}+\frac{\lVert \nabla^{2}\eta\rVert_{L^{\infty}(B_{\tilde{\sigma}\rho})}}{\lvert B_{\sigma\rho}\rvert}\cdot C_{2} \iint_{S_{\tilde{\sigma}\rho}}\lvert \nabla u_{\varepsilon}(x,\,t)-\Xi(t)\rvert\,{\mathrm d}x{\mathrm d}t \\&  \quad +\frac{\lVert \nabla\eta\rVert_{L^{\infty}(B_{\tilde{\sigma}\rho})}}{\lvert B_{\sigma\rho}\rvert} F\lvert Q_{\tilde{\sigma}\rho} \rvert^{1-1/q}  \\
    &\le \frac{C(n,\,K,\,\Lambda,\,\delta,\,M)\mu}{(1-\sigma)^{2}\rho^{2}}\cdot\frac{\nu\lvert Q_{\rho}\rvert}{\lvert B_{\sigma\rho}\rvert}+\frac{ C(n,\,\delta,\,M)  {\tilde \sigma}^{n+2}}{(1-\sigma)^{2}\sigma^{n}}\fiint_{Q_{\tilde{\sigma}\rho}}\lvert \nabla u_{\varepsilon}(x,\,t)-\Xi(t)\rvert\,{\mathrm d}x{\mathrm d}t+\frac{C(n)F\lvert Q_{\rho}\rvert^{1-1/q}}{(1-\sigma)\sigma^{n}\rho^{n+1}}\\ 
    &\le \frac{C(n,\,p,\,q,\,\delta,\,M)\mu}{\sigma^{n}}\left[\frac{\nu}{(1-\sigma)^{2}}+\frac{\sqrt{\nu}}{(1-\sigma)^{3}}+\frac{F\rho^{\beta}}{(1-\sigma)^{2}\sqrt{\nu}}+\frac{F\rho^{\beta}}{1-\sigma}\right].
   \end{align*}
    for a.e.~$t_{1},\,t_{2}\in I_{\sigma\rho}$. 
    Summing over $j\in \{\,1,\,\dots\,,\,n\,\}$, and recalling our choice of $\tilde\sigma$ and $\eta$,  we have 
   \begin{align*}
     & \lvert H_{1}(t_{1})-H_{1}(t_{2})\rvert\\ 
     & \le  \frac{1}{\lvert B_{\sigma\rho}\rvert}\left[ \left\lvert\int_{B_{{\tilde\sigma}\rho}\setminus B_{\sigma\rho} }\eta\left(\nabla u_{\varepsilon}(x,\,t_{1})-\nabla u_{\varepsilon}(x,\,t_{2}) \right)\,{\mathrm d}x  \right\rvert
     +\left\lvert\int_{B_{{\tilde\sigma}\rho}}\eta\left(\nabla u_{\varepsilon}(x,\,t_{1})-\nabla u_{\varepsilon}(x,\,t_{2}) \right)\,{\mathrm d}x  \right\rvert\right] \\ 
     & \le \frac{C\mu}{\sigma^{n}}\left[(1-\sigma)+\frac{\sqrt{\nu}}{(1-\sigma)^{3}}+\frac{F\rho^{\beta}}{(1-\sigma)^{2}\sqrt{\nu}}\right]
   \end{align*}
    for a.e.~$t_{1},\,t_{2}\in I_{\sigma\rho}$. 
 \end{proof}
 \begin{proof}[Proof of Lemma \ref{Lemma: Oscillation from Hessian}]
   For $\sigma\in (0,\,1)$, we apply (\ref{Eq (Section 6): L-2 oscillation minimizer}) to $\nabla u_{\varepsilon}\in L^{2}(Q_{\rho};\,{\mathbb R}^{n})$, and make use of (\ref{Eq (Section 6): Energy oscillation 0}). 
   Then, we have
   \[
     \Phi(\rho)
     \le \sigma^{n+2}\Phi(\sigma\rho)+\frac{\lvert Q_{\rho}\setminus Q_{\sigma\rho}\rvert}{\lvert Q_{\rho}\rvert}\cdot 4\mu^{2}
     \le C_{\dagger}\mu^{2}\left[(1-\sigma)+\frac{\sqrt{\nu}}{(1-\sigma)^{3}}+\frac{F\rho^{\beta}}{(1-\sigma)^{2}\sqrt{\nu}} \right]
   \]
   with $C_{\dagger}\in(0,\,\infty)$ independent of $\sigma$, $\nu$, and $\rho$. 
   Hence,  we conclude (\ref{Eq (Section 6): Energy oscillation bounds from Hessian}) by choosing $\sigma\in(0,\,1)$, $\nu\in(0,\,1/4)$, and $\rho_{\ast}\in(0,\,1)$ satisfying the following three inequalities; 
   \[C_{\dagger}(1-\sigma)\le \frac{\theta}{3},\quad \nu\le \min\left\{\,\frac{3-8\sqrt{\theta}}{23},\,\left(\frac{\theta(1-\sigma)^{3}}{3C_{\dagger}}\right)^{2}\,\right\},\quad \frac{C_{\dagger}F\rho_{\ast}^{\beta}}{(1-\sigma)^{2}\sqrt{\nu}}\le \frac{\theta}{3}.\]
   We use the triangle inequality, the Cauchy--Schwarz inequality, and (\ref{Eq (Section 6): Energy oscillation bounds from Hessian})--(\ref{Eq (Section 6): Lower bounds on superlevel set}) to get
   \begin{align*}
    \lvert (\nabla u_{\varepsilon})_{Q_{\rho}} \rvert
    &\ge \fiint_{Q_{\rho}}\lvert \nabla u_{\varepsilon}\rvert\,{\mathrm d}X -\left\lvert \fiint_{Q_{\rho}} \left[\lvert \nabla u_{\varepsilon}\rvert- \lvert (\nabla u_{\varepsilon})_{Q_{\rho}} \rvert \right]\,{\mathrm d}X \right\rvert\\ 
    &\ge \frac{\lvert S_{\rho}\rvert}{\lvert Q_{\rho}\rvert} \essinf_{S_{\rho}}\,\lvert \nabla u_{\varepsilon}\rvert -\sqrt{\Phi(\rho)}
    \ge (1-\nu)\left(\frac{7}{8}\delta+(1-\nu)\mu \right) -\sqrt{\theta}\mu.
   \end{align*}
   Recalling (\ref{Eq (Section 2): Delta vs Mu}) and our choice of $\nu$, we conclude (\ref{Eq (Section 6): Mean bounds by below}).
  \end{proof}

 \subsection{Higher integrability}
 Before comparing $u_{\varepsilon}$ with heat flows, we prove the higher integrability of $\lvert \nabla u_{\varepsilon}-(\nabla u_{\varepsilon})_{\rho}\rvert$ (Lemma \ref{Lemma: Higher integrability parabolic}), which can be verified when $(\nabla u_{\varepsilon})_{\rho}$ does not degenerate.
 \begin{lemma}\label{Lemma: Higher integrability parabolic}
 In addition to the assumptions of Theorem \ref{Thm: Key Hoelder Estimate}, let the positive numbers $\mu$, $M$, and a cylinder $Q_{2\rho}=Q_{2\rho}(x_{0},\,t_{0})\Subset Q_{R}(x_{\ast},\,t_{\ast})$ satisfy (\ref{Eq (Section 2): Bounds of V-epsilon})--(\ref{Eq (Section 2): Delta vs Mu}).
  If a vector $\xi\in{\mathbb R}^{n}$ satisfy 
   \begin{equation}\label{Eq (Section 6): Assumption of xi}
     \delta+\frac{\mu}{4} \le\lvert \xi\rvert\le \delta+\mu,
   \end{equation}
  then there exist $\vartheta\in(0,\,1)$ and $C\in(1,\,\infty)$, depending at most on $n$, $p$, $q$, $\lambda$, $\Lambda$, $K$, $\delta$ and $M$, such that $2(1+\vartheta)\le q$, and
   \[
     \left(\fiint_{Q_{\rho/2}}\lvert \nabla u_{\varepsilon}-\xi\rvert^{2(1+\vartheta)}\,{\mathrm d}X\right)^{\frac{1}{1+\vartheta}}\le C\left[\fiint_{Q_{\rho}}\lvert \nabla u_{\varepsilon}-\xi\rvert^{2}\,{\mathrm d}X+F^{2}\rho^{2\beta}\right].
   \] 
 \end{lemma}

 Without a proof, we infer Lemma \ref{Lemma: Parabolic Gehring lemma}, often called Gehring's lemma.
 \begin{lemma}\label{Lemma: Parabolic Gehring lemma}
  Fix a parabolic cylinder $Q_{2R}\subset{\mathbb R}^{n+1}$, and the exponents $1<s<{\tilde s}<\infty$.
  Assume that non-negative functions $g\in L^{s}(Q_{2R})$ and $h\in L^{\tilde s}(Q_{2R})$ satisfy
  \[\fiint_{Q_{r}(z)}g^{s} \,{\mathrm d}X \le B_{0}\left[\left(\fiint_{Q_{2r}(z)}g\,{\mathrm d}X\right)^{s}+\fiint_{Q_{2r}(z)}h^{s}\,{\mathrm d}X \right]+\theta\fiint_{Q_{2r}(z)}g^{s}\,{\mathrm d}X\]
  for any parabolic cylinder $Q_{2r}(z)\subset Q_{2R}$. If $\theta\le \theta_{0}$ holds for some sufficiently small $\theta_{0}=\theta_{0}(n)\in(0,\,1)$, then there exists a constant $\vartheta_{0}\in(0,\,{\tilde s}/s-1)$, depending at most on $n$, $s$, ${\tilde s}$, and $B_{0}$, such that there holds
  \[\fiint_{Q_{R}}g^{s(1+\vartheta)}\,{\mathrm d}X \le B\left[\left(\fiint_{Q_{2R}}g^{s}\,{\mathrm d}X\right)^{1+\vartheta}+\fiint_{Q_{2R}}h^{s(1+\vartheta)}\,{\mathrm d}X\right]\]
  for every $\vartheta\in(0,\,\vartheta_{0})$. 
  Here the constant $B$ depends at most on $n$, $s$, ${\tilde s}$, $B_{0}$, and $\vartheta$.
 \end{lemma}
 \begin{remark}\upshape
 For elliptic cases, Gehring's lemma is shown by the Calder\'{o}n--Zygmund cube decomposition and Vitali's covering lemma (see e.g., \cite[\S 6.5]{MR3099262}, \cite[\S 6.4]{MR1962933}, \cite{MR2173373}).
 It is often mentioned without a proof that the elliptic arguments therein work even for parabolic cubes of the form $C_{r_{0}}(x_{0},\,t_{0})\coloneqq K_{r}(x_{0})\times (t_{0}-r^{2},\,t_{0}\rbrack$, where $K_{r}(x_{0})\coloneqq (x_{0,\,1}-r,\,x_{0,\,1}+r)\times \dots (x_{0,\,n}-r,\,x_{0,\,n}+r)\subset{\mathbb R}^{n}$ for $r_{0}\in(0,\,\infty)$, $x_{0}=(x_{0,\,1},\,\dots\,,\,x_{0,\,n})\in{\mathbb R}^{n}$.
 Indeed, the interior of $C_{2r_{0}}(x_{0},\,t_{0})$ coincides with the open set $\{(x,\,t)\in {\mathbb R}^{n+1}\mid d^{\prime}((x,\,t),\,(x_{0},\,t_{0}-2r_{0}^{2}))<2r_{0}\}$, where $d^{\prime}$ is another metric in ${\mathbb R}^{n+1}$, defined as
   \(d^{\prime}((x,\,t),\,(y,\,s))\coloneqq \max\left\{\,\lvert x_{j}-y_{j}\rvert,\,\sqrt{2\lvert t-s\rvert} \,\right\}\) for \((x,\,t),\,(y,\,s)\in{\mathbb R}^{n+1}\).
   Thanks to this fact, we can prove Lemma \ref{Lemma: Parabolic Gehring lemma} by carrying out the Calder\'{o}n--Zygmund cube decomposition, and applying Vitali's covering lemma for this metric $d^{\prime}$ (see also \cite[\S 5]{krylov2021diffusion}).
\end{remark}

We provide the proof of Lemma \ref{Lemma: Higher integrability parabolic} (see also \cite[\S 2]{MR0652852} as the classical result).
 \begin{proof}[Proof of Lemma \ref{Lemma: Higher integrability parabolic}]
   It suffices to find $C=C(n,\,p,\,q,\,\lambda,\,\Lambda,\,K,\,M,\,\delta)\in(0,\,\infty)$, such that 
   \[
     \fiint_{Q_{r}}\lvert \nabla u_{\varepsilon}-\xi\rvert^{2} \,{\mathrm d}X\le C\left[\left(\fiint_{Q_{2r}}\lvert \nabla u_{\varepsilon}-\xi \rvert^{\frac{2n}{n+2}}\,{\mathrm d}X \right)^{\frac{n+2}{n}}+\fiint_{Q_{2r}}\lvert \rho f_{\varepsilon}\rvert^{2}\,{\mathrm d}X\right]+\theta_{0}\fiint_{Q_{2r}}\lvert \nabla u_{\varepsilon}-\xi\rvert^{2}\,{\mathrm d}X
   \]
   holds for any $Q_{2r}=Q_{2r}(y_{0},\,s_{0})\subset Q_{\rho}=Q_{\rho}(x_{0},\,t_{0})$, where $\theta_{0}=\theta_{0}(n)$ is given in Lemma \ref{Lemma: Parabolic Gehring lemma}.
   Indeed, the assertion enables us to apply Lemma \ref{Lemma: Parabolic Gehring lemma} to the functions $g\coloneqq \lvert \nabla u_{\varepsilon}-\xi\rvert^{\frac{2n}{n+2}}\in L^{\frac{n+2}{n}}(Q_{\rho})$ and $h\coloneqq \lvert\rho f_{\varepsilon}\rvert^{\frac{2n}{n+2}}\in L^{\frac{q(n+2)}{2n}}(Q_{\rho})$.
   By H\"{o}lder's inequality and (\ref{Eq (Section 2): Bound of f-epsilon}), we conclude 
   the desired estimate. 

   Fix $Q_{2r}=Q_{2r}(y_{0},\,s_{0})\subset Q_{\rho}=Q_{\rho}(x_{0},\,t_{0})$ arbitrarily, and define
   \(w_{\varepsilon}(x,\,t)\coloneqq u_{\varepsilon}(x,\,t)-\langle \xi\mid x-y_{0} \rangle\)
   for $(x,\,t)\in Q_{2r}$.
   Set $r_{l}\coloneqq 2^{l/2}r\in\lbrack r,\,2r\rbrack$ for $l\in\{\,0,\,1,\,2\,\}$.
   For each $l\in\{\,1,\,2\,\}$, we choose $\eta_{l}\in C_{\mathrm c}^{1}(B_{r_{l}}(y_{0});\,\lbrack 0,\,1\rbrack)$ and $\phi_{l}\in C^{1}(\overline{I_{r_{l}}};\,\lbrack 0,\,1\rbrack)$ satisfying $\phi_{l}(s_{0}-r_{l}^{2})=0$, $\eta_{l}|_{B_{r_{l-1}(y_{0})}}\equiv 1$, and $\phi_{l}|_{I_{r_{l-1}(y_{0})}}\equiv 1$.
   For this $\eta_{l}$, we define
   \[\tilde{w}_{\varepsilon,\,l}(t)\coloneqq \left[\int_{B_{2r}}\eta_{l}^{2}\,{\mathrm d}x\right]^{-1}\int_{B_{2r}}w_{\varepsilon}(x,\,t)\eta_{l}^{2}\,{\mathrm d}x\] for \(t\in I_{2r}(s_{0}).\)
   Let $\phi_{\mathrm h}\colon \overline{I_{2r}(s_{0})}\rightarrow \lbrack 0,\,1\rbrack$ be an arbitrary Lipschitz function that is non-increasing and $\phi_{\mathrm h}(s_{0})=0$.
   We test $\varphi\coloneqq \left(w_{\varepsilon}-{\tilde w}_{\varepsilon,\,l}\right)\eta_{l}^{2}\phi$ with $\phi\coloneqq \phi_{l}\phi_{\mathrm h}$ into the weak formulation
   \[-\iint_{I_{2r}}\langle \partial_{t}\varphi,\, w_{\varepsilon}\rangle\,{\mathrm d}t+\iint_{Q_{2r}}\left\langle \nabla E_{\varepsilon}(\nabla u_{\varepsilon})-\nabla E_{\varepsilon}(\xi)\mathrel{}\middle|\mathrel{}\nabla\varphi\right\rangle\,{\mathrm d}X=\iint_{Q_{2r}}f_{\varepsilon}\varphi\,{\mathrm d}X,\]
   which holds for any $\varphi\in X_{0}^{2}(s_{0}-4r^{2},\,s_{0};\,B_{2r}(y_{0}))$.
   Noting that the integral of $\varphi(\,\cdot\,,\,t)$ over $B_{2r}$ vanishes for all $t\in I_{2r}(s_{0})$, we can compute
   \begin{align*}
     I_{1}+I_{2}
     &\coloneqq -\iint_{Q_{2r}}(w_{\varepsilon}-{\tilde w}_{\varepsilon,\,l})\eta_{l}^{2}\partial_{t}\phi\,{\mathrm d}X+\iint_{Q_{2r}}\left\langle \nabla E_{\varepsilon}(\nabla u_{\varepsilon})-\nabla E_{\varepsilon}(\xi)\mathrel{}\middle|\mathrel{}\nabla w_{\varepsilon} \right\rangle \eta_{l}^{2}\phi\,{\mathrm d}X\\
     & =-2\iint_{Q_{2r}}\left\langle \nabla E_{\varepsilon}(\nabla u_{\varepsilon})-\nabla E_{\varepsilon}(\xi)\mathrel{}\middle|\mathrel{}\nabla\eta_{l} \right\rangle \eta_{l} \left(w_{\varepsilon}-{\tilde w}_{\varepsilon,\,l}\right) \phi\,{\mathrm d}X+\iint_{Q_{2r}}f_{\varepsilon}\left(w_{\varepsilon}-{\tilde w}_{\varepsilon,\,l}\right)\varphi\,{\mathrm d}X\\ 
     &\eqqcolon -2I_{3}+I_{4}.
   \end{align*}
  By (\ref{Eq (Section 6): Assumption of xi}), we can apply (\ref{Eq (Section 2): Monotonicity outside a facet})--(\ref{Eq (Section 2): Growth outside a facet}) to estimate $I_{2}$ and $I_{3}$ by below and by above respectively.
  Young's inequality and standard absorbing arguments yield
  \begin{align*}
    &-\frac{1}{2}\iint_{Q_{2r}}\left(w_{\varepsilon}-{\tilde w}_{\varepsilon,\,l}\right)^{2}\eta_{l}^{2}\phi_{l}\partial_{t}\phi_{\mathrm h}\,{\mathrm d}X+\frac{C_{1}}{2}\iint_{Q_{2r}}\lvert \nabla w_{\varepsilon}\rvert^{2}\eta_{l}^{2}\phi_{l}\phi_{\mathrm h}\,{\mathrm d}X\\ 
    &\le 
    C\left[\iint_{Q_{2r}}\left\lvert w_{\varepsilon}-{\tilde w}_{\varepsilon,\,l}\right\rvert^{2}\left(\lvert\nabla\eta_{l}\rvert^{2}+\lvert \partial_{t}\phi_{l}\rvert+r_{j}^{-2} \right)\,{\mathrm d}X
    +\frac{r_{j}^{2}}{2}\iint_{Q_{2r}}\lvert f_{\varepsilon}\rvert^{2}\eta_{l}^{2}\phi_{l}\,{\mathrm d}X\right ].
  \end{align*}
  Choosing $\phi_{\mathrm h}$ suitably, and noting $\lvert Q_{r_{j}}\rvert=r_{j}^{2}\lvert B_{r_{j}}\rvert$, we easily deduce
  \begin{align*}
  &  \sup_{\tau\in \ring{I_{r_{l}}}}\fint_{B_{r_{l}}\times\{\tau\}}\frac{\left\lvert w_{\varepsilon}-{\tilde w}_{\varepsilon,\,l}\right\rvert^{2}}{r_{l}^{2}}\eta_{l}^{2}\phi_{l}\,{\mathrm d}x+\fiint_{Q_{r_{l}}}\lvert \nabla w_{\varepsilon}\rvert^{2}\eta_{l}^{2}\phi_{l}\,{\mathrm d}X\\ 
  &\le C\left[\fiint_{Q_{r_{l}}}\frac{\left\lvert w_{\varepsilon}-{\tilde w}_{\varepsilon,\,l}\right\rvert^{2}}{r_{l}^{2}}\,{\mathrm d}X+\fiint_{Q_{r_{l}}}\lvert rf_{\varepsilon}\rvert^{2}\,{\mathrm d}X\right].\nonumber  
  \end{align*}
  Repeatedly using this estimate 
  with $l=1$ and $l=2$, we deduce
  \begin{align*}
    &\fiint_{Q_{r_{0}}}\lvert \nabla w_{\varepsilon}\rvert^{2}\,{\mathrm d}X
    \le\fiint_{Q_{r_{1}}}\lvert \nabla w_{\varepsilon}\rvert^{2}\eta_{1}^{2}\phi_{1} \,{\mathrm d}X\le C\fiint_{Q_{r_{1}}}\frac{\left\lvert w_{\varepsilon}-{\tilde w}_{\varepsilon,\,1} \right\rvert^{2}}{r_{1}^{2}}\,{\mathrm d}X+C\fiint_{Q_{r_{1}}}\lvert rf_{\varepsilon}\rvert^{2}\,{\mathrm d}X, \text{ and } \\ 
    &\sup_{\tau\in \ring{I_{r_{1}}}}\fint_{B_{r_{1}}\times\{\tau\}}\frac{\left\lvert w_{\varepsilon}-{\tilde w}_{\varepsilon,\,1}\right\rvert^{2}}{r_{1}^{2}}\,{\mathrm d}x\\ 
    &\le 2\sup_{\tau\in \ring{I_{r_{1}}}}\left(\fint_{B_{r_{1}}\times \{\tau \}} \frac{\left\lvert w_{\varepsilon}(x,\,\tau)-{\tilde w}_{\varepsilon,\,2}(\tau)\right\rvert^{2}}{r_{1}^{2}}\,{\mathrm d}x+\frac{\left\lvert {\tilde w}_{\varepsilon,\,2}(\tau)-{\tilde w}_{\varepsilon,\,1}(\tau)\right\rvert^{2}}{r_{1}^{2}}\right)\\
    &\le C_{n}\sup_{\tau\in \ring{I_{r_{1}}}}\fint_{B_{r_{1}}\times\{\tau\}}\frac{\left\lvert w_{\varepsilon}-{\tilde w}_{\varepsilon,\,2}\right\rvert^{2}}{r_{1}^{2}}\,{\mathrm d}x\le C_{n}\sup_{\tau\in \ring{I_{r_{2}}}}\fint_{B_{r_{2}}\times\{\tau\}}\frac{\left\lvert w_{\varepsilon}-{\tilde w}_{\varepsilon,\,2}\right\rvert^{2}}{r_{2}^{2}}\eta_{2}^{2}\phi_{2}\,{\mathrm d}x\\ 
    &\le C\fiint_{Q_{r_{2}}}\frac{\left\lvert w_{\varepsilon}-{\tilde w}_{\varepsilon,\,2} \right\rvert^{2}}{r_{2}^{2}}\,{\mathrm d}X+C\fiint_{Q_{r_{2}}}\lvert rf_{\varepsilon}\rvert^{2}\,{\mathrm d}X\le C\fiint_{Q_{r_{2}}}\lvert \nabla w_{\varepsilon}\rvert^{2}\,{\mathrm d}X+C\fiint_{Q_{r_{2}}}\lvert rf_{\varepsilon}\rvert^{2}\,{\mathrm d}X,
  \end{align*}
  where we have used the Poincar\`{e} inequality for $(w_{\varepsilon}-{\tilde w}_{\varepsilon,\,2})(\,\cdot\,,\,t)$. 
  Let the exponent $\kappa$ be defined as $\kappa\coloneqq 2n/(n-2)$ for $n\ge 3$, and otherwise satisfy $\kappa\in(2,\,\infty)$.
  Interpolating the function $\lvert w_{\varepsilon}(\,\cdot\,,\,t)-{\tilde w}_{1,\,\varepsilon}(t)\rvert/r_{1}$ among the Lebesgue spaces $L^{\kappa}(B_{r_{1}})\subset L^{2}(B_{r_{1}})\subset L^{\frac{2n}{n+2}}(B_{r_{1}})$, applying the Poincar\'{e} inequality, and using H\"{o}lder's inequality for the time variable, we get
  \[
    \fint_{I_{r_{1}}}\left(\fint_{B_{r_{1}}}\frac{\left\lvert w_{\varepsilon}(x,\,t)-{\tilde w}_{\varepsilon,\,1}(t)\right\rvert^{2}}{r_{1}^{2}}\,{\mathrm d}x \right)^{1/2}\,{\mathrm d}t\le C\left(\displaystyle\fiint_{Q_{r_{1}}}\lvert \nabla w_{\varepsilon}\rvert^{2}\,{\mathrm d}X \right)^{\pi_{1}}\left(\displaystyle\fiint_{Q_{r_{1}}}\lvert \nabla w_{\varepsilon}\rvert^{\frac{2n}{n+2}}\,{\mathrm d}X \right)^{\pi_{2}}
  \]
  with $(\pi_{1},\,\pi_{2})\coloneqq (\frac{1}{4},\,\frac{n+2}{4n})$ for $n\ge 3$, and $(\pi_{1},\,\pi_{2})=(\frac{\kappa^{\prime}}{4},\,\frac{\kappa-2}{2(\kappa-1)})$ for $n=2$.
  The interpolation among $L^{\infty}(I_{r_{1}}) \subset L^{2}(I_{r_{1}}) \subset L^{1}(I_{r_{1}})$ and Young's inequality yield
  \begin{align*}
    &\fiint_{Q_{r_{1}}}\frac{\left\lvert w_{\varepsilon}-{\tilde w}_{\varepsilon,\,1} \right\rvert^{2}}{r_{1}^{2}}\,{\mathrm d}X 
    \le \left(\sup_{\tau\in \ring{I_{r_{1}}}}\fint_{B_{r_{1}}\times\{\tau\}}\frac{\left\lvert w_{\varepsilon}-{\tilde w}_{\varepsilon,\,1} \right\rvert^{2}}{r_{1}^{2}}\,{\mathrm d}x\right)^{1/2}\fint_{I_{r_{1}}}\left(\fint_{B_{r_{1}}}\frac{\left\lvert w_{\varepsilon}-{\tilde w}_{\varepsilon,\,1} \right\rvert^{2}}{r_{1}^{2}} \,{\mathrm d}x\right)^{1/2}{\mathrm d}t\\ 
     &\le \sigma\fiint_{Q_{r_{2}}}\lvert \nabla w_{\varepsilon}\rvert^{2}\,{\mathrm d}X+C(\sigma)\left[\left(\fiint_{Q_{r_{2}}}\lvert \nabla w_{\varepsilon}\rvert^{\frac{2n}{n+2}}\,{\mathrm d}X\right)^{\frac{n+2}{n}}+\fiint_{Q_{r_{2}}}\lvert rf_{\varepsilon}\rvert^{2}\,{\mathrm d}X \right]
  \end{align*}
  for every $\sigma\in(0,\,1)$.
  Noting $\nabla w_{\varepsilon}=\nabla u_{\varepsilon}-\xi$ and $2r\le \rho$, we complete the proof by choosing sufficiently small $\sigma\in(0,\,1)$. 
\end{proof}
 \subsection{Comparison estimates}
 We consider a heat flow $v_{\varepsilon}$, and compare $u_{\varepsilon}$ with $v_{\varepsilon}$.
 More precisely, we choose the function $v_{\varepsilon}\in u_{\varepsilon}+X_{0}^{2}(t_{0}-\rho^{2}/4,\,t_{0};\,B_{\rho/2}(x_{0}))$ satisfying 
 \begin{equation}\label{Eq (Section 6): Dirichlet heat flow, the weak form}
  \int_{t_{0}-\rho^{2}/4}^{t_{0}}\langle \partial_{t}v_{\varepsilon},\,\varphi\rangle \,{\mathrm d}t+\iint_{Q_{\rho/2}}\left\langle \nabla^{2}E_{\varepsilon}((\nabla u_{\varepsilon})_{\rho})\nabla v_{\varepsilon}\mathrel{}\middle|\mathrel{}\nabla\varphi\right\rangle\,{\mathrm d}X=0
 \end{equation}
 for all $\varphi\in X_{0}^{2}(t_{0}-\rho^{2}/4,\,t_{0};\,B_{\rho/2}(x_{0}))$, and $v_{\varepsilon}|_{t=t_{0}-\rho^{2}/4}=u_{\varepsilon}|_{t=t_{0}-\rho^{2}/4}$ in $L^{2}(B_{\rho/2}(x_{0}))$. In other words, $v_{\varepsilon}$ is the weak solution of the Dirichlet boundary problem
 \begin{equation}\label{Eq (Section 6): Dirichlet heat flow}
  \left\{\begin{array}{rclcl}
    \partial_{t}v_{\varepsilon}-\divx\left(\nabla^{2}E_{\varepsilon}((\nabla u_{\varepsilon})_{\rho})\nabla v_{\varepsilon} \right) &=& 0&\textrm{in}&Q_{\rho/2}(x_{0},\,t_{0}),\\ 
    v_{\varepsilon} & = & u_{\varepsilon} & \textrm{on}&\partial_{\textrm{p}}Q_{\rho/2}(x_{0},\,t_{0}).
  \end{array} \right.
 \end{equation}
 \begin{lemma}
  In addition to the assumptions of Theorem \ref{Thm: Key Hoelder Estimate}, let the positive numbers $\mu$, $M$, and a cylinder $Q_{2\rho}=Q_{2\rho}(x_{0},\,t_{0})\Subset Q_{R}(x_{\ast},\,t_{\ast})$ satisfy (\ref{Eq (Section 2): Bounds of V-epsilon})--(\ref{Eq (Section 2): Delta vs Mu}), and   
   \begin{equation}\label{Eq (Section 6): Assume that average is non-degenerate}
     \delta +\frac{\mu}{4}\le \lvert (\nabla u_{\varepsilon})_{\rho}\rvert.
   \end{equation}
   Let $v_{\varepsilon}$ be the weak solution of (\ref{Eq (Section 6): Dirichlet heat flow}). 
   Then, the following (\ref{Eq (Section 6): Comparison Estimate})--(\ref{Eq (Section 6): Heat flow decay}) hold.
   \begin{align}
    &\fiint_{Q_{\rho/2}}\lvert \nabla u_{\varepsilon}-(\nabla u_{\varepsilon})_{\rho/2} \rvert^{2}\,{\mathrm d}X\le C\left[\omega\left(\sqrt{\Phi(\rho)/\mu^{2}} \right)^{\frac{\vartheta}{1+\vartheta}} \Phi(\rho)+F^{2}\rho^{2\beta} \right].\label{Eq (Section 6): Comparison Estimate}\\ 
    &\fiint_{Q_{\sigma\rho}}\lvert \nabla v_{\varepsilon}-(\nabla v_{\varepsilon})_{\sigma\rho}\rvert^{2}\,{\mathrm d}X\le C\sigma^{2}\fiint_{Q_{\rho/2}}\lvert \nabla v_{\varepsilon}-(\nabla v_{\varepsilon})_{\rho/2} \rvert^{2}\,{\mathrm d}X\quad \text{for all }\sigma\in(0,\,1/2).\label{Eq (Section 6): Heat flow decay}
   \end{align}
   %
   %
   Here the exponent $\vartheta$ is given by Lemma \ref{Lemma: Higher integrability parabolic}, and the constant $C\in(0,\,\infty)$ depends at most on $n$, $p$, $q$, $\lambda$, $\Lambda$, $K$, $\delta$, $M$, and $\omega$.
 \end{lemma}
 \begin{proof}
   For notational simplicity, we abbreviate $\xi\coloneqq (\nabla u_{\varepsilon})_{\rho}$, which satisfies (\ref{Eq (Section 6): Assumption of xi}) by (\ref{Eq (Section 2): Bounds of V-epsilon}) and (\ref{Eq (Section 6): Assume that average is non-degenerate}).
   Hence, it is easy to check
   \(\delta\le\left(\varepsilon^{2}+\lvert \xi\rvert^{2}\right)^{1/2}\le \frac{\delta}{8}+M.\)
   Combining with (\ref{Eq (Section 2): bounds of E-1-epsilon})--(\ref{Eq (Section 2): bounds of E-p-epsilon}), we can easily find a sufficiently small constant $m\in(0,\,1)$, depending at most on $p$, $\lambda$, $\Lambda$, $K$, $\delta$ and $M$, such that \(m\mathrm{id}_{n} \leqslant \nabla^{2}E_{\varepsilon}(\xi) \leqslant m^{-1}\mathrm{id}_{n}\) holds.
   In other words, the problem (\ref{Eq (Section 6): Dirichlet heat flow}) is uniformly parabolic in the classical sense, and therefore it is solvable.
   Moreover, since the coefficient matrix is constant, we can find a constant $C=C(n,\,m)\in(1,\,\infty)$ satisfying (\ref{Eq (Section 6): Heat flow decay}).
   For the unique existence and regularity properties of $v_{\varepsilon}$, we refer the reader to \cite[Chapter IV]{MR0241822}, \cite[Chapter IV]{MR1465184} (see also \cite{MR0213737}).

   We are left to prove (\ref{Eq (Section 6): Comparison Estimate}).
   By (\ref{Eq (Section 2): Approximation equation}) and (\ref{Eq (Section 6): Dirichlet heat flow, the weak form}), we can deduce
   \begin{align}\label{Eq (Section 6): Comparison Weak form}
     &\int_{t_{0}-\rho^{2}/4}^{t_{0}} \langle\partial_{t}(u_{\varepsilon}-v_{\varepsilon}),\,\varphi\rangle\,{\mathrm d}t+\iint_{Q_{\rho/2}}\left\langle \nabla^{2} E_{\varepsilon}(\xi) (\nabla u_{\varepsilon}-\nabla v_{\varepsilon})\mathrel{}\middle|\mathrel{}\nabla\varphi\right\rangle\,{\mathrm d}X\\ 
     &=\iint_{Q_{\rho/2}}\left\langle \nabla^{2}E_{\varepsilon}(\xi)(\nabla u_{\varepsilon}-\xi)- \left(\nabla E_{\varepsilon}(\nabla u_{\varepsilon})-\nabla E_{\varepsilon}(\xi)\right)\mathrel{}\middle|\mathrel{}\nabla\varphi \right\rangle\,{\mathrm d}X
     +\iint_{Q_{\rho/2}}f_{\varepsilon}\varphi\,{\mathrm d}X\nonumber
    \end{align}
   for all $\varphi\in X_{0}^{2}(t_{0}-\rho^{2}/4,\,t_{0};\,B_{\rho/2}(x_{0}))$.
   For sufficiently small ${\tilde\varepsilon}>0$, which tends to $0$ later,
   we choose $\phi_{\mathrm h}\colon \overline{I_{\rho/2}}\to \lbrack 0,\,1\rbrack$ satisfying $\phi_{\mathrm h}\equiv 1$ on $\lbrack t_{0}-\rho^{2}/4,\,t_{0}-{\tilde \varepsilon}\rbrack$, and linearly interpolated on the other domain with $\phi_{\mathrm h}(t_{0})=0$.
   We test $\varphi\coloneqq (u_{\varepsilon}-v_{\varepsilon})\phi_{\mathrm h}$ into (\ref{Eq (Section 6): Comparison Weak form}).
   Note that we may discard the first integral in (\ref{Eq (Section 6): Comparison Weak form}), since it is positive.
   Letting ${\tilde \varepsilon}\to 0$, making use of (\ref{Eq (Section 2): Continuity estimate on Hessian}), and applying the embedding $W_{0}^{1,\,2}(B_{\rho/2})\hookrightarrow L^{2}(B_{\rho/2})$ to $(u_{\varepsilon}-v_{\varepsilon})(\,\cdot\,,\,t)\in W_{0}^{1,\,2}(B_{\rho/2})$, we get
   \begin{align*}
     m\iint_{Q}\lvert \nabla u_{\varepsilon}-\nabla v_{\varepsilon}\rvert^{2}\,{\mathrm d}X& \le C\left(\iint_{Q}\omega\left(\frac{\lvert \nabla u_{\varepsilon}-\xi \rvert}{\mu} \right)^{2}\lvert \nabla u_{\varepsilon}-\xi\rvert^{2}\,{\mathrm d}X \right)^{\frac{1}{2}}\left( \iint_{Q} \lvert \nabla u_{\varepsilon}- \nabla v_{\varepsilon} \rvert^{2}\,{\mathrm d}X \right)^{\frac{1}{2}} \\ 
     & \quad +C(n)\rho\left(\iint_{Q}\lvert f_{\varepsilon}\rvert^{2}\,{\mathrm d}X \right)^{\frac{1}{2}}\left(\iint_{Q}\lvert \nabla u_{\varepsilon}-\nabla v_{\varepsilon} \rvert^{2}\,{\mathrm d}X \right)^{\frac{1}{2}}.
    \end{align*}
    with $Q=Q_{\rho/2}(x_{0},\,t_{0})$. 
   By making use of H\"{o}lder's inequality, and applying Jensen's inequality to the concave function $\omega$, we have
   \[
     \fiint_{Q}\lvert \nabla u_{\varepsilon}-\nabla v_{\varepsilon}\rvert^{2}\,{\mathrm d}X
     \le C\left[\left(\fiint_{Q}\lvert \nabla u_{\varepsilon}-\xi \rvert^{2(1+\vartheta)}\,{\mathrm d}X \right)^{\frac{1}{1+\vartheta}}\omega\left(\sqrt{\Phi(\rho)/\mu^{2}} \right)^{\frac{\vartheta}{1+\vartheta}}+F^{2}\rho^{2\beta}\right],
   \]
   where we note $\omega(\lvert \nabla u_{\varepsilon}-\xi\rvert/\mu)\le \omega(4)$ by (\ref{Eq (Section 2): Bounds of V-epsilon})--(\ref{Eq (Section 2): Delta vs Mu}).
   Using Lemma \ref{Lemma: Higher integrability parabolic}, we conclude (\ref{Eq (Section 6): Comparison Estimate}).
 \end{proof}
 \begin{lemma}\label{Lemma: Key lemma on Parabolic oscillation}
   In addition to the assumptions of Theorem \ref{Thm: Key Hoelder Estimate}, let the positive numbers $\mu$, $M$, and a cylinder $Q_{2\rho}=Q_{2\rho}(x_{0},\,t_{0})\Subset Q_{R}(x_{\ast},\,t_{\ast})$ satisfy (\ref{Eq (Section 2): Bounds of V-epsilon})--(\ref{Eq (Section 2): Delta vs Mu}).
   For each $\sigma\in(0,\,1/2)$, there exists a sufficiently small $\theta_{0}\in(0,\,1/16)$, depending on $\sigma$, $\omega$, and $\vartheta$, such that if (\ref{Eq (Section 6): Energy oscillation bounds from Hessian}) and (\ref{Eq (Section 6): Assume that average is non-degenerate}) hold with $\theta\le \theta_{0}$, then
   there exists a constant $C_{\ast}\in(1,\,\infty)$, depending at most on $n$, $p$, $q$, $\lambda$, $\Lambda$, $K$, $\delta$, $M$ and $\omega$, such that
   \begin{equation}\label{Eq (Section 6): Key oscillation decay}
     \Phi(\sigma\rho)\le C_{\ast}\left[\sigma^{2}\Phi(\rho)+\frac{F^{2}}{\sigma^{n+2}}\rho^{2\beta} \right].
   \end{equation}
 \end{lemma}
 \begin{proof}
   Let $v_{\varepsilon}$ be a weak solution of (\ref{Eq (Section 6): Dirichlet heat flow}). 
   By applying (\ref{Eq (Section 6): L-2 oscillation minimizer}) to $\nabla u_{\varepsilon}\in L^{2}(Q_{\sigma\rho};\,{\mathbb R}^{n})$ and $\nabla v_{\varepsilon}\in L^{2}(Q_{\rho/2};\,{\mathbb R}^{n})$, and using (\ref{Eq (Section 6): Energy oscillation bounds from Hessian}) and (\ref{Eq (Section 6): Comparison Estimate})--(\ref{Eq (Section 6): Heat flow decay}), we have 
   \begin{align*}
     \Phi(\sigma\rho)
     &\le \fiint_{Q_{\sigma\rho}}\lvert \nabla u_{\varepsilon}-(\nabla v_{\varepsilon})_{\sigma\rho} \rvert^{2}\,{\mathrm d}X\\ 
     &\le \frac{2}{(2\sigma)^{n+2}}\fiint_{Q_{\rho/2}}\lvert \nabla u_{\varepsilon}-\nabla v_{\varepsilon} \rvert^{2}\,{\mathrm d}X+2\fiint_{Q_{\sigma\rho}}\lvert \nabla v_{\varepsilon}-(\nabla v_{\varepsilon})_{\sigma\rho}\rvert^{2}\,{\mathrm d}X\\ 
     &\le C\left[\left(\sigma^{2}+\frac{1}{\sigma^{n+2}}\right)\fiint_{Q_{\rho/2}}\lvert \nabla u_{\varepsilon}-\nabla v_{\varepsilon} \rvert^{2}\,{\mathrm d}X+\sigma^{2}\Phi(\rho)\right]\\
     &\le \frac{C_{\ast}}{2}\left[\Phi(\rho)\frac{\omega(\theta^{1/2})^{\frac{\vartheta}{1+\vartheta}}}{\sigma^{n+2}}+\frac{F^{2}}{\sigma^{n+2}}\rho^{2\beta}+\sigma^{2}\Phi(\rho)\right].
   \end{align*}
   Choosing $\theta_{0}\in(0,\,1/16)$ such that $\omega(\theta_{0}^{1/2})^{\frac{\vartheta}{1+\vartheta}}\le \sigma^{n+4}$, we conclude (\ref{Eq (Section 6): Key oscillation decay}). 
 \end{proof}

 \subsection{Proof of the Campanato-type growth estimates}
 Before the proof of Proposition \ref{Proposition: Parabolic Campanato}, we infer an elementary lemma without a proof.
 \begin{lemma}\label{Lemma: An elementary lemma for Campanato decay}
   Provided that $g\in L^{2}(Q_{\rho};\,{\mathbb R}^{n})$ admits a constant $A\in(0,\,\infty)$ satisfying
   \(\fiint_{Q_{r}}\lvert g-(g)_{Q_{r}}\rvert^{2}\,{\mathrm d}X\le (\frac{r}{\rho})^{2\beta}\) for all \(r\in(0,\,\rho\rbrack\),
   the limit \(G_{0}\coloneqq \lim\limits_{r\to 0}(g)_{Q_{r}}\in{\mathbb R}^{n}\)
   exists, and admits a constant $c_{\dagger\dagger}=c_{\dagger\dagger}(n,\,\beta)\in(1,\,\infty)$ such that
   \(\fiint_{Q_{r}}  \lvert g-G_{0}\rvert^{2}  \,{\mathrm d}X\le c_{\dagger\dagger}A(\frac{r}{\rho})^{2\beta}\) for all \(r\in(0,\,\rho\rbrack\).
  \end{lemma}
 \begin{proof}[Proof of Proposition \ref{Proposition: Parabolic Campanato}]
   Let $c_{\dagger}$, $c_{\dagger\dagger}$ be the positive constants found in (\ref{Eq (Section 2): Lipschitz continuity of G-2delta-epsilon}) and Lemma \ref{Lemma: An elementary lemma for Campanato decay}.
   The constants $\sigma\in(0,\,1/2)$ and $\theta\in(0,\,1/16)$ are determined to satisfy
   \[C_{\ast}\sigma^{2(1-\beta)}\le \frac{1}{2},\quad \text{and}\quad \theta\le \min\left\{\,\frac{\sigma^{n+2}}{64},\, \theta_{0}(\sigma,\,\vartheta,\,\omega),\,\frac{\sigma^{n+2+2\beta}}{c_{\dagger}c_{\dagger\dagger}}\,\right\},\]
   where $\theta_{0}\in(0,\,1/16)$ and $C_{\ast}\in(0,\,\infty)$ are given by Lemma \ref{Lemma: Key lemma on Parabolic oscillation}.
   Corresponding to this $\theta$, we choose $\nu=\nu(\theta)\in(0,\,1/4)$ and $\rho_{\ast}=\rho_{\ast}(\theta,\,\nu)\in(0,\,1)$ as in Lemma \ref{Lemma: Oscillation from Hessian}.
   We choose ${\hat\rho}\in(0,\,\rho_{\ast}\rbrack$ satisfying
   \(
   C_{\ast}F^{2}{\hat \rho}^{2\beta}\le \theta\sigma^{n+2+2\beta}/2\).
   Assume $\rho\in(0,\,{\hat\rho}\rbrack$, and set $\rho_{k}\coloneqq \sigma^{k}\rho$ for each $k\in{\mathbb Z}_{\ge 0}$.
   By iteration, 
   we would like to prove
   \begin{equation}\label{Eq (Section 6): Induction Claim 1}
    \Phi(\rho_{k})\le \sigma^{2k\beta}\theta\mu^{2},\quad \text{and}\quad \left\lvert (\nabla u_{\varepsilon})_{\rho_{k}} \right\rvert\ge \delta+\left[\frac{1}{2}-\frac{1}{8}\sum_{j=1}^{k-1}2^{-j}\right]\mu\ge \delta+\frac{\mu}{4},
   \end{equation}
   for every $k\in{\mathbb Z}_{\ge 0}$.
   By (\ref{Eq (Section 2): Measure Assumption Campanato}) and $\rho_{0}\le \rho_{\ast}$, we can use Lemma \ref{Lemma: Oscillation from Hessian} to conclude (\ref{Eq (Section 6): Induction Claim 1}) for $k=0$.
   Let (\ref{Eq (Section 6): Induction Claim 1}) be valid for $k\in{\mathbb Z}_{\ge 0}$. 
   This induction hypothesis allows us to apply Lemma \ref{Lemma: Key lemma on Parabolic oscillation} over a smaller cylinder $Q_{\rho_{k}}$.
   In particular, (\ref{Eq (Section 6): Key oscillation decay}) with $\rho$ replaced by $\rho_{k}$ implies 
   \(\Phi(\rho_{k+1})\le C_{\ast}\sigma^{2(1-\beta)}\cdot \sigma^{2\beta}\Phi(\rho_{k})+\frac{C_{\ast}F^{2}}{\sigma^{n+2}}\left(\sigma^{k}{\hat\rho} \right)^{2\beta}\mu^{2}\le \sigma^{2(k+1)\beta}\theta\mu^{2}.\)
   Also, 
   the Cauchy--Schwarz inequality and the triangle inequality yield 
  \(
    \lvert (\nabla u_{\varepsilon})_{\rho_{k+1}} \rvert
    \ge \lvert (\nabla u_{\varepsilon})_{\rho_{k}} \rvert-\lvert (\nabla u_{\varepsilon})_{\rho_{k+1}}-(\nabla u_{\varepsilon})_{\rho_{k}} \rvert
    \ge \lvert (\nabla u_{\varepsilon})_{\rho_{k}} \rvert-\sigma^{-\frac{n+2}{2}}\sqrt{\Phi(\rho_{k})}, 
    \) 
  by which and the induction hypothesis, (\ref{Eq (Section 6): Induction Claim 1}) holds true for $k+1\in{\mathbb Z}_{\ge 0}$.

   For every $r\in(0,\,\rho\rbrack$, there corresponds a unique $k\in{\mathbb Z}_{\ge 0}$ such that $\rho_{k+1}<r\le \rho_{k}$.
   Repeatedly applying (\ref{Eq (Section 6): L-2 oscillation minimizer}) to ${\mathcal G}_{2\delta,\,\varepsilon}(\nabla u_{\varepsilon})$, $\nabla u_{\varepsilon}\in L^{2}(Q_{r};\,{\mathbb R}^{n})$ and using (\ref{Eq (Section 2): Lipschitz continuity of G-2delta-epsilon}) and (\ref{Eq (Section 6): Induction Claim 1}), we have
   \begin{align*}
     &\fiint_{Q_{r}}\left\lvert {\mathcal G}_{2\delta,\,\varepsilon}(\nabla u_{\varepsilon})-({\mathcal G}_{2\delta,\,\varepsilon}(\nabla u_{\varepsilon}))_{r} \right\rvert^{2}\,{\mathrm d}X\le \fiint_{Q_{r}}\left\lvert {\mathcal G}_{2\delta,\,\varepsilon}(\nabla u_{\varepsilon})-{\mathcal G}_{2\delta,\,\varepsilon}((\nabla u_{\varepsilon})_{r}) \right\rvert^{2}\,{\mathrm d}X \\
     &\le c_{\dagger}^{2} \fiint_{Q_{r}}\lvert \nabla u_{\varepsilon}-(\nabla u_{\varepsilon})_{r} \rvert^{2}\,{\mathrm d}X\le c_{\dagger}^{2} \sigma^{-(n+2)}\fiint_{Q_{\rho_{k}}} \lvert \nabla u_{\varepsilon}-(\nabla u_{\varepsilon})_{\rho_{k}} \rvert^{2}\,{\mathrm d}X\\ 
     &\le c_{\dagger}\sigma^{2k\beta-(n+2)}\theta\mu^{2}\le \frac{c_{\dagger}\theta}{\sigma^{n+2+2\beta}}\left(\frac{\rho}{r}\right)^{2\beta}\mu^{2}.
   \end{align*}
   From Lemma \ref{Lemma: An elementary lemma for Campanato decay} and our choice of $\theta$, we complete the proof of Proposition \ref{Proposition: Parabolic Campanato}.
 \end{proof}
 
 \section*{Acknowledgments}
 The author gratefully acknowledges Prof. Yoshikazu Giga for his valuable comments on mathematical models for Bingham fluids and crystal surfaces. 
  The author also would like to thank the reviewers for their valuable suggestions.

\addcontentsline{toc}{section}{References}
\bibliographystyle{siamplain}

\end{document}